\documentclass[11pt,a4paper,reqno]{amsart}

\usepackage[latin1]{inputenc}
\usepackage{amssymb,amsmath,amstext,amsthm,amsfonts}
\usepackage{graphicx}
\usepackage{comment}
\usepackage{multicol}
\usepackage{subfigure}
\usepackage{hyperref}
\usepackage{tikz-cd}
\usepackage{calrsfs}
\usepackage[shortlabels]{enumitem}
\usepackage{bbm}
\usepackage{float}
\usepackage{bm}
\usepackage{geometry}
\DeclareMathAlphabet{\pazocal}{OMS}{zplm}{m}{n}

\numberwithin{equation}{section}
\usepackage{a4wide}
\usepackage{xcolor}

\newcommand{\R}{\mathbb{R}}
\newcommand{\Z}{\mathbb{Z}}

\newcommand{\normi}[1]{{\left\vert\kern-0.25ex\left\vert\kern-0.25ex\left\vert #1 
    \right\vert\kern-0.25ex\right\vert\kern-0.25ex\right\vert}}

\newcommand{\qand}{\quad\text{and}\quad}
\newcommand{\sgn}{\text{sgn}}
\newcommand{\Leb}{\text{Leb}}

\newcommand{\cor}{\text{Cor}}
\newcommand{\e}{\text{e}}

\newcommand\smallO{
	\mathchoice
	{{\scriptstyle\mathcal{O}}}% \displaystyle
	{{\scriptstyle\mathcal{O}}}% \textstyle
	{{\scriptscriptstyle\mathcal{O}}}% \scriptstyle
	{\scalebox{.7}{$\scriptscriptstyle\mathcal{O}$}}%\scriptscriptstyle
}

\def\diam{\operatorname{diam}}

\def\id{\operatorname{id}}

\newtheorem{maintheorem}{Theorem}

\newcommand{\cmt}{\begin{maintheorem}}
\newcommand{\fmt}{\end{maintheorem}}

\newtheorem{maincorollary}[maintheorem]{Corollary}

\newcommand{\cmc}{\begin{maincorollary}}
\newcommand{\fmc}{\end{maincorollary}}

\newtheorem{lemma}{Lemma}[section]

\newtheorem{proposition}[lemma]{Proposition}

\theoremstyle{remark}
\newtheorem{remark}[lemma]{Remark}

\newtheorem{definition}{Definition}

\makeatletter
\@namedef{subjclassname@2020}{\textup{2020} Mathematics Subject Classification}
\makeatother

\thanks{OE was partially supported by CMUP %(UID/MAT/00144/2013),
(UIDB/00144/2020)
and PTDC/MAT-PUR/4048/2021, which are funded by FCT (Portugal) with national (MEC) and European structural funds through the programs COMPTE and FEDER, under the partnership agreement PT2020.
}

\begin{document}
\title[Linear response for intermittent circle maps]{Linear response for intermittent circle maps}
\date{}

\author[O. Etubi]{Odaudu  Etubi}
\address{Odaudu  Etubi\\
Centro de Matem\'{a}tica da Universidade do Porto\\ Rua do Campo Alegre 687\\ 4169-007 Porto\\ Portugal.}
\email{etubiodaudu@gmail.com} 

\maketitle

%----------------------------------------------------------------------------------------------------------------------------------------------------------

%\begin{figure}[h]
%\centering
%\begin{minipage}{1 \linewidth}
%\centering
%\includegraphics[trim=0mm 0mm 0mm 0mm,clip,scale=0.8]{}
%\caption{}
%\end{minipage}
%\end{figure}

%----------------------------------------------------------------------------------------------------------------------------------------------------------

\begin{abstract}
Using the cone technique of Baladi and Todd \cite{BT16}, we show some form of weak differentiability of the SRB measure for the intermittent circle maps, demonstrating {\it{linear response}} in the process. Subsequently, as an application, we lift the regularity from the base dynamics of the solenoid map with intermittency, showing that this family is {\it{statistically stable}}. 
\end{abstract}

%----------------------------------------------------------------------------------------------------------------------------------------------------------

%\tableofcontents

 \section{Introduction}
 Intermittent behaviour in dynamical systems is characterized by orbits spending long times in a small area of the phase space and then a spontaneous burst of chaoticity on the remainder of the phase space.
 %a gold bar in the description of chaoticity
A prototype for this was given in the seminal paper by Pomeau and Manneville \cite{PM80}, a simplification which was modeled by the Liverani, Saussol and Vaienti (LSV) map in \cite{LSV98} characterized by the presence of a neutral fixed point at the origin.
Young
% in a seminal paper 
%studied in generality the statistical properties of such  in an abstract setting, 
introduced an intermittent circle map as an application 
to the statistical properties and rate of mixing 
of the general systems studied in \cite{YL99}, characterizing it
%the dynamical system 
in terms of some local properties around the neutral fixed point at $0$ and $1$, see section \ref{se.inte} for more details.
%One useful way of describing chaotic dynamical systems is via invariant probability measures, since then, we can analyze using probabilistic techniques.
These classes of maps are shown to admit an \emph{absolutely continuous invariant probability measure~(ACIP)} (or \emph{Sinai-Ruelle-Bowen (SRB) measures}) \cite{P80} and decay at a slow rate \cite{H04, LSV98, YL99}.
%The existence of an invariant density for such maps was shown in \cite{P80}.
% % ACIM (SRB) for the systems given in... 
% The slow decay rates of the above classes of maps were shown in \cite{H04, LSV98, YL99}.
% % decays at a polynomial rate.
% 

Let $\{f_\alpha\}_{\alpha \in \mathcal{A} }$ be a parameterised family of maps, where $\mathcal{A}$ is the parameter space. Suppose that $\mathcal{A}$ is a neighbourhood of $0$, with $f_0 = f$
% the system as unperturbed when $\alpha=0$
% $f_0$ as the unperturbed map and $f_\alpha$,
and perturbed whenever  $\alpha \ne 0$, with each $f_\alpha$ possessing a unique \emph{SRB} measure $\mu_\alpha$.
% SRB measure.
The topic of the regularity of $\alpha \mapsto \mu_\alpha$ is an active area of research in
%the area 
%of 
smooth ergodic theory. The study has been carried out under varying degrees, namely;
%The regularity of $\alpha \mapsto \mu_\alpha$ is an active topic dynamical systems and ergodic theory and has been studied under differing degrees.
%it
the continuity with respect to the parameters (called \emph{statistical stability}) was studied in \cite{ACF10, APV17,AV02}.
%A04, 
One way that the question of statistical stability has been posed in the literature is to ask whether the perturbed density $h_\alpha$ converges to the unperturbed density $h_0$ in the $L^1$-norm, this is usually 
%convergence of density $h_t$ of ACIP 
known as {\it{strong statistical stability}}. However, the notion of statistical stability we shall consider in this paper is given in terms of the weak$^*$ convergence of the measures $\mu_{\alpha_n}$ to $\mu_{\alpha_0}$ as $\alpha_n \to \alpha_0$, and sometimes it is referred to as {\it{weak statistical stability}}, in this paper, without making any distinction, we shall simply call it statistical stability.
 In some other cases, under some perturbation of the parameter, the invariant density
% might 
varies H\"older continuously \cite{T14}, log-Lipschitz continuously \cite{AE25, GL20}, or
%even smoothly 
is differentiable. In the last case, we say that the system admits \emph{linear response}. Another formulation is that for $\psi$ a fixed observable in a suitable class, linear response holds if $\alpha \mapsto \mathcal{R}_\psi= \int \psi \,d\mu_\alpha$, $\alpha \in \mathcal{A}$
%$\alpha \in \left[0, 1 \right)$ 
is differentiable. %with respect to $\alpha$ 
If the formula for this derivative exists in terms of the unperturbed terms of $f_0$, $\mu_0$, $\psi$ and the vector field $v_0:=\partial_\alpha f_\alpha|_{\alpha=0}$, then we call this the linear response formula.
%that is, it is the first order approximation of the SRB of the system with respect to the SRB of the unperturbed system and in that case a formula  for this derivative is called the \emph{linear response formula} of the system.

The notion of linear response has been around for quite some time in statistical mechanics. A pioneering result in this area has been the work of Ruelle for Axiom A attractors \cite{R97} and in the Anosov case \cite{KKPW89}, and for system with exponential decay of correlation or at least summable decay of correlations \cite{B07,BS08, BL07, D04, GL06, R97}, and in random dynamical systems \cite{BR20}. However, 
%report that
linear response have been reported to fail in some cases \cite{B07,B14, BS08}. For cases of systems which decays rather slowly, linear response have been shown for the Pomeau-Manneville type maps, particularly, system containing an intermittent fixed point albeit using different techniques \cite{BS16,BT16, K16}. The linear response given in \cite{ BT16, K16} was given in the weak sense, since
obtaining  linear response in the strong form, that is, in some topology \cite{ BG24, BR20, BS16}, 
%however 
is not always possible.  The technique used by Bahsoun-Saussol in \cite{BS16} is based on the construction of a well-known inducing
scheme with a first return time and an explicit formula for the density of the ACIP in terms of the density of the induced map's ACIP --none of which is known in the case of the intermittent circle map.
%
% in some situations. 
%However, linear response may be given in a weak sense . 
%By this, we mean that for a fixed observable, say $\psi$, in a suitable class, the function $\mathcal{R}_\psi= \int \psi d\mu_\alpha$, $\alpha \in \left[0, 1 \right)$. 
In this paper, we study the linear response 
%theory 
for the intermittent circle map introduced in \cite{YL99} using 
an interesting observation made in the case of the LSV map in \cite{BT16} that the singularity of the density and that of its derivative at zero were compensated by suitable factors (which we shall introduce in section \ref{se.mech}). Their analysis was based on the cone techniques used in the LSV paper \cite{LSV98}. 
%where they showed the linear response of the LSV map is that where they observed with $\mu_\alpha$ the unique SRB measure of $f_\alpha$ is instead shown to be differentiable at $0$
Here, we apply their approach to the intermittent circle map, which posses the neutral fixed point in the first and last branch, 
%is differentiable
% with respect to $\alpha$. 
%
% we study how $\int \varphi d\mu_\alpha$ changes with respect to $\alpha$. We want to find in some sense a weak form of derivative, 
We show
%ed 
that for any $\psi \in L^q, q \ge 1$, 
%$\alpha \mapsto \int \psi d\mu_\alpha$ 
$\mathcal{R}_\psi(\alpha)$ is
% continuously 
differentiable on $\alpha \in \left[0, 1-1/q\right)$. 
%A linear response
% A linear response formula was given in \cite{B14} for the smooth expanding circle map.
%\todo{Read this Baladi paper again} 
%.\todo{David Ruelle, give state of the art here}
%The question "How does the SRB measure respond when the parameter of the system is perturbed?" is at the heart of linear response theory. 

One interesting application of the class maps introduced in \cite{YL99} 
%by Young was 
%(Skew product)
is the solenoid map with intermittency introduced in \cite{AP08},
% intermittent transformations such as those considered in this paper gives rise to the solenoid with intermittency,
where the $2x \!\!\! \mod 1$ map in the base dynamics  of the classical solenoid map was replaced by the intermittent circle map.
The linear response result in the base map of this dynamics implies statistical stability and using the techniques from Alves and Soufi in \cite{AS14}, we show the the SRB measure for this skew product is in fact statistically stable. It remains an interesting open problem to determine whether linear response can be derived for the solenoid map.
%solenoid map with intermittency. 

 \section{The intermittent circle map} \label{se.inte}
 Let $f(x)$ and $g(x)$ be real valued functions, we write $f(x) \lesssim g(x)$ (resp. $f(x) \approx g(x)$) to mean that there exists a uniform constant $C\ge 1$, such that $f(x) \le C\,g(x)$  (resp. $f(x) \in [{C}^{-1}\,g(x), {C}\,g(x)]$) for all $x$. 
 % $\displaystyle {C}^{-1}~~g(x) ~\le~f(x)~\le C\, g(x)$
 % Observe that, $f(x) \approx g(x)$ means that both $f(x) \lesssim g(x)$ and $g(x) \lesssim f(x)$ holds. 
%  The same notation is applicable to sequences, $a_n, b_n$, for all $n$. 
  In what follows, for functions depending on both $x$ and the parameter $\alpha$, we write the partial derivative with respect to the parameter as $\partial_\alpha$ and the derivative with respect to $x$ as $(\cdot)^\prime$.
 %The proof we shall present in this chapter follows from the approach in \cite{BT16}. 
 
 Let $f_\alpha:S^1 \rightarrow S^1$, $S^1=\R/\Z$, be a degree $d\ge2$ circle map with $\alpha>0$ satisfying:
 \begin{enumerate}[label=\text{(s${\arabic*}$)}]
 	\item  \label{s1} $f_\alpha(0)=0$ and $f_\alpha^\prime(0)=1;$
 	\item \label{s2} $f_\alpha^{\prime}>1$ on $S^1\setminus \{0\};$
 	\item \label{s3} $f_\alpha$ is $C^2$ on $S^1\setminus \{0\}$ and $xf_\alpha^{\prime \prime}(x) \approx \vert x \vert ^\alpha$, for $x$ close to $0$.
 \end{enumerate}
 
 By \ref{s1} and \ref{s3},
 \begin{align}
 	\begin{split}
 		f_{\alpha}^{\prime} (x)-1  & \approx \vert x \vert ^{\alpha},\label{eq:firstderivative}
 	\end{split}
 \end{align}
 integrating the above,
 % \begin{equation}
 	%    f_{\alpha}(x)\approx \int_{0}^{x}f_{\alpha}^\prime (t)dt\approx \int_{0}^{x}(1+ c \vert t \vert ^{\alpha}))\approx x+\frac{c}{\alpha+1}\sgn(x) \vert x \vert ^{\alpha+1} \label{eq:zerothderivative}
 	% \end{equation}
 \begin{equation}
 	f_{\alpha}(x) \approx x+\sgn(x) \vert x \vert ^{\alpha+1}, \label{eq:zerothderivative}
 \end{equation}
 where $\sgn(\cdot)$ is the signum function,
%  defined as 
% \begin{equation*}
% 	\sgn(x) = \left\{
% 	\begin{array}{lll}
% 		-1, & \mbox{if }    x<0;\\
% 		0, & \mbox{if } x=0;\\
% 		1, & \mbox{if }    x>0;%
% 	\end{array}
% 	\right.
% \end{equation*}
 with $\displaystyle \sgn(x)=\frac{x}{\vert x \vert}= \frac{\vert x \vert}{x}$, for $x \ne 0$. A map satisfying the condition \ref{s1} is said to have an \emph{indifferent (or neutral) fixed point}.
 % if it satisfies (s1).
 % The intermittent circle map $f_\alpha$ as we have defined is a weak Gibbs Markov and wih some extra work can be showed to be indeed a Gibbs Markov map, this fact will be important when we study the Lipschitz continuity of the SRB measure of the solenoid map with intermittency in Chapter \ref{3}. 
 %In addition to this, it  has some really interesting properties, such as polynomial decay of correlation, and posses a unique SRB measure $\mu_\alpha$, for $\alpha \in (0,1)$ \cite[Theorem 3.62]{JFA20}.
 %To complete the introduction, 

 \subsection{Asymptotic behaviour near the fixed point} \label{ABNIP}
 Restricting ourselves to $f_\alpha\vert_{[0,\varepsilon_0]}$ (resp. $f_\alpha\vert_{[\varepsilon^\prime_0, 0]}$), where $(0, \varepsilon_0]$ (resp. $[\varepsilon^\prime_0,0)$) is an interval where the conditions \ref{s1} and \ref{s3} holds. When we say a degree $d$ map, we mean that there exists an $m\!\! \mod 0$ partition of $S^1$, $\{I_1, \dots, I_d\}$ into open intervals, 
 %taken in natural order, 
 such that $f_\alpha|_{I_i}: I_i \to S^1\setminus \{0\}$ is a diffeomorphism, for each $1 \le i \le d$. 
 Let $0\sim 1$ in such a way that $0$ is the infimum of $I_1$ and the supremum of $I_d$.
 The intervals $I_1$ and $I_d$ are further partitioned into countably many subintervals $J_n$ and $J_n^\prime$ respectively as follows; define the sequences $(z_n)_n$ and $(z_n^\prime)_n$ as 
 \begin{equation*}
 	f_\alpha(z_{n+1})=z_n \text{ and } f_\alpha(z^\prime_{n+1})= z^\prime_{n}, \quad n \ge 0, z_0 \in (0, \varepsilon_0], z^\prime_0 \in [\varepsilon^\prime_0,0).
 \end{equation*}
 
 Set for each $n \ge 1$,
 %\begin{equation*}
 	$J_n=(z_n, z_{n-1})
 	%\quad 
 	\text{ and }
 	% \quad 
 	J^\prime_n=(z^\prime_{n-1},z^\prime_n).$
% \end{equation*}
 The dynamics is related to the subintervals as follows
 %\begin{equation*}
 	$f_\alpha(J_{n+1})=J_{n}
 	% \quad 
 	\text{ and }
 	%\quad 
 	f_\alpha(J^\prime_{n+1})=J^\prime_{n}.$
 %\end{equation*}
 The local analysis in the interval containing the intermittent fixed point is given by $z _n \approx n^{-1/\alpha}$ \cite{JFA20, YL99}.
 By 
 %Theorem \ref{est.thaler}, 
 a theorem due to Thaler \cite[Theorem 1]{TM80}, there exists positive  constants $c_1, c_2$ such that the invariant density of $f_\alpha$
 %for the intermittent circle map 
 is bounded as follows
 \begin{equation}\label{eq:dens.bound}
 	c_1 \vert x\vert^{-\alpha} \le h_\alpha(x) \le c_2 \vert x \vert^{-\alpha}, \quad x \in S^1\setminus \{0\}, c_2 \ge c_1>0,
 \end{equation}
 The density concentrating around the origin and at $1$, the singularity point of $h_\alpha$ being at $x=0$. 
% Indeed, from equation \eqref{eq:zerothderivative}, for $x$ close to $0$
% \begin{align}\label{eq.formula}
% 	\frac{|f_\alpha(x)|}{|x|} \approx 1 + | x |^\alpha.
% \end{align}

 Observe that,
 \begin{align*}
 	|x|^{\alpha+1}= \left|f_\alpha(x)\right|^{\alpha+1}\left(\frac{|f_\alpha(x)|}{|x|}\right)^{-(\alpha+1)},
 \end{align*}
 which together with~\eqref{eq:zerothderivative} yields
 $$|x|^{\alpha+1}\approx \left|f_\alpha(x)\right|^{\alpha+1}\left(1 + | x |^\alpha\right)^{-(\alpha+1)}.
 $$
 Again, from \eqref{eq:zerothderivative}, we have that
 \begin{align*}
% 	f_\alpha(x) &\approx x + \left|f_\alpha(x)\right|^{\alpha+1} \sgn(x) \left(1 + | x |^\alpha\right)^{-(\alpha+1)}\\
 	x &\approx f_\alpha(x)- \sgn(f_\alpha(x))\left|f_\alpha(x)\right|^{\alpha+1}(\sgn(x))^2 \left(1 + | x |^\alpha\right)^{-(\alpha+1)},
 \end{align*}
 which leads us to conclude that
 \begin{equation}\label{eq.inverse}
 	g_{\alpha,i} (x)\approx x - \sgn(x)\vert x \vert^{1+\alpha} \cdot p(x), \quad i\in\{1,d\},
 \end{equation}
 where $p(x)=
 %  \sgn(g_{\alpha,i}(x))\,
 (1+|g_{\alpha,i}(x)|^\alpha)^{-(\alpha+1)}$ and $g_{\alpha,i}:S^1 \to I_i$ is the inverse branch of $f_{\alpha,i}$, $i \in \{1, d\}$.
% From equations \eqref{eq:zerothderivative}, \eqref{eq.FG} and \eqref{eq.inverse}, we get
% \begin{align*}
% 	G_j(x)&\approx \frac{x}{\sgn(x)|x|^{\alpha+1}}=\frac{1}{|x|^{\alpha}},\\
% 	F_j(x) &\approx \frac{x}{\sgn(x)|x|^{\alpha+1}\cdot p(x)}=\frac{1}{p(x)|x|^\alpha},
% \end{align*}
% which verifies equation \eqref{eq:dens.bound}.
% % From  the $F_j(x)$ and $G_j(x)$ respectively.
% 
% 
% %  The indifferent fixed point is a regular source by the condition \ref{s3}. 

 %We have verified all the assumptions of the theorem.
 %Now, since $f_\alpha^\prime(x)$ is bounded away from 1 on $S^1\setminus\{0\}$, we are therefore in the setting to apply Proposition \ref{Pthaler}, and we get the required conclusion.
 %and $\mathcal{H}$ the set of H\"older continuous functions on $S^1$.
 
 %$\,$\\

\section{The mechanism}\label{se.mech}

%Using the approach put forth by Baladi and Todd \cite{BT16}, we explain the mechanism that will be used in showing the linear response for the family of maps introduced in the previous section.

Following the approach of Baladi and Todd \cite{BT16}, we briefly explain the mechanism for showing the linear response for the family of maps introduced in the previous section.
% family.
%of and  

Firstly,
we shall assume that the perturbation occurs at the image. That is, there exists a vector field $X_\varrho$ such that
\begin{equation*}\label{eq.veee}
	v_\varrho(x):=\partial_\alpha f_\alpha(x)\bigg|_{\alpha=\varrho}=X_\varrho \circ f_\varrho(x), \quad \alpha, \varrho \in V,
\end{equation*}
and $V$ a small neighbourhood of $0$.
Since $f_\alpha$ is defined in the neighbourhood of $0$ and is invertible on the branches $i=\{1,d\}$, from the above equation we have
%\eqref{eq.veee} 
that

\begin{equation}\label{eq.eqeks}
	X_{\alpha,i}(x) = v_{\alpha} \circ g_{\alpha,i}(x), \quad i=\{1,d\}.
\end{equation}
%
%Now, let us denote the inverse branches of $f_\alpha$ by $g_{\alpha,1}:S^1 \to [0,\kappa]$ and $g_{\alpha,d}:S^1 \to [1-\kappa,1]$.
%$$g_{\alpha,1}(y)= f_{\alpha,1}^{-1}(y) \quad \text{ and } g_{\alpha,d}(y)=f_{\alpha,d}^{-1}(y).$$
%We define for $x$ in the neighbourhood of $0$,
%and
%on the first and last branches in $f_\alpha$

%\begin{equation*}\label{eq.veee}
%	v_{\alpha}(x)= \partial_{\alpha}f_{ \alpha}(x).
%\quad i=\{1,d\},
%\end{equation*}

%of course, $\partial_\alpha(\cdot)$ is taken on each $i$. Set,
% $$$$
% we see that, $X_{\varrho,i}(x)=v_{\varrho} \circ g_{\varrho,i}(x)= \partial_{\varrho}f_{\varrho}(g_{\varrho,i}(x))$. 
For $x$
% \in S^1$, 
in the neighbourhood of $0$, $\varrho \in [0,1)$. From \eqref{eq:zerothderivative} we have that
%the expression for $f_\varrho(x)$ for $x$ close to zero. Thus
\begin{align}
	v_{\varrho}(x)&=\partial_{\varrho}f_{\varrho}(x) \approx \sgn(x)  \vert x \vert^{\varrho+1} \ln(\vert x \vert) \label{eq.eqvee},\\
	X_{\varrho,i}(x) &\approx  \sgn(g_{\varrho,i}(x))  \vert g_{\varrho,i}(x) \vert^{\varrho+1} \ln(\vert g_{\varrho,i}(x)\vert). \label{eq:zzeroth}
\end{align}

%From equation \eqref{eq.inverse},  for $x \in S^1\setminus \{0\}$, $x$ near $0$, there exists $C\ge 1$ such that,
%\begin{equation}\label{eq.gbound}
%	\begin{split}
	%		&g_{\varrho,i}(x)\le C\, x,\\
	%		&g^\prime_{\varrho,i}(x) \text{ is bounded},\\ &g_{\varrho,i}^{\prime\prime}(x)\le C|x|^{\varrho-1},\\ &g_{\varrho,i}^{\prime\prime\prime}(x)\le C|x|^{\varrho-2}.
	%	\end{split}
%\end{equation} 

%$g_{\varrho,i}(x)\le Cx$. There exists 
%$K\ge 0$ and 
%a uniform constant $C\ge1$ such that for all $x$ in the neighbourhood of zero,

Next, we shall make some assumptions on the map $f_\varrho|_{I_i}: I_i \to S^1\setminus \{0\}$.
\begin{enumerate}[label=\text{(A${\arabic*}$)}]
	\item  \label{A1} There exists $C\ge1$ such that
	\begin{equation}\label{Eq.eqvee}
		v_{\varrho}(x)=\partial_{\varrho}f_{\varrho}(x) \le C\, \sgn(x)  \vert x \vert^{\varrho+1} \ln(\vert x \vert), \quad \text{ for } x\in I_i, i=1,d.
	\end{equation}
	\item  \label{A2} For a $d\ge 2$ branch map, we define the right end point of the first branch by $I_{1,+}$ and the left end point of the last branch by $I_{d,-}$ 
	and assume that
	\begin{align}
		%\begin{split}
			v_{\varrho}(I_{1,+})=0 \qand v_{\varrho}(I_{d,-})=0.
		%\end{split}
		\label{endp}
	\end{align} 
	%	with  $v_{\alpha}$  as defined in \eqref{eq.eqvee}.
	%	\item .
	% 	\item[iii.] $f_\alpha$ is independent of $\alpha$ for $x \notin I_i$ for $i = \{1,d\}$.
	\item  \label{A3} $\varrho \mapsto f_{\varrho,i} \in C^2$
	%	. For each of the branches, 
	and the following partial derivatives exist, and also satisfy the commutation relation
	\begin{equation}
		\partial_{\varrho}g_{\varrho,i}^{\prime}\approx (\partial_{\varrho}g_{\varrho,i})^{\prime} \qand
		\partial_{\varrho}f_{\varrho,i}^{\prime}\approx (\partial_{\varrho}f_{\varrho,i})^{\prime}, \quad i=1,d.
		\label{eq:commutation}
	\end{equation}
\end{enumerate}
We now state 
%here 
the main result of this paper. 
\begin{maintheorem}\label{LR}
	Suppose that $f_\alpha$ is the family of circle maps described above for  $\alpha \in (0, 1)$ and satisfy the assumptions \ref{A1}-\ref{A3}. Then for any $\psi \in L^q(m)$ with $q > (1-\alpha)^{-1}$,
	\begin{equation}
		\lim_{\varepsilon \to 0} \frac{\int_{S^1} \psi \,d \mu_{\alpha+ \varepsilon}-\int_{S^1} \psi \,d \mu_{\alpha}}{\varepsilon}= \int_{S^1} \psi (\id- \mathcal{L}_{\alpha})^{-1}\left[\sum_{i \in \{1,d\}}(X_{\alpha,i}\,\mathcal{N}_{\alpha,i}(h_{\alpha}))^{\prime}\right]dx. \label{eq:linearresponse}
	\end{equation}
	Taking limit $\varepsilon \to 0^+$, \eqref{eq:linearresponse} holds for $\alpha=0$.
	% Furthermore, $\alpha \mapsto \partial_\alpha h_\alpha \in L^p(m)$, $p \in [1, \infty)$, is continuous on $\alpha \in [0, 1/p)$.
\end{maintheorem}

% Where 
%%$h_{\alpha}$ is the density of $\mu_\alpha$, 
%$X_{\alpha,i}$ a vector field (In dimension 1, it is simply a function. See equation \eqref{eq.eqeks}),
%%the operators 

Where $\mathcal{L}_\alpha: L^1 \to L^1$ is the transfer operator associated with $f_\alpha$ and $\mathcal{N}_{\alpha,i}: L^1 \to L^1$  is the transfer operator associated with the $ith$ branch of $f_\alpha$, defined respectively as,
%with the indifferent fixed point
%defined as 
%follows 
%defined on piecewise branches for the whole map and at the branches with the intermittent fixed point (see equation \eqref{eq:PF}, equation \eqref{eq:PF_branch} respectively). 

%By the property \ref{s2} of this family of maps, The Perron-Frobenius operator \eqref{eq.appPF} is given by
%for this family of map is
\begin{equation}\label{eq:PF}
	\mathcal{L}_{\alpha}\varphi(x)= \sum_{f_{\alpha}(y)=x} \frac{\varphi(y)}{f_{\alpha}^{\prime}(y)},
\end{equation}
%\quad \varphi \in L^{1}(m)
%We don't need the absolute value in the denominator 
because of the property
%this follows from equation \eqref{eq.appPF} and 
\ref{s2}, and
%ext, for $\varphi \in L^{1}(m)$,
%Next, 
%we define the transfer operator associated to the first and last branches since $f_\alpha(x)$ has the indifferent fixed point on these branches.
% and is dependent on $\alpha$
%and  
%we define the transfer operator on these 

\begin{equation}
	\mathcal{N}_{\alpha,i}\varphi(x) = g_{\alpha,i}^{\prime}(x) \cdot \varphi(g_{\alpha,i}(x)), \quad  i\in\{1,d \}. \label{eq:PF_branch}
\end{equation}
%Now, the Perron-Frobenius operator may be written as
%\begin{equation}
%   \mathcal{L}_{\alpha}\varphi(x) = \sum_{i=1}^{d} \mathcal{N}_{\alpha,i} \varphi(x)
%\end{equation}
%\vspace{1cm}
From \eqref{eq.inverse},  
and $x \in S^1$, $i=1,d$, there exists $C\ge 1$ such that,
\begin{equation}\label{eq.gbound}
	\begin{split}
		&g_{\varrho,i}(x)\le C\, x,\\
		&g^\prime_{\varrho,i}(x) \text{ is bounded},\\ &g_{\varrho,i}^{\prime\prime}(x)\le C|x|^{\varrho-1},\\ &g_{\varrho,i}^{\prime\prime\prime}(x)\le C|x|^{\varrho-2}.
	\end{split}
\end{equation}

%\section{The mechanism}\label{se.mech}
%intermittent circle
%map $f_\alpha$. 

Suppose that $f_\alpha$ is the one parameter family of
%the intermittent circle 
map with $d\ge2$ branches, 
%introduced in the previous section, 
satisfying the assumptions \ref{A1}-\ref{A3}. Let $f_{\alpha,1}:[0,\kappa] \to S^1$, $f_{\alpha,d}:[1-\kappa,1] \to S^1, \kappa= 1/d$, be its first and last branches 
% of $f_\alpha$ 
respectively. The middle $(d-2)$ branches are piecewise expanding.
Using \eqref{eq.gbound}, there exists $\tilde{C}>1$ such that we bound \eqref{eq:zzeroth} as follows
\begin{equation}
	\vert X_{\varrho,i}(x) \vert \le \tilde{C}  \vert x \vert^{\varrho+1} (1+\vert \ln(\vert x\vert) \vert). \label{eq:zeroth}
\end{equation}

Subsequently, differentiating \eqref{eq:zzeroth}
\begin{equation}
	X_{\varrho,i}^\prime(x) \approx g^\prime_{\varrho,i}(x) \vert g_{\varrho,i}(x) \vert^\varrho  \left[1+(1+\varrho) \ln \vert g_{\varrho,i}(x) \vert\right], \label{eq:ffirst}
\end{equation}
we bound the above using the bounds in \eqref{eq.gbound} 
\begin{align}\label{eq:first}
	\vert X^\prime _{\varrho,i}(x) \vert &\le C \vert x \vert^{\varrho} \left[1 + (1+\varrho)(\ln C + \vert \ln(\vert x \vert)\vert)\right]\nonumber\\
	&\le \tilde{C} \vert x \vert^{\varrho} (1+\vert \ln (\vert x \vert) \vert).
\end{align}

Differentiating \eqref{eq:ffirst},
\begin{align}\label{eq.eqssecond}
	X^{\prime\prime}_{\varrho,i}(x)\approx \vert g_{\varrho,i}(x)\vert^{\varrho-1} \bigg\{(1+\varrho)\sgn(g_{\alpha,i}(x))&(g^{\prime}_{\varrho,i}(x))^2 + \bigg[ \varrho\, \sgn(g_{\varrho,i}(x))(g^{\prime}_{\varrho,i}(x))^2 \nonumber\\&\quad+ \vert g_{\varrho,i}(x)\vert g^{\prime \prime}_{\varrho,i}(x) \bigg]
	\cdot\left[ 1+(1+\varrho) \ln (\vert g_{\varrho,i}(x)\vert) \right]\bigg\},
\end{align}
%  \begin{align*}
	%      X^{\prime\prime}_{\varrho,i}(x) \approx \vert g_{\varrho,i}(x)\vert^{\varrho-1} \bigg[ \vert &g_{\varrho,i}(x)\vert g^{\prime \prime}_{\varrho,i}(x) \{(1+\varrho)  \ln \vert g_{\varrho,i}(x)\vert +1\}\\
	%      &+(g^{\prime}_{\varrho,i}(x))^2\{\varrho(1+\varrho)\sgn(g_{\varrho,i}(x)) \ln\vert g_{\varrho,i}(x) \vert+ (2\varrho+1) \sgn(g_{\varrho,i}(x))\}\bigg]
	%    \end{align*}
% \todo{Check plane white note, 29-03-2024, for calculations}
which we bound using \eqref{eq.gbound} as 
\begin{equation} \label{eq:ssecond}
	\vert X^{\prime\prime}_{\varrho,i}(x) \vert \le \tilde{C} \vert x \vert^{\varrho-1} ( 1 + \vert \ln (\vert x \vert) \vert).
\end{equation} 

Differentiating \eqref{eq.eqssecond},
\begin{align*}
	X_{\varrho,i}^{\prime\prime\prime}(x) \approx &|g_{\varrho,i}(x)|^{\varrho-1}\bigg\{2(1+\varrho) \,\sgn(g_{\varrho,i}(x))g_{\varrho,i}^{\prime}(x)\,g_{\varrho,i}^{\prime \prime}(x)+(1+\varrho)[\varrho\,\sgn(g_{\varrho,i}(x))\,(g_{\varrho,i}^\prime(x))^2\\
	&+|g_{\varrho,i}(x)|\, g_{\varrho,i}^{\prime \prime}(x)]\,\frac{g_{\varrho,i}^\prime(x)}{g_{\varrho,i}(x)}+[1+(1+\varrho)\,\ln(|g_{\varrho,i}(x)|)]\cdot [2\varrho\sgn(g_{\varrho,i}(x))g_{\varrho,i}^\prime(x)g_{\varrho,i}^{\prime\prime}(x)\\
	&\quad\quad\quad\quad\quad\quad\quad\quad\quad\quad\quad\quad \quad\quad+|g_{\varrho,i}(x)|\, g_{\varrho,i}^{\prime \prime \prime}(x)+\sgn(g_{\varrho,i}(x))\,g_{\varrho,i}^{\prime}(x)\, g_{\varrho,i}^{\prime\prime}(x)]\bigg\}\\
	&+(\varrho-1)\sgn(g_{\varrho,i}(x))g_{\varrho,i}^\prime(x)|g_{\varrho,i}(x)|^{\varrho-2}\bigg\{(1+\varrho)\, \sgn(g_{\varrho,i}(x))\,(g_{\varrho,i}^\prime(x))^2\\
	&\quad\quad\quad\quad+[\varrho\, \sgn(g_{\varrho,i}(x))(g_{\varrho,i}^{\prime}(x))^2
	+|g_{\varrho,i}(x)|g_{\varrho,i}^{\prime\prime}(x)]\cdot [1+(1+\varrho)\ln(|g_{\varrho,i}(x)|)]\bigg\},
\end{align*}
using \eqref{eq.gbound}, we have the following bounds
\begin{equation} \label{eq:third}
	\vert X^{\prime\prime\prime}_{\varrho,i}(x) \vert \le \tilde{C} \vert x \vert^{\varrho-2} ( 1 + \vert \ln (\vert x \vert) \vert).
\end{equation}

\subsection{Invariant cones}
\begin{definition}[Cone] 
	%\cite{MV04}
	Let $E$ be a vector space. A \emph{cone} in $E$ is a subset $C \subset E\setminus \{0\}$ such that for $\varphi \in C \text{ then } \lambda \varphi \in C, \text{ for each } \lambda >0.$
	%    \begin{equation*}
		%        \varphi \in C \text{ then } a \varphi \in C, \text{ for each } a >0.
		%    \end{equation*}
\end{definition}

Following the idea of \cite{LSV98} we define certain cones and show that they are invariant with respect to the operators defined in \eqref{eq:PF} and \eqref{eq:PF_branch}.
% $\mathcal{L}_\alpha$ and $\mathcal{N}_{\alpha,i}$.sssssss
We denote the Lebesgue measure on $S^1$ by $m$ and define 
\begin{equation} \label{eq:em}
	m(\varphi)=\int_{S^1} \varphi(x) \,dm=\int_{S^1} \varphi(x) \,dx.
\end{equation}
%which gives,

For all $\varphi,\psi$ for which the integrals make sense, we have
\begin{equation*}
	\int_{S^1} \psi\, \mathcal{L}_{\alpha} \varphi \,dm =\int_{S^1}  \varphi\, \psi \circ f \, dm. \label{eq:kop}
\end{equation*}

Setting, $\psi \equiv 1$ in the equation above, %\eqref{eq:kop}, 
we have that
\begin{equation}
	m(\mathcal{L}_{\alpha} \varphi)=m(\varphi). \label{eq.emel}
\end{equation}

%Next, we define some cones and show that they areinvariant with respect to $\mathcal{L}_\alpha$ and $\mathcal{N}_{\alpha,i}$.

%the Perron-Frobenius operator and the operator defined in \eqref{eq:PF_branch}. 
Let $a_1,b_1>0$, and define the cone 
\begin{equation}\label{eq:Kone1}
	\mathcal{C}_{*,1}=\left\{\varphi \in C^1(S^1\setminus\{0\}) \bigg\vert 0\le \varphi(x) \le 2h_\alpha(x) m(\varphi) , \vert \varphi^\prime (x) \vert \le \bigg(\frac{a_1}{\vert x \vert} + b_1 \bigg) \varphi(x) \right\}.
\end{equation}

It is straightforward to check that this is indeed a cone. Since $0<\alpha<1$ for the bounds on the density in \eqref{eq:dens.bound}, it follows that $\mathcal{C}_{*,1} \subset L^1(m)$. 
\begin{comment}
	I think that it is true for $\alpha=0$  as well.
\end{comment}
Also,
observe that 
%from equation \eqref{eq:dens.bound} and the definition of the cone $\mathcal{C}_{*,1}$ in equation \eqref{eq:Kone1}, 
\begin{equation}\label{eq:phi:bound}
	\varphi(x)\le \frac{2c_2}{\vert x \vert^\alpha}m(\varphi), \quad  \forall \varphi \in \mathcal{C}_{*,1}, x \in S^1\setminus\{0\},
\end{equation}
and for $\beta \ge \alpha \ge 0$, 
\begin{equation}\label{eq.incbeta}
	\mathcal{C}_{*,1}(\alpha,1, a_1, b_1) \subset \mathcal{C}_{*,1}\left(\beta, \frac{c_2}{c_1}, a_1, b_1\right).
\end{equation}
\begin{lemma}\label{Kone1} 
	%	For $a_1, b_1>0$, 
	$\mathcal{C}_{*,1}$ is invariant with respect to the Perron-Frobenius operator, provided we choose $a_1$, $b_1/a_1$ big enough.
\end{lemma}
\begin{proof} For $\varphi \in \mathcal{C}_{*,1}$, we show that the first condition is invariant by $\mathcal{L}_\alpha$. Indeed, from \eqref{eq:PF}, we have that by the contraction property of $\mathcal{L}_{\alpha}$ and \eqref{eq.emel}, $\mathcal{L}_\alpha\varphi(x) \le 2h_\alpha(x) m(\mathcal{L}_\alpha\varphi)$.
%	\begin{align*}
%		\mathcal{L}_\alpha\varphi(x)&=\sum_{f_{\alpha}(y)=x} \frac{\varphi(y)}{f_{\alpha}^{\prime}(y)}\\
%		& \le  \sum_{f_{\alpha}(y)=x} \frac{2h_\alpha(y)\int_{S^1}\varphi(y)\,dy}{f_{\alpha}^{\prime}(y)}\\
%		& \le  2\int_{S^1}\varphi(y)\,dy\sum_{f_{\alpha}(y)=x} \frac{h_\alpha(y)}{f_{\alpha}^{\prime}(y)}\\
%		&= 2 \mathcal{L}_\alpha h_\alpha(x) m(\varphi) =2h_\alpha(x) m(\mathcal{L}_\alpha\varphi).
%	\end{align*}	
By the positivity property of the Perron-Frobenius operator, we conclude the invariance of the first condition.
	Next, we show the invariance by $\mathcal{L}_\alpha$ of the second condition
	\begin{align*}
		\vert (\mathcal{L}_\alpha& \varphi)^\prime(x) \vert=\left\vert \sum_{f_{\alpha}(y)=x} \frac{f_{\alpha}^{\prime \prime}(y)}{( f_{\alpha}^{\prime}(y))^3} \varphi(y) + \frac{1}{(f_{\alpha}^{\prime}(y))^2} \varphi^\prime(y)\right\vert\\
	%	&=\left\vert \sum_{f_{\alpha}(y)=x} \bigg(\frac{f_{\alpha}^{\prime \prime}(y)}{( f_{\alpha}^{\prime}(y))^2} + \frac{1}{\varphi(y)( f_{\alpha}^{\prime}(y))} \varphi^\prime(y)\bigg)\frac{\varphi(y)}{(f_{\alpha}^{\prime}(y))}\right\vert\\
	%	&\le \sum_{f_{\alpha}(y)=x} \frac{\varphi(y)}{( f_{\alpha}^{\prime}(y))} \left(\frac{\vert f_{\alpha}^{\prime \prime}(y)\vert}{( f_{\alpha}^{\prime}(y))^2} + \frac{1}{\varphi(y)( f_{\alpha}^{\prime}(y))} \vert \varphi^\prime(y)\vert\right)\\
		&\le \sum_{f_{\alpha}(y)=x} \frac{\varphi(y)}{( f_{\alpha}^{\prime}(y))} \left(\frac{C\vert y \vert^{\alpha-1}}{(f_{\alpha}^{\prime}(y))^2} + \frac{a_1+b_1\vert y \vert}{\vert y \vert (f_{\alpha}^{\prime}(y))} \right)\\
	%	&\le \left(\frac{a_1}{\vert x \vert} + b_1 \right) \sum_{f_{\alpha}(y)=x} \frac{\varphi(y)}{(f_{\alpha}^{\prime}(y))} \sup_{y\in S^1}\left[\frac{\vert f_\alpha(y) \vert}{a_1+ b_1 \vert f_\alpha(y) \vert}\cdot\left(\frac{C\vert y \vert^{\alpha-1}}{(f_{\alpha}^{\prime}(y))^2} + \frac{a_1+b_1\vert y \vert}{\vert y \vert (f_{\alpha}^{\prime}(y))} \right)\right]\\
		&\le \left(\frac{a_1}{\vert x \vert} + b_1 \right) \mathcal{L}_\alpha\varphi(x) \sup_{y\in S^1}\left[\left(\frac{\vert f_\alpha(y) \vert}{a_1+ b_1 \vert f_\alpha(y) \vert}\cdot\frac{C\vert y \vert^{\alpha-1}}{(f_{\alpha}^{\prime}(y))^2} + \frac{\vert f_\alpha(y) \vert}{\vert y \vert \cdot f_{\alpha}^{\prime}(y) }\cdot\frac{a_1+b_1\vert y \vert}{a_1+ b_1 \vert f_\alpha(y) \vert} \right)\right].
	\end{align*}
	
Setting the expression in the square bracket as $\Omega_{1}(y)$.
%	We set
%	\begin{equation}\label{eq:Vmega}
%		\Omega_1(y)= \frac{\vert f_\alpha(y) \vert}{a_1+ b_1 \vert f_\alpha(y) \vert}\cdot\frac{C\vert y \vert^{\alpha-1}}{(f_{\alpha}^{\prime}(y))^2} + \frac{\vert f_\alpha(y) \vert}{\vert y \vert \cdot f_{\alpha}^{\prime}(y) }\cdot\frac{a_1+b_1\vert y \vert}{a_1+ b_1 \vert f_\alpha(y) \vert}.
%	\end{equation}
To complete the proof, we need to show that $\Omega_1(y)\le1$. We do this for $y$ in the neighbourhood of
	% $[0,\delta]$
	$\delta$, $\delta$ small, and also for $\delta<y<1-\delta$.
	
	For $y$ in the neighbourhood of $\delta$, we only need to show that
	\begin{equation*}
		\underbrace{\frac{|f_\alpha(y)|}{|y|f_{\alpha}^{\prime}(y)}\left(\frac{C|y|^{\alpha+1}}{|y|f_{\alpha}^{\prime}(y)(a_1+ b_1 \vert f_\alpha(y) \vert)-|f_\alpha(y)|(a_1+ b_1 \vert y \vert)}\right)}_{\Lambda_1(y)}\le 1
	\end{equation*}
	\begin{align*}
	%	\Lambda_1(y)&=\frac{|f_\alpha(y)|}{|y|f_{\alpha}^{\prime}(y)}\left(\frac{C|y|^{\alpha+1}}{a_1(|y|f_{\alpha}^{\prime}(y)-|f_\alpha(y)|)+ b_1 |y|\vert f_\alpha(y) \vert(f_{\alpha}^{\prime}(y)-1)}\right)\\
		\Lambda_1(y)&\le \frac{|f_\alpha(y)|}{|y|f_{\alpha}^{\prime}(y)}\left(\frac{C|y|^{\alpha+1}}{a_1(|y|f_{\alpha}^{\prime}(y)-|f_\alpha(y)|)}\right)
	\end{align*}

	%	We have that
	%	\begin{align*}
		%		\frac{1}{a_1+ b_1 \vert f_\alpha(y) \vert}&\le \frac{1}{a_1}\\
		%		\frac{a_1+b_1\vert y \vert}{a_1+ b_1 \vert f_\alpha(y) \vert}&=1- \frac{b_1(|f_\alpha(y)|-\vert y \vert)}{a_1+ b_1 \vert f_\alpha(y) \vert} = 1-\lambda_1,
		%	\end{align*}
	%	where $\displaystyle \lambda_1= \frac{b_1(|f_\alpha(y)|-\vert y \vert)}{a_1+ b_1 \vert f_\alpha(y) \vert}< 1$. Equation \eqref{eq:Vmega} now gives
	%	\begin{align*}
		%		\Omega(y)&= \frac{\vert f_\alpha(y) \vert}{a_1}\cdot\frac{C\vert y \vert^{\alpha-1}}{(f_{\alpha}^{\prime}(y))^2} + \frac{\vert f_\alpha(y) \vert}{\vert y \vert \cdot f_{\alpha}^{\prime}(y) }\cdot(1-\lambda_1)\\
		%		&\le \left(1+\frac{1}{a_1}\cdot\frac{C\vert y \vert^{\alpha}}{(f_{\alpha}^{\prime}(y))} -\lambda_1\right) \frac{\vert f_\alpha(y) \vert}{\vert y \vert \cdot f_{\alpha}^{\prime}(y)}.
		%	\end{align*}
	%	
	From \eqref{eq:zerothderivative} and \eqref{eq:firstderivative}, $\displaystyle\frac{\vert f_\alpha(y) \vert}{\vert y \vert \, f_{\alpha}^{\prime}(y)}\lesssim 1$. Hence, there exists $\delta>0$ such that choosing $a_1$ big enough, $\Omega_1(y)\le 1$.
	
	For $1-\delta> y > \delta$, there exists $\gamma$ such that $f_\alpha^\prime(y) \ge \gamma > 1$.
%	, 
%	which implies that $\displaystyle \frac{1}{f_\alpha^\prime(y)}\le~ \frac{1}{\gamma} <~1$. 
	Therefore, 
	\begin{equation*}
		\vert y \vert > \delta \,\, \Rightarrow \,\, \vert y \vert^{\alpha-1} < \delta^{\alpha-1}, \quad \text{for } \alpha \in [0,1).
	\end{equation*}
	
	For $b_1>0$, 
%	$\displaystyle \frac{\vert f_\alpha(y) \vert}{a_1+ b_1 \vert f_\alpha(y) \vert} \le \frac{1}{a_1}$. Simplifying, we have that $$\displaystyle \vert f_\alpha(y) \vert \bigg(\frac{a_1+b_1\vert y \vert}{a_1+ b_1 \vert f_\alpha(y) \vert}\bigg)\le~\frac{a_1}{b_1}+~\vert y \vert.$$ 
%	
%	Substituting values we get that,
we have that 
	\begin{align*}
	%	\Omega(y)&\le \frac{1}{a_1}\cdot\frac{C\vert y \vert^{\alpha-1}}{\vert f_{\alpha}^{\prime}(y) \vert^2} + \frac{1}{\vert y \vert \vert f_{\alpha}^{\prime}(y) \vert}\cdot\bigg(\frac{a_1}{b_1}+\vert y \vert \bigg)\\
		\Omega_1(y)&\le \frac{1}{a_1}\cdot\frac{C\delta^{\alpha-1}}{\gamma^2} + \frac{a_1}{b_1} \cdot \frac{1}{\delta \gamma}+\frac{1}{ \gamma}.
	\end{align*}
	$\Omega_1(y)\le 1$, provided we choose $a_1$ and $b_1/a_1$, big enough.
\end{proof}

%Observe that from equation \eqref{eq:dens.bound} and the definition of the cone $\mathcal{C}_{*,1}$ in equation \eqref{eq:Kone1}, 
%\begin{equation}\label{eq:phi:bound}
%	\varphi(x)\le \frac{2c_2}{\vert x \vert^\alpha}m(\varphi), \quad  \forall \varphi \in \mathcal{C}_{*,1}, x \in S^1\setminus\{0\}.
%\end{equation}
%\noindent
Following similar steps as in \cite[Lemma 5.2]{LSV98}, we have the following result.
\begin{proposition} \label{minphi}
	If $\varphi \in \mathcal{C}_{*,1}$, then
	\begin{equation*}
		\min_{x \in S^1\setminus B_\delta(0)} \varphi(x) \ge \frac{\delta^{a_1}}{2 \e^{b_1(1-\delta)}} \int_{S^1} \varphi(x) \,dx,
	\end{equation*}
	choosing $\delta$ small enough.
\end{proposition}
\begin{lemma}\label{konetwo}
	There exists a $\delta >0$ and $\gamma>0$, such that
	\begin{equation}
		\mathcal{C}_{*,2}=\bigg\{\varphi \in \mathcal{C}_{*,1} \bigg\vert \varphi(x) \ge \gamma \int_{S^1} \varphi(x) \,dx, \text{ for } |x|\le \delta \bigg\}
	\end{equation}
	is invariant with respect to the Perron-Frobenius operator.
\end{lemma}
\begin{proof}
	For $|x| \le \delta$, let $f_{\alpha,i}^{-1}(x)=y_i$, $i=1,\cdots, d$. 
	%consists of at least four points, depending on the degree of the circle, 
	Denote by $y_*$ the $y_i$ on the first or last branch such that  $|y_{*}|\le \delta$. Suppose also that $\mu=\Vert f_\alpha ^\prime(y_{i}) \Vert_\infty$.  We choose $\delta$ small enough such that  
	%and $|{y}_{i}|~\ge~\delta$
	%and hence, $\varphi \in \mathcal{C}_{*,1}$
	%on the remaining branches
	%(for $|y_{i}|\ge \delta$), w 
	%so that for 
	% with $|y_*|\le \delta$
	by \eqref{eq:firstderivative}, $f_\alpha^\prime(y) \approx (1 + |y|^\alpha)$, so that
	\begin{align*}
		\frac{1}{f_\alpha^\prime(y_*)} \ge \frac{1}{C(1 + \vert y_* \vert^\alpha)} \ge \frac{1}{C}(1- \delta^\alpha), \quad C\ge 1,
	\end{align*}
	Proposition \ref{minphi} and $\frac{1}{C}(1- \delta^\alpha)+\mu^{-1}>1$ holds. From \eqref{eq:PF}, we have that
	\begin{align*}
		\mathcal{L}_\alpha \varphi(x) &\ge \frac{\varphi(y_{*})}{f_\alpha^\prime(y_{*})}+\frac{\varphi(y_{i})}{\|f_\alpha^\prime(y_{i})\|_\infty}= (f_\alpha^\prime(y_{*}))^{-1} \varphi(y_{*}) + \|f_\alpha^\prime(y_{i})\|_\infty^{-1}\varphi(y_{i})\\
%		&\ge \left[\frac{1}{C}(1- \delta^\alpha) \cdot \gamma \int_{S^1} \varphi(x) \,dx + \mu^{-1} \varphi(y_i) \right]\\
%		&\ge \left[\frac{1}{C}(1- \delta^\alpha) \cdot \gamma \int_{S^1} \varphi(x) \,dx + \mu^{-1} \min\left\{\gamma\int_{S^1} \varphi(x)\, dx,\frac{\delta^{a}}{2\e^{b(1-\delta)}}\int_{S^1}\varphi \,dx \right\} \right]\\
%		&\ge \left[\frac{1}{C}(1- \delta^\alpha) \cdot \gamma + \mu^{-1} \min\left\{\gamma,\frac{\delta^{a}}{2 \e^{b(1-\delta)}} \right\} \right] \int_{S^1}\varphi \,dx\\
		&\ge \left[\frac{1}{C}(1- \delta^\alpha) \cdot \gamma + \mu^{-1} \min\left\{\gamma, \frac{\delta^{a}}{2\e^{b(1-\delta)}} \right\} \right] \int_{S^1}\mathcal{L}_\alpha\varphi \,dx\\
		&\ge \gamma\int_{S^1}\mathcal{L}_\alpha\varphi \,dx.
	\end{align*}
	
	%Provided we choose $\gamma$ small.
\end{proof}
From Proposition \ref{minphi} and Lemma \ref{konetwo} we have that $\inf_{x \in S^1} \varphi(x) \ge \gamma\int_{S^1} \varphi(x) dx$, which implies that
\begin{equation}\label{eq.PFgam}
	\inf_{n\ge 0} \inf_{x \in S^1} \mathcal{L}_{\alpha}^n 1 \ge \gamma >0,
\end{equation}
particularly, since the constant function $1 \in \mathcal{C}_{*,1}$.

In the spirit of \cite{BT16}, we define the following cone for higher order derivatives, and show that it is invariant with respect to the Perron-Frobenius 
%transfer 
operator. For $a_1,a_2, a_3, b_1, b_2, b_3>~0$, define the cone
\begin{align}
	\mathcal{C}
	=\bigg\{\varphi\in C^{(3)}(S^1\setminus\{0\})\bigg\vert \varphi(x)\ge 0,\, &\vert \varphi^{\prime}(x) \vert \le \bigg(\frac{a_1}{\vert x \vert}+{b}_1\bigg)\varphi(x), \vert \varphi^{\prime \prime}(x) \vert \le \left(\frac{a_2}{x^2}+{b}_2\right)\varphi(x),\nonumber\\
	& \quad\quad \quad \quad \vert \varphi^{\prime \prime \prime}(x)\vert \le \left(\frac{a_3}{|x|^3}+{b}_3\right)\varphi(x), \forall x \in S^1\setminus\{0\} \bigg\}.
\end{align}
The proof of the result below follows similar ideas from appendix A of \cite{BT16} and Lemma \ref{Kone1}.
\begin{lemma} \label{inv.cone}
	Suppose that
	% $a_1, a_2, a_3, b_1, b_2, b_3>0$ 
	%  $a_j, b_j>0$, $j=1,2,3$, $\frac{b_2}{b_1}, \frac{b_1}{a_1}, \frac{b_2}{a_1}, \frac{b_2}{a_2}, \frac{b_3}{a_1}, \frac{b_3}{b_1}, \frac{b_3}{b_2}, \frac{b_3}{a_2}, \frac{b_3}{a_1}, \frac{b_3}{a_3}, 
	$\frac{\min\{a_2,b_2\}}{\max\{a_1,b_1\}}$, $ \frac{\min\{a_3,b_3\}}{\max\{a_1,b_1\}}, \frac{\min\{a_3,b_3\}}{\max\{a_2,b_2\}}$
	%$b_2 \ge b_1$, $b_1/a_1$, $b_2/a_1$, $b_2/a_2$
	are large enough. Then 
	%$h_\alpha \in \mathcal{C}$ and 
	$\mathcal{C}$ is invariant with respect to the operators $\mathcal{L}_\alpha$ and $\mathcal{N}_{\alpha,i}$, for $i \in \{1, d\}$. 
	%That is, 
	%$\mathcal{L}_{\alpha} \mathcal{C} \subset \mathcal{C}, \quad \mathcal{N}_{\alpha,i} \mathcal{C} \subset \mathcal{C} \quad \text{for } i \in \{1, d\}.$
%\begin{proof}
%	 appendix.
%\end{proof}
\end{lemma}
%\begin{proof}
	%Now, we define a transfer operator associated with the first and last branch, the middle branch(es) of $f_\alpha(x)$ may be independent of $\alpha$. For any $\varphi \in L^1(m)$ and $i \in \{1,d\}$, define\\
	% If $\varphi \in \mathcal{C}$.
	%    $$\mathcal{N}_{\alpha,i} \varphi(x)=\frac{\varphi(g_{\alpha,i}(x))}{f_{\alpha,i}^{\prime}(g_{\alpha,i}(x)}=\frac{\varphi(y)}{f_{\alpha,i}^{\prime}(y)}, \quad i \in \{1,d\}.$$

\subsection{A random perturbed operator, distortion property and decay estimate}
%Let $\mu_\varrho$ be an absolutely continuous invariant measure with respect the Lebesgue measure on $S^1$, with density $h_\varrho=\frac{d\mu_{\varrho}}{dx}$. 
%Before we state the rate of decay result with respect to the Lebesgue measure. 
Following the approach in \cite{LSV98}, we state the rate of decay. Firstly, we define the averaging operator and the perturbed operator respectively as follows:
% we introduce the concept of random perturbation 
\begin{align}
%	B_{\varepsilon}(x)&=\left\{y \in S^1: \vert x-y \vert\le \varepsilon \right\}, \nonumber\\
%	%\end{align*}
%	%and the averaging operator
%	%\begin{align}
	\mathbf{A}_{\varepsilon} \varphi(x)&=\frac{1}{2\varepsilon} \int_{B_\varepsilon(x)} \varphi(y) \,dy, \quad \varepsilon>0 \label{eq:average},\\
	%\end{align}
	%We also define the perturbed operator as
	%\begin{align}
	\mathbf{P}_{\varepsilon}&=\mathcal{L}^{n_\varepsilon}_{\alpha} \mathbf{A}_{\varepsilon}, \quad n_\varepsilon \in \mathbb{N} \label{eq:pertbd},
\end{align}
where $B_{\varepsilon}(x)$ is a ball centered at $x \in S^1$ with radius $\varepsilon$, 
%$\mathbf{A}_{\varepsilon}$ and $\mathbf{P}_{\varepsilon}$ are , 
with
$n_\varepsilon=\mathcal{O}(\varepsilon^{-\alpha})$. Next, we show that for $\varphi \in \mathcal{C}_{*,1}$, 
%(\alpha, a, b)
the Perron-Frobenius operator is approximated by the random perturbed operator.
\begin{lemma} \label{radom.pert}
	For $\varphi \in \mathcal{C}_{*,1}$,
	\begin{equation*}
		\Vert \mathcal{L}^{n_\varepsilon}_{\alpha} \varphi - \mathbf{P}_{\varepsilon} \varphi \Vert_{1} \le k_1 \Vert \varphi \Vert_1 \varepsilon^{1-\alpha},
	\end{equation*}
	where $\displaystyle k_1=\frac{18 c_2 \max\{a_1,b_1,1\}}{\alpha(1-\alpha)}$.
\end{lemma}
\begin{proof}
%	This follows from \cite{LSV98} and using
%	
%	$$\vert \varphi(x)-\varphi(y) \vert \le \sup_{z \in [x,y]} \vert \varphi^\prime (z) \vert \varepsilon \le 2c_2\varepsilon(a_1  \vert x \vert^{-1-\alpha}+ b_1\vert x \vert^{-\alpha}).$$
%	

	From the definition of the perturbed operator and the contraction property of the Perron-Frobenius operator, 
	\begin{comment}
		We justify the use of the property \ref{C2} by noting that $\mathcal{C}_{*,1} \subset L^1(m)$
	\end{comment}
	\begin{equation*}
		\Vert \mathcal{L}^{n_\varepsilon}_{\alpha} \varphi - \mathbf{P}_{\varepsilon} \varphi \Vert_{1} \le \Vert \varphi - \mathbf{A}_{\varepsilon} \varphi \Vert_{1}.
	\end{equation*}
	
	Assuming that $m(\varphi)=1$, the estimates in \eqref{eq:phi:bound} gives that $\varphi(x) \le 2c_2 \vert x \vert^{-\alpha},$
%	%and \eqref{eq:est:varphi}
%	%We didn`t need the estimate \eqref{eq:est:varphi}
%	\begin{align*}
%		\varphi(x) &\le 2c_2 \vert x \vert^{-\alpha},
%		%\\
%		%\frac{\varphi(x)}{\varphi(y)} &\le \bigg(\frac{\vert x \vert^{\sgn(x)}}{\vert y \vert^{\sgn(y)}}\bigg)^a \exp(b(x-y)), \quad \forall x \ge y > \delta.
%	\end{align*}
	which would enable us get the desired bounds.
	\begin{align*}
		\Vert \varphi - \mathbf{A}_{\varepsilon}\varphi \Vert_1
%		&=\bigg\Vert \varphi(x) -\frac{1}{2\varepsilon}\int_{B_\varepsilon(x)} \varphi(y)\, dy \bigg\Vert_1\\
%	%%	&=\int_0^1\bigg\vert \varphi(x) -\frac{1}{2\varepsilon}\int_{B_\varepsilon(x)} \varphi(y) \,dy \bigg\vert \,dx\\
%	%	&= \int_{\varepsilon}^{1-\varepsilon} \bigg\vert \varphi(x) -\frac{1}{2\varepsilon}\int_{B_\varepsilon(x)} \varphi(y) \,dy \bigg\vert \,dx + \int_{B_{\varepsilon}(0)}\bigg\vert \varphi(x) -\frac{1}{2\varepsilon}\int_{B_\varepsilon(x)} \varphi(y)\, dy \bigg\vert \,dx\\
%	%	&= \int_{\varepsilon}^{1-\varepsilon} \bigg\vert \frac{1}{2\varepsilon}\int_{B_\varepsilon(x)}\left[\varphi(x) - \varphi(y)\right] \,dy \bigg\vert\, dx + \int_{B_{\varepsilon}(0)}\bigg\vert \varphi(x) -\frac{1}{2\varepsilon}\int_{B_\varepsilon(x)} \varphi(y) \,dy \bigg\vert\, dx\\
%		%%&\le \int_{\varepsilon}^{1-\varepsilon}  \frac{1}{2\varepsilon}\int_{B_\varepsilon(x)}\vert\varphi(x) - \varphi(y)\vert dy  dx + \int_{B_{\varepsilon}(0)}\bigg\vert \varphi(x) -\frac{1}{2\varepsilon}\int_{B_\varepsilon(x)} \varphi(y) dy \bigg\vert dx\\
%		&\le \frac{1}{2\varepsilon}\int_{\varepsilon}^{1-\varepsilon} \int_{B_{\varepsilon}(x)} \vert \varphi(x)-\varphi(y) \vert \,dy \,dx + \int_{B_{\varepsilon}(0)} \bigg\vert \varphi(x) - \frac{1}{2\varepsilon} \int_{B_{\varepsilon}(x)} \varphi(y)\,dy \bigg\vert\, dx\\
		&\le \frac{1}{2\varepsilon}\int_{\varepsilon}^{1-\varepsilon} \int_{B_{\varepsilon}(x)} \vert \varphi(x)-\varphi(y) \vert \,dy\, dx+ \int_{B_{\varepsilon}(0)} \vert \varphi(x) \vert\, dx \\ 
		&\quad\quad\quad\quad\quad\quad\quad\quad\quad\quad\quad\quad\quad\quad\quad\quad\quad\quad\quad+ \frac{1}{2\varepsilon} \int_{B_{\varepsilon}(0)}\int_{B_{\varepsilon}(x)} \vert \varphi(y)\vert \,dy\,dx.
	\end{align*}
	% Since $\varphi$ is decreasing, $\varphi(y)-\varphi(x)\ge 0$ on $[x-\varepsilon, x]$ and $\varphi(x)-\varphi(y)\ge 0$ on $[x, x+\varepsilon]$, $\varphi(\cdot) \ge 0$.
	
	By changing the order of integration in the last integral, we have that
	\begin{align*}
	%	\Vert \varphi -\mathbf{A}_\varepsilon \varphi\Vert_1 &\le \frac{1}{2\varepsilon}\int_{\varepsilon}^{1-\varepsilon} \int_{B_{\varepsilon}(x)} \vert \varphi(x)-\varphi(y) \vert \,dy\, dx+ \int_{B_{\varepsilon}(0)}  \varphi(x)  \,dx  + \int_{B_{2\varepsilon}(0)} \varphi(y) \,dy\\
			\Vert \varphi -\mathbf{A}_\varepsilon \varphi\Vert_1&\le \frac{1}{2\varepsilon}\int_{\varepsilon}^{1-\varepsilon} \int_{B_{\varepsilon}(x)} \vert \varphi(x)-\varphi(y) \vert \,dy\, dx+ 2\int_{-{2\varepsilon}}^{{2\varepsilon}}\varphi(y)\, dy.
	\end{align*}
	%  \begin{align*}
		%     \Vert \varphi -\mathbf{A}_\varepsilon \varphi\Vert_1 \le \frac{1}{2\varepsilon}\int_{\varepsilon}^{1-\varepsilon}\bigg[\int_{x-\varepsilon}^x (\varphi(y)-\varphi(x)) dy + \int_x^{x+\varepsilon} (\varphi(x)-\varphi(y)) dy\bigg] dx+ 2\int_{-{2\varepsilon}}^{{2\varepsilon}}\varphi(y) dy
		% \end{align*}
	% In the first sum, split the inner integral, and use the property that the function $\varphi (\cdot)$ is non-decreasing. in the second sum, changing the order of integration, we have that
	%\begin{align*}
	%       \Vert \varphi - A_{\varepsilon}\varphi \Vert_1 \le \frac{1}{2\varepsilon}\int_{\varepsilon}^{1-\varepsilon} \bigg\{\int_{x-\varepsilon}^{x} ( \varphi(y)-\varphi(x)) dx +\int_{x}^{x+ \varepsilon}& (\varphi(x)-\varphi(y))dy \bigg\} dx \\&+ \int_{B_{\varepsilon}(0)} \varphi(x)dx+ \int_{B_{2\varepsilon}(0)} \varphi(y)dy 
	%    \end{align*}

The integrand in the first integral is bounded as follows. For $x,y \in S^1$ 
%and $x \le y$  
such that $\vert x - y\vert \le \varepsilon$, by \eqref{eq:Kone1} and 
%equation
\eqref{eq:phi:bound},
$$\vert \varphi(x)-\varphi(y) \vert \le \sup_{z \in [x,y]} \vert \varphi^\prime (z) \vert \varepsilon \le 2c_2\varepsilon(a_1  \vert x \vert^{-1-\alpha}+ b_1\vert x \vert^{-\alpha}).$$

We have that
\begin{align*}
%	\Vert \varphi - \mathbf{A}_{\varepsilon}\varphi \Vert_1 &\le c_2\int_{\varepsilon}^{1-\varepsilon} \int_{B_{\varepsilon}(x)} (a_1  \vert x \vert^{-1-\alpha}+ b_1\vert x \vert^{-\alpha}) \,dy\,dx +  8c_2\int_{0}^{{2\varepsilon}} \vert y \vert^{-\alpha} \,dy\\
\Vert \varphi - \mathbf{A}_{\varepsilon}\varphi \Vert_1 &\le 2c_2\varepsilon\int_{\varepsilon}^{1-\varepsilon}  (a_1  \vert x \vert^{-1-\alpha}+ b_1\vert x \vert^{-\alpha}) \,dx +  8c_2\int_{0}^{{2\varepsilon}} \vert y \vert^{-\alpha} \,dy\\
	&= 2c_2\varepsilon\bigg[a_1\frac{\sgn(x) \vert x \vert^{-\alpha}}{\alpha}\bigg\vert_{1-\varepsilon}^{\varepsilon} + b_1\frac{\sgn(x)\vert x \vert^{1-\alpha}}{1-\alpha}\bigg\vert_{\varepsilon}^{1-\varepsilon}\bigg] + 8c_2 \frac{\sgn(x)\vert x \vert^{1-\alpha}}{1-\alpha}\bigg\vert_0^{2\varepsilon}\\
%	&= 2c_2\varepsilon\bigg[a_1\bigg(\frac{\varepsilon^{-\alpha}}{\alpha}-\frac{(1-\varepsilon)^{-\alpha}}{\alpha}\bigg)+ b_1\bigg(\frac{(1-\varepsilon)^{1-\alpha}}{1-\alpha}-\frac{\varepsilon^{1-\alpha}}{1-\alpha}\bigg)\bigg] + 8c_2 \frac{(2\varepsilon)^{1-\alpha}}{1-\alpha}\\
	&\le2c_2 \max\{a_1,b_1\}\varepsilon \cdot\\
	&\quad \quad\quad\left(\frac{(1-\alpha)\varepsilon^{-\alpha}-(1-\alpha)(1-\varepsilon)^{-\alpha}+\alpha(1-\varepsilon)^{1-\alpha}-\alpha\varepsilon^{1-\alpha}}{\alpha(1-\alpha)}\right)+ 8c_2 \frac{(2\varepsilon)^{1-\alpha}}{1-\alpha}\\
	%	&\le2c_2 \max\{a_1,b_1\}\varepsilon\bigg(\frac{\varepsilon^{-\alpha}}{\alpha}-\frac{\varepsilon^{1-\alpha}}{1-\alpha}-\frac{1}{\alpha}+ \frac{1}{1-\alpha}\bigg)+ 8c_2 \frac{(2\varepsilon)^{1-\alpha}}{1-\alpha}\\
	&\le2c_2 \max\{a_1,b_1\}\varepsilon \cdot\\
	&\left(\frac{\varepsilon^{-\alpha}}{\alpha(1-\alpha)}+\frac{\alpha(1-\varepsilon)^{1-\alpha}-(1-\alpha)(1-\varepsilon)^{-\alpha}-\alpha\varepsilon^{-\alpha}-\alpha\varepsilon^{1-\alpha}}{\alpha(1-\alpha)}\right)+ 8c_2 \frac{(2\varepsilon)^{1-\alpha}}{1-\alpha}\\
	%	&\le2c_2 \max\{a_1,b_1\}\frac{\varepsilon^{1-\alpha}}{\alpha(1-\alpha)}+ 8c_2 \frac{(2\varepsilon)^{1-\alpha}}{1-\alpha}\\
	&\le2c_2 \max\{a_1,b_1\}\frac{\varepsilon^{1-\alpha}}{\alpha(1-\alpha)}+ \frac{16 c_2\varepsilon^{1-\alpha}}{\alpha(1-\alpha)}\\
	\Vert \varphi - \mathbf{A}_{\varepsilon}\varphi \Vert_1 &\le \frac{18 c_2 \max\{a_1,b_1,1\}}{\alpha(1-\alpha)}\varepsilon^{1-\alpha}.
\end{align*}
\end{proof}
From \eqref{eq:PF}, \eqref{eq:average}, \eqref{eq:pertbd} and  the kernel $\displaystyle \mathcal{K}_\varepsilon(x,z):= \frac{1}{2\varepsilon}\mathcal{L}_\alpha^{n_\varepsilon}\chi_{B_\varepsilon(z)}(x)$,
\begin{align}
%\mathbf{P}_\varepsilon \varphi(x)&=   \mathcal{L}_{\alpha}^{n_\varepsilon}\frac{1}{2\varepsilon} \int_{B_\varepsilon(x)} \varphi(y)\, dy \nonumber\\
%%&= \frac{1}{2\varepsilon}\sum_{f^{n_\varepsilon}_{\alpha}(y)=x} \frac{ \int_{B_\varepsilon(y)} \varphi(z)\, dz}{(f_{\alpha}^{n_\varepsilon})^{\prime}(y)} \nonumber\\
%&= \frac{1}{2\varepsilon}\sum_{f^{n_\varepsilon}_{\alpha}(y)=x} \frac{ \int_0^1 \chi_{B_\varepsilon(y)}(z) \varphi(z)\, dz}{(f_{\alpha}^{n_\varepsilon})^{\prime}(y)} \nonumber\\
%&= \frac{1}{2\varepsilon}\sum_{f^{n_\varepsilon}_{\alpha}(y)=x} \frac{ \int_0^1 \chi_{B_\varepsilon(z)}(y) \varphi(z)\, dz}{(f_{\alpha}^{n_\varepsilon})^{\prime}(y)}\nonumber\\
%&=\frac{1}{2\varepsilon} \int_0^1 \mathcal{L}_\alpha^{n_\varepsilon}\chi_{B_\varepsilon(z)}(x) \varphi(z)\, dz \nonumber\\
\mathbf{P}_\varepsilon \varphi(x)&= \int_0^1 \mathcal{K}_\varepsilon(x,z) \varphi(z) \,dz. \label{eq:peekay}
\end{align}

Next, for an appropriate choice of $n_\varepsilon$,
we verify the positivity of the kernel
$\mathcal{K}_\varepsilon(x,z)$, an 
estimate that plays a crucial role in establishing the desired decay properties. 
%Since the Perron-Frobenius operator for this class of maps lacks the spectral gap property, the positivity of this kernel provides a key estimate for demonstrating the decay of correlation despite this absence.

%
%Next, for some good choice of $n_\varepsilon$ we 
%check the positivity of the kernel $\mathcal{K}_\varepsilon(x,z)$, an estimate which will be important in showing the decay claims. Since for this class of maps, the Perron-Frobenius operator does not posses the spectral gap property, the positivity of this kernel will allow us show
%some decay of correlation in the absence of spectral gap.
% %spectral gap property for the perturbed operator acting on $L^1(m)$ functions.

\begin{proposition}\label{propkay}
	There exists $n_\varepsilon= \mathcal{O}(\varepsilon^{-\alpha})$ and $\gamma>0$ such that for each $\varepsilon>0$, $x,z \in S^1$,
	\begin{equation} \label{eq:kay}
		\mathcal{K}_\varepsilon(x,z) \ge \gamma.
	\end{equation}
	%choosing .
\end{proposition}
\begin{proof}
	Firstly, recall the definition
	\begin{equation*}
		2\varepsilon \mathcal{K}_\varepsilon(x, z) = \mathcal{L}_\alpha^{n_\varepsilon}\chi_{B_\varepsilon(z)}(x).
	\end{equation*}
	
	We have from \eqref{eq:PF} that, 
	%following
	\begin{align*}
		\mathcal{L}_\alpha^{n_\varepsilon}\chi_{B_\varepsilon(z)}(x)
		%&= \sum_{f_\alpha^{n_\varepsilon}(y)=x} \frac{\chi_{B_\varepsilon(z)}(y) }{(f_\alpha^{n_\varepsilon})^\prime(y)}\\
		%&= \chi_{f_\alpha^{n_\varepsilon}({B_\varepsilon(z)})}(x) \sum_{f_\alpha^{n_\varepsilon}(y)=x} \frac{1}{(f_\alpha^{n_\varepsilon})^\prime(y)}\\
		&\ge \chi_{f_\alpha^{n_\varepsilon}({B_\varepsilon(z)})}(x) \inf_{y \in B_\varepsilon(z) } \frac{1}{(f_\alpha^{n_\varepsilon})^\prime(y)}.
	\end{align*}
	
	By $\ref{s1}-\ref{s3}$, the inverse of the derivative of all $f_{\alpha}$ are bounded from below, we therefore conclude that for any interval $I$ of length at least $z_{k_0-1}-z_{k_0}$(resp. $z^{\prime}_{k_0}-z^{\prime}_{k_0-1}$) ($k_0$ fixed), there are constants $n_0$ and $c_0$ such that
	$$ \mathcal{L}^n_{\alpha} \chi_I \ge c_0,$$
	provided $n \ge n_0$. Hence, we have to control 
	$$\inf_{y \in B_\varepsilon(z) } \frac{1}{(f_\alpha^{m})^\prime(y)},
	$$
	where $m$ is the time needed for an interval $J=B_\varepsilon(z)$ of length at least
	%greater than 
	$2\varepsilon$ to cover the whole circle. 
	In addition, we estimate 
	%the minimal time (which we shall call $n_\varepsilon$) needed for an interval of length at least $2\varepsilon$ to cover the whole circle.
	$$n_\varepsilon:= \inf\{n \ge 1: f_\alpha^n (J)=S^1, \text{ for all } J \text{ with } |J|\ge 2\varepsilon\}.$$
	%From Proposition \ref{minphi}. 
	%We use the notation 
	
	To check the distortion 
	%property 
	and thus the positivity of the kernel, we follow closely the strategy of proof in \cite{AHN14,LSV98}. Now, we fix the notation that we shall be using. Recall the definition of $z_k$ (resp. $z_k^\prime$) in Subsection \ref{ABNIP},
	setting $-z_{k}= z_{k}^\prime$  
	and let $I_0= B_{z_{k}}(0)$ (for a fixed $k$) be the \emph{intermittent region} and $I_0^c=S^1\setminus B_{z_{k}}(0)$ be the \emph{hyperbolic region}, we note that the map is uniformly expanding in the hyperbolic region,
	%no matter the degree of the map, and this is where the degrees $\{2, \dots,d-1\}$ exists, for $d\ge 3$ 
	and possesses a uniformly bounded second derivative. 
	
	%$I_n=[z_n, z_{n-1}]$ and $I^\prime_n=[z^\prime_{n-1}, z^\prime_n]$, and satisfies $f^n_\alpha(I_n)=[1/2,1]$ and $f^n_\alpha(I^\prime_n)=[0,1/2]$.

	%%We set $-z_{k_0}= z_{k_0}^\prime$, as defined in Section \ref{ABNIP}.\\
	%Let $J$ be an interval of length $2\varepsilon$,
	
	Consider the interval $J$  and
	%$J$ an interval of length $2\varepsilon$ and 
	its iterates which we call $K= f_\alpha^n(J)$, for some $n$.
	%until it covers the whole $S^1$, 
	Controlling the distortion, we explore different possibilities that the dynamics 
	%of
	%this map 
	might take. $K$ takes one of the following
	\begin{enumerate}[label=${\arabic*}$]
		\item\label{DC1} $K \cap I_0= \emptyset$;
		\item \label{DC2} $K \cap I_0 \neq \emptyset$ and $K$ contains, at most, one $z_l$ or $z_l^\prime$ for $l>k$; 
		\item \label{DC3} $K \cap I_0 \neq \emptyset$ and $K$ contains more than one $z_l$ or $z_l^\prime$ for $l>k$. 
	\end{enumerate}
	%Assume that $d=2$.
	
	We remark that the proof for the above cases when $d=2$ and $d\ge 3$ are similar, since for $d\ge 3$, the middle branches are covered by the case 1. Thus, we proceed with the proof for $d=2$.
		%  we  shall proceed with both of them separately. 
	%\end{remark}
	
	{\bf{Case \boldmath{\ref{DC1}}:}} Now, suppose that we are in the scenario 1
	we let $n_1 \ge 1$ be the time spent iterating the interval $K$ in the region $I_0^c$ before it enters the $I_0$ region and case 2 or 3 occurs.
	
	Let $\displaystyle D:= \sup_{\xi \in I_0^c}\frac{f_\alpha^{\prime \prime}(\xi)}{(f_\alpha^\prime(\xi))^2}$. By 
	%, for $\displaystyle y\in I_0^c$, $\frac{1}{f_\alpha^\prime(y)}\le \frac{1}{\lambda}<1$, by 
	the standard distortion estimate we have that, for all $x,y \in K$, 
	%we note that we can only start the algorithm at either the sub-partition $J_{z_n} \text{ or } J_{z^\prime_n}$ (of course $d=2$ here).
	using the mean value theorem twice, there exists $\eta, \xi \in K$, such that
	\begin{align*}
		\log \frac{|f_\alpha^\prime (x)|}{|f_\alpha^\prime (y)|}&= 	\log \left(1+\frac{|f_\alpha^\prime (x)-f_\alpha^\prime (y)|}{|f_\alpha^\prime (y)|}\right)\le \frac{|f_\alpha^\prime (x)-f_\alpha^\prime (y)|}{|f_\alpha^\prime (y)|}= \frac{f_\alpha^{\prime\prime} (\xi)\,|x-y|}{|f_\alpha^\prime (y)|}\\
		&=\frac{|f_\alpha^{\prime\prime}(\xi)|}{|f_\alpha^\prime (y)|} \frac{ |f_\alpha(x)-f_\alpha(y)|}{f_\alpha^\prime (\eta)}.
	\end{align*}
	
	Since we are in the hyperbolic region, $f_\alpha^\prime>1$ (by \ref{s2}),
	%of the map)
	also $f_\alpha$ is $C^2$ on a compact space, $|f_\alpha^{\prime\prime}(\xi)|$ is bounded. Therefore,
	\begin{equation*}
		\log \frac{|f_\alpha^\prime (x)|}{|f_\alpha^\prime (y)|}\le D |f_\alpha(x)-f_\alpha(y)|.
	\end{equation*}
	
	Now, by the chain rule,
	\begin{align*}
		\log \frac{(f_\alpha^{n_1})^\prime(x)}{(f_\alpha^{n_1})^\prime(y)} %&\le \sum_{j=0}^{{n_1}-1}\left| \log f_\alpha^\prime(f_\alpha^j (x)) - \log f_\alpha^\prime(f_\alpha^j (y)) \right|
		%\le  D\sum_{j=0}^{{n_1}-1} \left| f_\alpha^j (x) - f_\alpha^j (y) \right|\\
		&\le  D \sum_{j=0}^{{n_1}-1} \frac{1}{\lambda^{{n_1}-j}} \left| f_\alpha^{n_1} (x) - f_\alpha^{n_1} (y) \right|\le  \frac{D}{\lambda-1} \left| f_\alpha^{n_1} (K) \right|.
	\end{align*}
	
%	Hence,
%	\begin{align*}
%		\frac{(f_\alpha^{n_1})^\prime(x)}{(f_\alpha^{n_1})^\prime(y)}	&\le  \exp\left(\frac{D}{\lambda-1} \left| f_\alpha^{n_1} (K) \right|\right).
%	\end{align*}
	
	Integrating with respect to $y$,
%	\begin{align*}
%		\frac{(f_\alpha^{n_1})^\prime(x)\,|K|}{|f_\alpha^{n_1}(K)|}	&\le  \exp\left(\frac{D}{\lambda-1} \left| f_\alpha^{n_1} (K) \right|\right),
%	\end{align*}
	we therefore deduce that
	$$\mathcal{L}_\alpha^{{n_1}}\chi_{B_\varepsilon(z)}(x) \ge \chi_{f_\alpha^{{n_1}}({B_\varepsilon(z)})}(x)\, 	\frac{|K|}{|f_\alpha^{n_1}(K)|}	\exp\left(-N_1 \left| f_\alpha^{n_1} (K) \right|\right),
	$$
	taking $N_1=\frac{D}{\lambda-1}$.
	
	{\bf{Case \boldmath{\ref{DC2}}:}} Let us assume that $K$ is in $I_0$, such that  $K \subset (z_{l},z_{l-2})$ or $K \subset (z_{l-2}^\prime, z_{l}^\prime)$, where $l=k+k_1$. Here, after $k_1$ iterations, the image will be in the hyperbolic region $I_0^c$ and we continue the algorithm as in case 1. We control the distortion while $K$ traverses 
	%the 
	$I_0$ using the Koebe principle, which we state below for completeness.
	% we state it below (Lemma \ref{Koebe}).

	\begin{lemma}(Koebe Principle, \cite[Theorem IV.1.2]{dMvS93})\label{Koebe}
		Let $g$ be a $C^3$ diffeomorphism with non-positive Schwarzian derivative. Then for constants $\tau >0$ and $C=C(\tau)>0$. For any subinterval $J_1 \subset J_2$ such that $g(J_2)$ contains a $\tau$-scaled neighbourhood of $g(J_1)$, then
		$$
		\frac{g^\prime(x)}{g^\prime(y)} \le \exp\left(C \frac{|g(x)-g(y)|}{|g(J_1)|}\right) \quad \text{ for all } x, y \in J_1.
		$$
	\end{lemma}
	\begin{remark}\label{rem.koebe}
		The Schwarzian derivative of a $C^3$ diffeomorphism $f$, $\mathbf{S}g(\cdot)$ is given by
		$$
		\mathbf{S}g(x)= \frac{g^{\prime\prime\prime}(x)}{g^{\prime}(x)}-\frac{3}{2}\left(\frac{g^{\prime\prime}(x)}{g^{\prime}(x)}\right)^2.
		$$
		
		Let $U \subset V$ be two intervals, $V$ is said to contain a \emph{$\tau$-scaled neighbourhood} of $U$ if both components of $V\setminus U$ has a length of at least $\tau \cdot |U|$. Where $|U|$ is the length of $U$.
	\end{remark}

	There exists a $\delta >0$ such that the Schwarzian derivative
	%\footnote{See Remark \ref{rem.koebe}for the definition.} 
	of $f_\alpha$ is non-positive for $x$ close to $0$.
	%in the neighbourhood of $\delta$ around the intermittent fixed point. 
	Indeed, this is so since by \ref{s3}, $f_\alpha^{\prime \prime \prime}<0$ close to $0$ and $f_\alpha^\prime>0$. This particularly implies that 
	%o be precise, 
	we can fix a $k$ such that $\mathbf{S}f_\alpha \le 0$ on $[0,z_{k-3}]$ (resp. $[z^\prime_{k-3},1]$). We define $g(\cdot)=f_\alpha^{k_1}(\cdot)$ on $[0,z_{l-3}]$ i.e $g:[0,z_{l-3}] \to [0,z_{k-3}]$. We define $J_1=[z_{l}, z_{l-2}]$, hence $g(J_1)=[z_{k}, z_{k-2}]$. Now, we choose $\beta$ small enough such that $\beta < z_{l}$ and
	%$f_\alpha^{k_1}(\beta)<\frac{z_{k+2}}{2}$
	$g(\beta)<\frac{z_{k}}{2}$. Next, we choose $J_2=[\beta, z_{l-3}]$, with $g(J_2)=[g(\beta),z_{k-3}]$. The Schwarzian derivative of $g$ is non-positive on $J_2$, since the composition of maps with non-positive Schwarzian derivative is also non-positive. Next, we show that $g(J_2)$ contains a $\tau$-scaled neighbourhood of $g(J_1)$.
	%\footnote{See Remark \ref{rem.koebe} for the definition.}.
	%(see Remark \ref{rem.koebe})
	Indeed, if we refer to the left and right components of $g(J_2) \setminus g(J_1)$ as $K_l$ and $K_r$ respectively,
	$$|K_l|\ge \frac{z_{k}}{2} \ge \tau |g(J_1)|= \tau | z_{k-2} - z_{k}|,$$
	taking $\tau \le \frac{z_{k}}{2(| z_{k-2} - z_{k}|)}$.
	$$|K_r|\ge |z_{k-3}-z_{k-2}| \ge \tau | z_{k-2} - z_{k}|,$$
	where $\tau \le \frac{|z_{k-3}-z_{k-2}|}{| z_{k-2} - z_{k}|}$. If we choose $\tau ~\le ~\min \left\{\frac{z_{k}}{2(| z_{k-2} - z_{k}|)}, \frac{|z_{k-3}-z_{k-2}|}{| z_{k-2} - z_{k}|} \right\}$, then by the Koebe principle, there exists $C=C(\tau)>0$ such that
	\begin{align*}
		\frac{(f_\alpha^{k_1})^\prime(x)}{(f_\alpha^{k_1})^\prime(y)} %&\le \exp\left(C \frac{|(f_\alpha^{k_1})(x)-(f_\alpha^{k_1})(y)|}{|(f_\alpha^{k_1})(J_1)|}\right)\\
		&\le \exp\left(M|(f_\alpha^{k_1})(x)-(f_\alpha^{k_1})(y)|\right) \quad \text{ for all } x, y \in J_1,
	\end{align*}
	%with $M\le 2 C$, since $f_\alpha^{l}(J_1) \supset [1/2, 1]$.
	taking $M=\frac{C}{f_\alpha^{k_1}(J_1)}$. Since $K \subset J_1$, we have that for all $x, y \in K$,
	\begin{equation*}
		\frac{(f_\alpha^{k_1})^\prime(x)}{(f_\alpha^{k_1})^\prime(y)} \le \exp\left(M |f_\alpha^{k_1} (K)|\right).
	\end{equation*}
	
	Integrating with respect to $y$ implies that
	$$\mathcal{L}_\alpha^{k_1}\chi_{K}(x) \ge \chi_{f_\alpha^{k_1}(K)}(x)\, 	\frac{|K|}{|f_\alpha^{k_1}(K)|}	\exp\left(-M_1 \left| f_\alpha^{k_1} (K) \right|\right).
	$$
	
	A similar calculation also applies when $x,y \in K \subset (z_{l-2}^\prime, z_{l}^\prime)$. 
	%After going out into $I_0^c$, case 1 occurs until it enters 
	%$K$ enters 
	%the other intermittent region or the case 3 below occurs.
	
	{\bf{Case \boldmath{\ref{DC3}}:}} %Now, we consider the third case.
	Suppose that $K$ contains more than one $z_l$ or $z_l^\prime$ for $l>k$ and more
	%that is when $K\cap I_0$ contains more than one $z_l$ or $z_l^\prime$ for $l>k$. 
	%Suppose 
	than one-third of $K$ is in $I_0^c$, then we consider $K\cap I_0^c$, such that the fixed $k$ is sufficiently large 
	%such that it 
	to contain $z_{k-1}$, which brings us to case \ref{DC1}, such that after a finite number of iterations, is sent to the whole of $S^1$ and ultimately ends the algorithm.
	Otherwise, we split this into sub-cases. To present these sub-cases, we define $l^\prime$ as the least integer such that $[z_{l^\prime+1}, z_{l^\prime}]$ belongs to $K$. The first of the sub-cases we consider is when $|b-z_{l^\prime}|>\frac{|K|}{3}$,
	%supposing that $K\supset L$ such that $L$ contains only one $z_l^\prime$, , $L\subset (z_{l^\prime+1}, b)$, 
	where $b$ is the right end-point of $K$. This then leads us back to case 2. Now, set $K^\prime=[z_{l^\prime}, b]$ such that $|K^\prime| \ge \frac{|K|}{3}$. Since, $K^\prime \subset [z_{l^\prime}, z_{l^\prime -1}]$ and after $l^\prime - k$ iterations, the image of $K^\prime$ will be in the hyperbolic region, we use the estimate from case 2
	%i.e $f_\alpha^{l^\prime-k+1}(K^\prime)\supset [z_k, z_{k-1}]$
	\begin{align*}
		\mathcal{L}^{l^\prime-k}_\alpha\chi_K (x)&\ge \mathcal{L}^{l^\prime-k}_\alpha \chi_{K^\prime} (x) \ge \chi_{f_\alpha^{l^\prime-k}(K^\prime)}(x)\,\frac{|K^\prime|}{|f_\alpha^{l^\prime-k}(K^\prime)|}	\exp\left(-M_1 \left| f_\alpha^{l^\prime-k} (K^\prime) \right|\right)\\
	%	&\ge \chi_{[z_{k}, z_{k-1}]}(x)\,\frac{|K^\prime|}{|f_\alpha^{l^\prime-k}(K^\prime)|}	\exp\left(-M_1 \left| f_\alpha^{l^\prime-k} (K^\prime) \right|\right)\\
		&\ge \frac{1}{3}\chi_{[z_{k}, z_{k-1}]}(x)\,\frac{|K|}{|f_\alpha^{l^\prime-k}(K^\prime)|}	\exp\left(-M_1 \left| f_\alpha^{l^\prime-k} (K^\prime) \right|\right),
	\end{align*}
	we note that $f_\alpha^{k+1}([z_k, z_{k-1}])=S^1$ and we have that $f_\alpha^\prime$ is bounded from above by $N>0$
	\begin{align*}
		\mathcal{L}^{l^\prime+1}_\alpha\chi_K (x) \ge \frac{1}{3 N^{k+1}} \frac{|K|}{|f_\alpha^{l^\prime-k}(K^\prime)|}	\exp\left(-M_1 \left| f_\alpha^{l^\prime-k} (K^\prime) \right|\right).
	\end{align*}
	%\begin{align*}
	%	\mathcal{L}^{l^\prime}_\alpha\chi_K (x)\ge \mathcal{L}^{l^\prime}_\alpha \chi_L (x) \ge \chi_{f_\alpha^{l^\prime}(L)}(x)\,\frac{|L|}{|f_\alpha^{l^\prime}(L)|}	\exp\left(-M_1 \left| f_\alpha^{l^\prime} (L) \right|\right).
	%\end{align*}
	%We split this into sub-cases. The first of the sub-case is if one third of $K$ belongs to $I_0^c$, we consider $K \cap I_0^c$ and (situation 1, 2 or 3) will hold until we cover the whole interval.
	
	Next, suppose that $K=[a, z_{l^\prime}]$, where $a>0$,
	% is  in such a way that 
	% $l\le k$,
	and we choose $K^\prime$ in  
	such a way that $|K^\prime|\ge \frac{|K|}{3}$, $K^\prime \supset \bigsqcup_{l=l^*}^{l^\prime} ~[z_{l+1}, z_{l}]$, and 
	$\left|\bigsqcup_{l=l^*}^{l^\prime} ~[z_{l+1}, z_{l}]\right|\ge \frac{|K^\prime|}{3}~\ge \frac{|K|}{9},$
	taking the minimal number of $z_l$ to make this happen. We therefore estimate $l^*$ as follows, $|[a, z_{l^*-1}]| \ge \frac{2|K^\prime|}{3} \ge \frac{2|K|}{9}$. %From \eqref{eq.zed},
	Since $z _n \approx n^{-1/\alpha}$, we have that $C(l^*-1)^{1/\alpha} \ge z_{l^*-1}\ge z_{l^*-1}-a \ge \frac{2|K|}{9}$, which leads to $l^*=\mathcal{O}(|K|^{-\alpha})$.
	\begin{align*}
		\mathcal{L}_\alpha^l \chi_{[z_{l+1}, z_{l}]}\ge \chi_{f_\alpha^l([z_{l+1}, z_{l}])} (x)\frac{|z_l-z_{l+1}|}{|f_\alpha^l([z_{l+1}, z_{l}])|} \exp \left(-M |f_\alpha^l([z_{l+1}, z_{l}])|\right)
	\end{align*}
	
	Hence, by the computation in case \ref{DC2}, we have that
	\begin{align*}
		\mathcal{L}_\alpha^{l^*+1} \chi_K (x) &\ge \sum_{l=l^*}^{l^\prime} \mathcal{L}_\alpha^{l^*+1} \chi_{[z_{l+1}, z_{l}]}(x) \\
	%	&= \sum_{l=l^*}^{l^\prime}  \mathcal{L}_\alpha^{l^*-l} \mathcal{L}_\alpha^{l+1} \chi_{[z_{l+1}, z_{l}]}(x)\\
	%	&\ge \sum_{l=l^*}^{l^\prime} \mathcal{L}_\alpha^{l^*-l}  \chi_{f_\alpha^{l+1}([z_{l+1}, z_l])}(x)\, 	\frac{(z_l-z_{l+1})}{|f_\alpha^{l+1}([z_{l+1}, z_l])|}	\exp\left(-M_1 \left| f_\alpha^{l+1} ([z_{l+1}, z_l]) \right|\right)\\
	%	& =\sum_{l=l^*}^{l^\prime} \mathcal{L}_\alpha^{l^*-l}  \chi_{[1/2,1]}(x)\, 2(z_l-z_{l+1})\exp\left(-\frac{M_1}{2} \right)\\
		&\ge 2 \mathcal{L}_\alpha^{l^*-l}  \chi_{[1/2,1]}(x) \exp\left(-\frac{M_1}{2}\right) \sum_{l=l^*}^{l^\prime} (z_l-z_{l+1})\\
		&\ge \frac{c_0}{9} |K| \exp\left(-\frac{M_1}{2} \right).
	\end{align*}
	%We note that $f_\alpha^{l+1} ([z_{l+1}, z_l])=[1/2,1]$. 
%	where $\gamma$ is as defined in  \eqref{eq.PFgam}.
	%
	%Assume now that $d\ge 3$, the calculations above work, we need only control what happens on $\{I_2, \cdots, I_{d-1}\}$. If $K$ 
	%%starts out in any $I_i$, 
	%is such that $f_\alpha(K)$ contains at least one $I_i$, then this ends the algorithm, since
	%\begin{align*}
	%	\mathcal{L}_\alpha \chi_{f^m_\alpha(K)}(x)\ge \mathcal{L}_\alpha \chi_{I_i}(x)\ge \chi_{S^1}(x) \inf_{y\in I_i} \frac{1}{f_\alpha^\prime(y)}
	%\end{align*}  
	%and since it is in the hyperbolic region, the derivative is bounded from below, hence we can control the above bound.
	
	Let $J$ be as defined starting out from any part of $S^1$, we associate to $J$ a sequence of integers $n_1, m_1, n_2, m_2, \cdots, n_p$, such that iterating it $n_1$ times, we are in $I_0^c$ (if $J$ starts out from $I_0^c$, then $n_1=0$) and hence, satisfies case 1. Then after $m_1$ iterations, it is in case $2$. Taking $n_2$ iterations, we leave $I_0$ and are back in the hyperbolic region and so on, until we fall into case 3 (for $d=2$) or the iterates contains at least one $I_i$ (for $d\ge 3$). These two situations lead to the end of the algorithm. However, we only focus on the situation where it leads to case 3, such that $[z_{l^\prime},b] < \frac{|K|}{3}$.
	
	For $n \ge n_1+m_1+ \cdots + n_p+l^*+1$, we have that
	\begin{align*}
		&\mathcal{L}_\alpha^n \chi_J(x) \ge \mathcal{L}_\alpha^{n-(n_1+m_1+ \cdots + n_p+l^*+1)}  \mathcal{L}_\alpha^{l^*+1}\mathcal{L}_\alpha^{n_p}\cdots\mathcal{L}_\alpha^{m_1}\mathcal{L}_\alpha^{n_1}  \chi_J(x)\\
		&\ge \mathcal{L}_\alpha^{n-(n_1+m_1+ \cdots + n_p+l^*+1)}  \mathcal{L}_\alpha^{l^*+1}\mathcal{L}_\alpha^{n_p}\cdots\mathcal{L}_\alpha^{m_1}\chi_{f_\alpha^{{n_1}}(J)}\, 	\frac{|J|}{|f_\alpha^{n_1}(J)|}	\exp\left(-N_1 \left| f_\alpha^{n_1} (J) \right|\right)\\
		&\ge \mathcal{L}_\alpha^{n-(n_1+m_1+ \cdots + n_p+l^*+1)}  \mathcal{L}_\alpha^{l^*+1}\mathcal{L}_\alpha^{n_p}\cdots\chi_{f_\alpha^{{n_1+m_1}}(J)}\,\frac{|f_\alpha^{n_1}(J)|}{|f_\alpha^{n_1+m_1}(J)|}	\frac{|J|}{|f_\alpha^{n_1}(J)|}\\	&\exp\left(-M_1|f_\alpha^{n_1+m_1}(J)|-N_1 \left| f_\alpha^{n_1} (J) \right|\right)\\
		&\ge (\mathcal{L}_\alpha^{n-(n_1+m_1+ \cdots + n_p+l^*+1)}\chi) \frac{c_0}{9}|f_\alpha^{n_1+m_1+\cdots+n_p}(J)| \frac{|f_\alpha^{n_1+\cdots+m_{p-1}}(J)|}{|f_\alpha^{n_1+m_1+\cdots+m_{p-1}+n_p}(J)|}\cdots\frac{|f_\alpha^{n_1}(J)|}{|f_\alpha^{n_1+m_1}(J)|}\\	
		&\frac{|J|}{|f_\alpha^{n_1}(J)|}\exp\left(-M_1|f_\alpha^{n_1+m_1+\cdots+n_{p}}(J)|-\cdots-M_1|f_\alpha^{n_1+m_1}(J)|-N_1 \left| f_\alpha^{n_1} (J) \right|-\frac{M_1}{2}\right)\\
%				&\ge 
%		%	(\mathcal{L}_\alpha^{n-(n_1+m_1+ \cdots + n_p+l^*+1)}\chi)
%		\frac{\gamma^2}{9}|J |\exp\left(-M_1|f_\alpha^{n_1+m_1+\cdots+n_{p}}(J)|-\cdots-M_1|f_\alpha^{n_1+m_1}(J)|-N_1 \left| f_\alpha^{n_1} (J) \right|-\frac{M_1}{2}\right)\\
%			%here, we used telescoping canclation.\\
		&\ge \frac{c_0^2}{9}|J |\exp\left(-2\max\{M_1, N_1\}(\lambda^{n_1}+\lambda^{n_1+n_2}+ \cdots+ \lambda^{n_1+n_2+\cdots+n_p})\right)\\
		 &\ge \frac{c_0^2}{9}|J |\exp\left(\frac{-2\max\{M_1, N_1\} \lambda}{1-\lambda}\right)
			%\\
				%&
				=:{\gamma}|J|.
	\end{align*}
%	\begin{align*}
%
%	\end{align*}
	
	Furthermore, we have that $n_1+m_1+ \cdots + n_p+l^*+1=\mathcal{O}(\varepsilon^{-\alpha})$. Observe that $n_1+m_1+ \cdots + n_p \le n_\varepsilon=~\mathcal{O}(\varepsilon^{-\alpha})$ and hence has not covered the whole of $S^1$. As seen in case $3$, we showed that $l^*=\mathcal{O}(|K|^{-\alpha})=~\mathcal{O}(\varepsilon^{-\alpha})$. We claim that $n_\varepsilon=\mathcal{O}(\varepsilon^{-\alpha})$, to prove the claim, let us go through possible scenarios that the dynamics may take. If after iteration, a given scenario coincides with a previous one, we use the estimate of the previous scenario, even though the image of $J$ is bigger this time.% We ~assume that the length of $J$ is at least $2\varepsilon$. 
	 We notice that the worst scenario is when case 3 happens at $0$ (resp. $1$), and $J=(-2\varepsilon/3, 4\varepsilon/3)$, from which setting $n_\varepsilon= \mathcal{O}(\varepsilon^{-\alpha})$ is large enough, and the result is proven.
\end{proof}
Using the previous results, we prove that the random perturbed transfer operator decays at an exponential rate. 
\begin{proposition} \label{Pert.proposition}For $\varphi \in L^1$, with $\int_\Omega \varphi(x) \,dx =0$, we have that
	$$\Vert \mathbf{P}^k_\varepsilon \varphi \Vert_1 \le(1-\gamma)^k \Vert \varphi \Vert_1, \quad \text{ for all }k \in \mathbb{N}.$$
\end{proposition}
\begin{proof}
	From \eqref{eq:average} and \eqref{eq:pertbd}, $\mathbf{P}_\varepsilon 1= \mathcal{L}_\alpha^{n_\varepsilon}1={\mathcal{L}_\alpha^*}^{n_\varepsilon}1=\mathbf{P}_\varepsilon 1=1$. Now, set $\Omega= S^1$, and define $\Omega_0=\{x \in \Omega : \varphi(x) \ge 0\}, \Omega_1=\{x \in \Omega : \mathbf{P}_\varepsilon\varphi(x) \ge 0\}$. We observe that
	$$\int_\Omega\vert \mathbf{P}_\varepsilon \varphi(x)\vert \,dx = 2\int_{\Omega_1} \mathbf{P}_\varepsilon \varphi(x)\, dx.$$
	
	We use the bound in \eqref{eq:kay} and \eqref{eq:peekay} to get the following estimate
	\begin{align}
		\Vert \mathbf{P}_\varepsilon \varphi \Vert_1&=
		 %\int \vert \mathbf{P}_\varepsilon \varphi \vert \,dx= 
		 2 \int_{\Omega_1} \left(\int_\Omega \mathcal{K}_\varepsilon(x,y) \varphi(y)\, dy\right)\,dx\nonumber\\
		&= 2 \int_{\Omega_1} \left(\int_\Omega \mathcal{K}_\varepsilon(x,y) \varphi(y)\, dy\right)\,dx -2\Omega_1 \gamma \int_\Omega \varphi(x) \,dx, \quad \left(\Leftarrow \int_\Omega \varphi(x)\, dx=0\right)\nonumber\\
		%&=2\int_\Omega\left( \int_{\Omega_1} (\mathcal{K}_\varepsilon(x,y)-\gamma)\,dx\right) \varphi(y) \,dy\nonumber\\
		&\le 2\int_\Omega\left( \int_{\Omega} (\mathcal{K}_\varepsilon(x,y)-\gamma)\,dx\right) \varphi(y) \,dy\nonumber\\
		%&\le 2\int_{\Omega_0}\left( \int_\Omega (\mathcal{K}_\varepsilon(x,y)-\gamma)\,dx\right) \varphi(y)\, dy\nonumber\\
		%&=2\int_{\Omega_0} (\mathbf{P}_\varepsilon 1-\gamma) \varphi(y) \,dy\nonumber\\
		&= 2\int_{\Omega_0} (1-\gamma) \varphi(y) \,dy\nonumber\\
		&=(1-\gamma) \Vert \varphi \Vert_1. \nonumber
	\end{align}
	
	Iterating the above estimate, we get the result.
%	\begin{equation}\label{eq:k_pert}
%		\Vert \mathbf{P}^k_\varepsilon \varphi \Vert_1 \le(1-\gamma)^k \Vert \varphi \Vert_1, \quad \forall k \in \mathbb{N}.
%	\end{equation}
\end{proof}

The success of the mechanism depends on the decay estimate (with respect the Lebesgue measure) of the system under consideration. 
%So, in what follows in this section, we
%Here, present decay claims. 
The estimates in Lemma \ref{propkay} and Proposition \ref{Pert.proposition} will be particularly useful in proving the following result.
\begin{lemma} \label{Qdecay}
	For $C_1>0$, $\varphi \in \mathcal{C}_{*,1}+\mathbb{R}$, $\int \varphi\, dx=0$ and $\psi \in L^\infty(m)$, 
	\begin{equation*}
		\bigg\vert \int \psi \mathcal{L}_{\varrho}^n \varphi \,dx\bigg\vert \le C_1 \Vert \varphi \Vert_1 \Vert \psi \Vert_{\infty} n^{1-1/\varrho}(\log n)^{1/\varrho}.
	\end{equation*}
\end{lemma}
\begin{proof}
	For each $n=kn_\varepsilon + j$ with $k \in \mathbb{N}$, $j < n_\varepsilon$, in order to get the required estimate, we decompose the Perron-Frobenius operator as follows  
	\begin{align}\label{eq.decom}
		\bigg\vert \int \psi \mathcal{L}_{\varrho}^n \varphi\, dx\bigg\vert \le \Vert \psi \Vert_{\infty} \left(\Vert \mathcal{L}_{\varrho}^n \varphi - \mathbf{P}_\varepsilon ^k \mathcal{L}_\varrho^j \varphi \Vert_1 + \Vert \mathbf{P}_\varepsilon ^k \mathcal{L}_\varrho^j \varphi\Vert_1\right).
	\end{align}
	%Using the fact that $\Vert \mathbf{P}_\varepsilon f \Vert_1 \le (1-\gamma) \Vert f \Vert_1$.
	
	By Proposition \ref{Pert.proposition},
	\begin{align}\label{eq.pfirst}
		\Vert \mathbf{P}_\varepsilon ^k \mathcal{L}_\varrho^j \varphi\Vert_1 \le (1-\gamma)^k \Vert \mathcal{L}_\varrho^j \varphi\Vert_1 \le  (1-\gamma)^k \Vert \varphi\Vert_1 \le \exp (-\gamma k) \Vert \varphi \Vert_1.
	\end{align}
	
	From Lemma \ref{radom.pert}, we have that
	\begin{align}\label{eq.psecond}
		\Vert \mathcal{L}_{\varrho}^n \varphi - \mathbf{P}_\varepsilon ^k \mathcal{L}_\varrho^j \varphi \Vert_1 &\le \sum_{i=0}^{k-1} \left\Vert \mathcal{L}_{\varrho}^{(i+1)n_\varepsilon} \mathcal{L}_{\varrho}^j\varphi - \mathbf{P}_\varepsilon ^k \mathcal{L}_{\varrho}^{in_\varepsilon}\mathcal{L}_\varrho^j \varphi \right\Vert_1 \nonumber\\
%		%	&\le C \sum_{i=0}^{k-1} \left\Vert \mathcal{L}_{\varrho}^{(in_\varepsilon + j)} \varphi \right\Vert_1 \varepsilon^{1-\alpha}\nonumber\\
%		& \le C k \Vert \varphi \Vert_1 \varepsilon^{1-\varrho} \nonumber\\
		& \le C \Vert \varphi \Vert_1 \frac{n}{n_\varepsilon} \varepsilon^{1-\varrho}.
	\end{align}
	
	From \eqref{eq.pfirst} and \eqref{eq.psecond}, we obtain the following estimate for \eqref{eq.decom}
	\begin{align}
		\left\vert \int \psi \mathcal{L}_{\varrho}^n \varphi\, dx\right\vert &\le C \Vert \psi \Vert_\infty \left(C \Vert \varphi \Vert_1 \frac{n}{n_\varepsilon} \varepsilon^{1-\varrho} + \exp (-\gamma k) \Vert \varphi \Vert_1 \right) \nonumber\\
	%	&\le C \Vert \psi \Vert_\infty \Vert \varphi \Vert_1  \left(C \frac{n}{n_\varepsilon} \varepsilon^{1-\varrho} + \exp \left[-\gamma \left(\frac{n}{n_\varepsilon}-\frac{j}{n_\varepsilon} \right)\right] \right)\nonumber\\
		&\le C \Vert \psi \Vert_\infty \Vert \varphi \Vert_1  \left(C \frac{n}{n_\varepsilon} \varepsilon^{1-\varrho} + \exp(\gamma) \, \exp \left(-\gamma\,\frac{n}{n_\varepsilon} \right) \right)\nonumber\\
		&\le C \Vert \psi \Vert_\infty \Vert \varphi \Vert_1 n^{1-1/\varrho} (\log n)^{1/\varrho},
	\end{align}
	provided we take $\varepsilon=C_{\gamma,\varrho}n^{-1/\varrho}(\log n)^{1/\varrho}$.
\end{proof}
%\begin{remark}
%	The last estimate and the value of $\varepsilon$, basically depends on the choice the choice of $n_\varepsilon$ in Proposition \ref{propkay} (Remove this remark later on).\todo{choose a suitable $\varepsilon$}
%\end{remark}
%We have almost everything we need to show that the defined cone is invariant, we only need to show that for this class of maps,
%   \begin{equation*}
	%      \int_0^x P_f\psi(x) \leq a x^{\alpha}
	% \end{equation*}
%\begin{lemma}
%   The cone $C_a$ is invariant for the family of circle maps, provided $a$ is chosen big enough.
%\end{lemma}
%\begin{proof}
%   L
%\end{proof}
%
%in the spirit of \cite{LSV98}, the cone
%\begin{equation}
%   \mathcal{C}_{*}(\alpha, a_1, b_1)=\bigg\{\varphi \in C^{(1)}((0,1)) \bigg\vert 0\le \varphi(x) \le 2h_{\alpha}(x)\int_{0}^{1} \varphi ; \vert \varphi^\prime (x) \vert \le \frac{a_1+b_1\vert x \vert }{\vert x \vert} \varphi(x) \bigg\}
%\end{equation}
%was shown to be invariant with respect to the Perron-Frobenius operator, for $a_1> \alpha$ and $b_1>0$ \cite[Lemma 5.1]{LSV98}. For completeness, we show the proof here.
%\begin{lemma}
%   There exists $a_1>\alpha$ and $b_1>0$, such that $\mathcal{N}_{\alpha,i} \mathcal{C}_* \subset \mathcal{C}_*$.\todo{Maybe we need more assumptions.}
%\end{lemma}
%\begin{proof}
%w
%\end{proof}

Define the cone 
\begin{equation*}
	\mathcal{C}_0=\left\{\varphi \in C^{0}(S^1\setminus\{0\})\big\vert 
	% 0\le \varphi(x) \le 2h_\alpha(x)\int_{S^1} \varphi \,dx 
	\varphi \ge 0 
	\text{ and } \varphi \text{ is decreasing}\right\},
\end{equation*}
it is easy to check that $\mathcal{C}_0$ is invariant with respect to $\mathcal{L}_\alpha$.
%the Perron-Frobenius operator \eqref{eq:PF}.
%By the definition of $\mathcal{C}_{0}$,
Let $\kappa=\frac{1}{d}$ as defined, then
%\begin{equation}\label{eq.czero}
%    \frac{d-1}{d}\int_{0}^{\kappa} \varphi\, dx + \frac{1}{d}\int_{1-\kappa}^{1} \varphi \,dx \ge \kappa m(\varphi), \quad \forall \varphi \in \mathcal{C}_0.
%\end{equation}
\begin{equation}\label{eq.czero}
	(1-\kappa)\int_{0}^{\kappa} \varphi\, dx + \kappa\int_{1-\kappa}^{1} \varphi \,dx \ge \kappa m(\varphi), \quad \forall \varphi \in \mathcal{C}_0.
\end{equation}

Indeed, 
\begin{align*}
	\kappa m(\varphi)=\kappa \int_{S^1} \varphi(x)\,dx = \kappa \int_{0}^{\kappa} \varphi(x) \,dx +\kappa \sum_{i=1}^{d-2} \int_{l_i}^{l_{i+1}} \varphi(x)\, dx + \kappa \int_{1-\kappa}^1 \varphi(x)\, dx, \quad l_i=\frac{i}{d},
\end{align*}
since $\varphi(x) \ge 0$ and decreasing, we have that $\kappa \sum_{i=1}^{d-2} \int_{l_i}^{l_{i+1}} \varphi(x) \,dx \le  \kappa (d-2)\int_{0}^{\kappa}\varphi(x)\, dx $. Hence, our claim.
%\begin{align*}
%	\kappa m(\varphi) \le (1-\kappa)\int_{0}^{\kappa}\varphi(x)+ \kappa \int_{1-\kappa}^1 \varphi(x) dx.
%\end{align*}
\begin{remark} We have equality in \eqref{eq.czero} when $d=2$. $h_\alpha \in \mathcal{C}_0 \cap \mathcal{C} \cap \mathcal{C}_{*,1}$, for $a_1,b_1$ large enough. In addition, $h_\alpha$ is 
	locally 
	Lipschitz. 
	%\todo{Check the parameters or maybe remove the density part.}  
\end{remark}
\begin{proposition} \label{prop.zerodecay}
	For $\alpha \in (0,1)$, $a_1$ and $b_1$ big enough. Then
	% $$\mathcal{L}_{\alpha}(\mathcal{C}_{*,1}(\alpha,a,b_1)) \subset \mathcal{C}_{*,1}(\alpha,a,b_1),$$
	$$\mathcal{N}_{\alpha,i}\left(\mathcal{C}_{*,1}(\alpha,1,a_1,b_1) \cap \bigg\{ (1-\kappa)\int_{0}^{\kappa} \varphi \,dx + \kappa\int_{1-\kappa}^{1} \varphi \,dx \ge \kappa m(\varphi)\bigg\}\right) \subset \mathcal{C}_{*,1}(\alpha,(d-1),a_1,b_1),$$
	$d$ the number of branches and $\kappa=\frac{1}{d}$. Furthermore, for any $\psi \in L^\infty(m)$ and $\varphi \in \mathcal{C}_{*,1}(\alpha) + \mathbb{R}$, with zero average, there exists $C >0$ independent of $\alpha, a_1, b_1$, such that
	\begin{equation*}
		\left|\int_{0}^{1} \psi \mathcal{L}_{0}^k(\varphi)\, dx\right| \le \frac{Cab_1}{(1-\beta)(\log k)k^{-2+1/\beta}} \Vert \psi \Vert_\infty \Vert \varphi \Vert_1, \quad \forall k \ge 1, \beta \in (0,1).
	\end{equation*}
\end{proposition}
\begin{proof}
	%  $\mathcal{L}_{\alpha}(\mathcal{C}_{*,1}(\alpha,a,b_1)) \subset \mathcal{C}_{*,1}(\alpha,a,b_1)$ follows immediately from the fact that $m(\mathcal{L}_{\alpha}(\varphi))=m(\varphi)$ and $\mathcal{L}_{\alpha}(h_\alpha)=h_\alpha$.
	\begin{align*}
		(1-\kappa)m(\mathcal{N}_{\alpha,i}\varphi(x))
		%&= (1-\kappa)\int_{S^1}\mathcal{N}_{\alpha,i}\varphi(x) dx\\
		&=(1-\kappa)\int_{0}^{\kappa} \varphi \,dx + (1-\kappa)\int_{1-\kappa}^{1} \varphi \,dx\\
		%&\ge (1-\kappa)\int_{0}^{\kappa} \varphi \,dx + \kappa \int_{1-\kappa}^{1} \varphi \,dx\\
		&\ge \kappa m(\varphi) \quad \text{(using \eqref{eq.czero})}.
	\end{align*}
	
	Hence, $m(\varphi) \le (d-1) m(\mathcal{N}_{\alpha,i}\varphi(x))$, we remark that when $d=2$ this is an equality and we are back to \eqref{eq.emel}. By Lemma \ref{inv.cone} and the fact that $\mathcal{N}_{\alpha,i} h_\alpha \le h_\alpha$, the invariance of $\mathcal{N}_{\alpha,i}$ follows.
	Next, we show the decay of correlations at $\alpha=0$.
	%we note that at $\alpha=0$, $f_{0}(x) \approx 2x$
	
	We fix $\beta$ for any $\beta \in (0,1)$. Recall the inclusion in \eqref{eq.incbeta} for a parameter say $\varrho=0$, we then have from Lemma \ref{radom.pert}, for $\varphi \in \mathcal{C}_{*,1}(\beta)$ that
	\begin{equation*}
		\Vert \mathcal{L}^{n_\varepsilon}_{0} (\id - \mathbf{A}_{\varepsilon}) \varphi \Vert_{1} \le\frac{18 c_2 \max\{a_1,b_1,1\}}{\beta(1-\beta)} \Vert \varphi \Vert_1 \varepsilon^{1-\beta}.
	\end{equation*}
	we may take $n_\varepsilon=\frac{|\log \varepsilon|}{\log 2}$ in the proof of Proposition \ref{propkay}.
	% , using the \ref{C2} property of the Perron-Frobenius operator, we have that
	From Lemma \ref{Qdecay}, taking $\varepsilon=n^{-1/\alpha}$, we get the result.
	%\begin{align*}
	%	\left|\int_{0}^{1} \psi \mathcal{L}_{0}^k(\varphi)\, dx\right| \le 
	%\end{align*}
	
	%\begin{remark}
	%	The proof for Proposition \ref{propkay} when $\alpha>0$ is a little bit different, since here we use easier bounds for $\varepsilon$ and $n_\varepsilon$.
	%\end{remark}
\end{proof}
\subsection{Important estimates}\label{Gouezel}
From Young \cite[Theorem 5]{YL99}, for  H\"older observables $\varphi$ on the circle map, with density $h_\alpha$ and  $\int \varphi \,dm=1$, then $\int \vert \mathcal{L}_{\alpha}^n(\varphi)-h_\alpha \vert \,dm \approx n^{1-1/\alpha}.$ If in addition $\psi \in L^{\infty}(m)$, then for $\alpha=0$ it has a sharp decay of correlation of $\mathcal{O}(\theta^n)$, $\theta<1$ depending only on the H\"older exponents of the observables.

%\begin{theorem}\cite[Theorem 5]{YL99} \label{Young.circ}
%	Let $\mathcal{L}_\alpha$ be the Perron-Frobenius operator associated with $f_\alpha$ the circle map with parameter $\alpha \in (0,1)$, and $h_\alpha$ its density, then for all H\"older continuous function ~$\varphi:S^1 \to \R$,
%	% \in \mathcal{H}$
%	$\psi \in L^{\infty}(m)$
%	%  \begin{equation*}
%		%     \bigg\vert \int (\psi \circ f_{\alpha}^{n})\varphi d\mu- \int \varphi d\mu \int \psi d\mu \bigg\vert = \mathcal{O}(n^{1-1/\alpha}).
%		%\end{equation*}
%		with $\int \varphi \,dm=1$,
%		% then
%		\begin{equation}
%			\int \vert \mathcal{L}_{\alpha}^n(\varphi)-h_\alpha \vert \,dm \approx n^{1-1/\alpha}. \label{eq:Yzero}
%		\end{equation}
%		For $\alpha=0$
%		\begin{equation}
%			\bigg\vert \int (\psi \circ f_{0}^{n})\varphi \,d\mu- \int \varphi \,d\mu \int \psi \,d\mu \bigg\vert \le C \theta^n, \label{eq:YyoungZero}
%		\end{equation}
%		$\theta < 1$ depending only on the H\"older exponents of the observables.
%	\end{theorem}
%	%For $\alpha=0$, the observables decays exponentially.
	For $\alpha \in (0,1)$, $f_\alpha$ has a unique SRB measure $\mu_\alpha$ and has a sharp decay of correlation for  H\"older observables of $\mathcal{O}\left({n^{1-1/\alpha}}\right)$ \cite{JFA20, YL99}.
%	\begin{theorem}\cite[Theorem 3.62]{JFA20} \label{Alves2}
%		Let $f_\alpha$ be the circle map with $\alpha \in (0,1)$. Then $f_\alpha$ has a unique SRB measure $\mu_\alpha$. Moreover, $\mu_\alpha$ is exact, equivalent to $m$, its basin covers $m$ almost all of $S^1$ and for every H\"older continuous function $\varphi:S^1 \to \mathbb{R}$ and $\psi \in L^\infty(m)$
%		
%		%with mean zero, and bounded $\psi$ , which cancel each other in the vicinity of $0$, we have
%		\begin{equation}
%			\bigg\vert \int (\psi \circ f_{\alpha}^{n})\varphi d\mu- \int \varphi d\mu \int \psi d\mu \bigg\vert \lesssim \bigg(\frac{1}{n^{1/\alpha-1}}\bigg) \label{eq:docL}
%		\end{equation}
%		%   \begin{equation}
%			%      \int \varphi \cdot \psi \circ f_{\alpha}^n dx= \mathcal{O}\bigg(\frac{1}{n^{1/\alpha}}\bigg) \label{eq:docL}
%			% \end{equation}
%	\end{theorem}
Gou\"ezel imposing some extra conditions ($\varphi$ a H\"older function with zero mean, $\psi$ bounded and vanishes in the neighbourhood of $0$) on the observables, showed a much sharper decay of correlation of order $\mathcal{O}\left({n^{-1/\alpha}}\right)$ \cite[Corollary 2.4.6]{gouezel2004vitesse}.
%	\begin{proposition} \cite[Corollary 2.4.6]{gouezel2004vitesse} \label{Gouezeldecay}
%		%    Let $f_\alpha$ be the LSV map with parameter $\alpha \in (0,1)$. Then for all 
%		%   % functions 
%		%    H\"older function $\varphi$  with mean zero, and $\psi$ a bounded function which cancel each other in the vicinity of $0$, we have
%		For all 
%		H\"older function $\varphi$  with mean zero, $\alpha \in (0,1)$ and $\psi$ a bounded function which cancel each other in the vicinity of $0$, we have
%		\begin{equation*}
%			\int \varphi \cdot \psi \circ f_{\alpha}^n \,dx= \mathcal{O}\bigg(\frac{1}{n^{1/\alpha}}\bigg).
%		\end{equation*}
%	\end{proposition}
%	%Furthermore, for $\alpha \in (0,1)$, $(f_\alpha,\mu_\alpha)$ is mixing \cite[Theorem 5b]{YL99} and decays with a polynomial rate for H\"older observables.~ 
%	
	For the particular mechanism we shall deploy, the rate of decay given by Gou\"ezel in \cite[Corollary 2.4.6]{gouezel2004vitesse} shall play a pivotal role when $1/2\le \alpha < 1$. Next, if we assume in addition that $\varphi(0)=0$, satisfying $\vert \varphi(x) \vert \le C x^{\gamma}$, for a certain $\gamma >0$. Then by \cite[Theorem 2.4.14]{gouezel2004vitesse},
	$\Vert \mathcal{L}^n \varphi\Vert_1= \mathcal{O}\left(\frac{1}{n^{\min\{\lambda, \lambda(1+\gamma)-1\}}}\right),$ \label{Gouezel}
	where $\lambda=\frac{1}{\alpha}$.
%	\begin{theorem}\cite[Theorem 2.4.14]{gouezel2004vitesse} \label{Gouezel1}
%		%	Let $f_\alpha$ be the LSV map with the parameter $\alpha \in (0,1)$, and $\lambda=\frac{1}{\alpha}$. Let $\varphi$ be a zero average
%		%	%{by zero average, we mean that $\int \varphi d\mu_\alpha=0$} 
%		%	H\"{o}lder function with $\varphi(0)=0$, satisfying $\vert \varphi(x) \vert \le C x^{\gamma}$, for a certain $\gamma >0$. Then
%		%	$$\Vert \mathcal{L}^n \varphi\Vert_1= \mathcal{O}\bigg(\frac{1}{n^{\min\{\lambda, \lambda(1+\gamma)-1\}}}\bigg).$$ \label{Gouezel}
%		%Let $f_\alpha$ be the LSV map with the parameter $\alpha \in (0,1)$, and . 
%		Let $f_\alpha$ be the intermittent circle map with the parameter $\alpha \in (0,1)$ and $\varphi$ be a zero average
%		%{by zero average, we mean that $\int \varphi d\mu_\alpha=0$} 
%		H\"{o}lder function with $\varphi(0)=0$, satisfying $\vert \varphi(x) \vert \le C x^{\gamma}$, for a certain $\gamma >0$. Then
%		$$\Vert \mathcal{L}^n \varphi\Vert_1= \mathcal{O}\bigg(\frac{1}{n^{\min\{\lambda, \lambda(1+\gamma)-1\}}}\bigg),$$ \label{Gouezel}
%		where $\lambda=\frac{1}{\alpha}$.
%	\end{theorem}
	\begin{remark}
		Although the above results in \cite{gouezel2004vitesse} were stated for the LSV map, the theorems are written in the general setting of the Young tower and applies to the circle map with indifferent fixed points we are considering.
	\end{remark}
	% i.e for $\psi \in L^{\infty}$ and $\varphi \in Lip$, with mean zero, which cancel each other in the vicinity of $0$, we have
	%%     \begin{equation}
		% %       \bigg\vert \int (\psi \circ f_{\alpha}^{n})\varphi d\mu- \int \varphi d\mu \int \psi d\mu \bigg\vert \lesssim \bigg(\frac{1}{n^{1/\alpha-1}}\bigg) \label{eq:docL}
		% %   \end{equation}
	%    \begin{equation}
		%        \int \varphi \cdot \psi \circ f_{\alpha}^n d\mu_\alpha= \mathcal{O}\bigg(\frac{1}{n^{1/\alpha}}\bigg) 
		%  %\label{eq:docL}
		%    \end{equation}
	%    A sharper decay of correlation result.
	%\end{comment}
	\subsection{Some properties of the transfer operator}\,
	For $\alpha \mapsto \mathcal{L}_{\alpha} \varphi (x)$, and $\varphi: S^1 \to \R$ sufficiently regular, we give some properties of $\partial_{\alpha}\mathcal{L}_{\alpha}$ that will be particularly useful going forward.
	\begin{lemma}
		For $\alpha \in (0,1)$, $\alpha \mapsto g_{\alpha,i}(y)$, and for all $x \in S^1\setminus\{0\}$,
		
		\begin{equation}
			\partial_{\alpha}g_{\alpha,i}(x)=- \frac{X_{\alpha,i}(x)}{f_{\alpha}^{\prime}(g_{\alpha,i}(x))}, \quad i\in\{1, d\}; \label{eq:par g}
		\end{equation}
		\begin{equation}
			\partial_{\alpha}g_{\alpha,i}^{\prime}(x)=- \frac{X_{\alpha,i}^{\prime}(x)}{f_{\alpha,i}^{\prime}(g_{\alpha,i}(x))} + X_{\alpha,i}(x) \frac{f_{\alpha,i}^{\prime \prime}(g_{\alpha,i}(x))}{(f_{\alpha,i}^{\prime}(g_{\alpha,i}(x))^3}, \quad i\in\{1, d\}.\label{eq:par g'}
		\end{equation}
	\end{lemma}
	\begin{proof}
		Since for $x \in S^1 \setminus \{0\}$,
		$f_{\alpha,i}(g_{\alpha,i}(x))=x$. Then by the chain rule and
%		, we have that
%		$$
%		(\partial_\alpha g_{\alpha,i}(x))\, f_\alpha^\prime (g_{\alpha,i}(x)) + \partial_\alpha f_\alpha (g_{\alpha,i}(x)) =0,
%		$$
		using \eqref{eq.eqeks}, gives \eqref{eq:par g}.
%		Using the chain rule, we have that
%		\begin{equation}
%			g_{\alpha,i}^{\prime}(x)=\frac{1}{f_{\alpha,i}^{\prime}(g_{\alpha,i}(x))}. \label{eq:impl}
%		\end{equation}
%		
%		Now,
Next, by the inverse function theorem applied to $g_{\alpha+\varepsilon,i}^{\prime}(x)-g_{\alpha,i}^{\prime}(x)$,
%, we have that
%		\begin{align}
%			g_{\alpha+\varepsilon,i}^{\prime}(x)-g_{\alpha,i}^{\prime}(x)
%			%=\frac{1}{f_{\alpha+\varepsilon,i}^{\prime}(g_{\alpha+\varepsilon,i}(x))}-\frac{1}{f_{\alpha,i}^{\prime}(g_{\alpha,i}(x))}\nonumber\\
%			=\frac{f_{\alpha,i}^{\prime}(g_{\alpha,i}(x))-f_{\alpha+\varepsilon,i}^{\prime}(g_{\alpha+\varepsilon,i}(x))}{f_{\alpha+\varepsilon,i}^{\prime}(g_{\alpha+\varepsilon,i}(x))\cdot f_{\alpha,i}^{\prime}(g_{\alpha,i}(x))} \label{eq:diff}
%		\end{align}
then simplifying
% equation \eqref{eq:diff}, we 
 by first multiplying by $\frac{f_{\alpha+\varepsilon,i}^{\prime}(g_{\alpha,i}(x))}{f_{\alpha+\varepsilon,i}^{\prime}(g_{\alpha,i}(x))}$, then add and subtract $f_{\alpha+\varepsilon,i}^{\prime}(g_{\alpha+\varepsilon,i}(x))\cdot f_{\alpha,i}^{\prime}(g_{\alpha,i}(x))$ to the numerator, we get
 
 	\begin{align*}
 	%&=\frac{f_{\alpha+\varepsilon,i}^{\prime}(g_{\alpha,i}(x))\cdot f_{\alpha,i}^{\prime}(g_{\alpha,i}(x))-f_{\alpha+\varepsilon,i}^{\prime}(g_{\alpha+\varepsilon,i}(x))\cdot f_{\alpha,i}^{\prime}(g_{\alpha,i}(x))+f_{\alpha+\varepsilon,i}^{\prime}(g_{\alpha+\varepsilon,i}(x))\cdot f_{\alpha,i}^{\prime}(g_{\alpha,i}(x))-f_{\alpha+\varepsilon,i}^{\prime}(g_{\alpha,i}(x))\cdot f_{\alpha+\varepsilon,i}^{\prime}(g_{\alpha+\varepsilon,i}(x))}{f_{\alpha+\varepsilon,i}^{\prime}(g_{\alpha,i}(x))\cdot f_{\alpha+\varepsilon,i}^{\prime}(g_{\alpha+\varepsilon,i}(x))\cdot f_{\alpha,i}^{\prime}(g_{\alpha,i}(x))}\nonumber\\
 	g_{\alpha+\varepsilon,i}^{\prime}(x)-g_{\alpha,i}^{\prime}(x) =\underbrace{\frac{f_{\alpha,i}^{\prime}(g_{\alpha,i}(x))-f_{\alpha+\varepsilon,i}^{\prime}(g_{\alpha,i}(x))}{f_{\alpha+\varepsilon,i}^{\prime}(g_{\alpha+\varepsilon,i}(x))\cdot f_{\alpha+\varepsilon,i}^{\prime}(g_{\alpha,i}(x))}}_\text{(I)}+\underbrace{\frac{f_{\alpha+\varepsilon,i}^{\prime}(g_{\alpha,i}(x))-f_{\alpha+\varepsilon,i}^{\prime}(g_{\alpha+\varepsilon,i}(x))}{f_{\alpha+\varepsilon,i}^{\prime}(g_{\alpha,i}(x))\cdot f_{\alpha+\varepsilon,i}^{\prime}(g_{\alpha+\varepsilon,i}(x))}.}_\text{(II)}
 \end{align*}
 We simplify (I) and (II) above by recalling the 	
 Taylor's formula of $f_{\alpha+\varepsilon, i}^{\prime}(g_{\alpha+\varepsilon, i}(x))$ and $f_{\alpha+\varepsilon,i}^{\prime}(y)$,	
%		
%		Now, using $\partial_{\alpha}g_{\alpha,i}^{\prime}(x)= \lim_{\varepsilon \to 0} \frac{g_{\alpha+\varepsilon,i}^{\prime}(x)-g_{\alpha,i}^{\prime}(x)}{\varepsilon}$.
%		To simplify (I) and (II) above, 
%		
%		we simplify by recalling that by Taylor's formula, we have that
%		\begin{equation}
%			f_{\alpha+\varepsilon,i}^{\prime}(y)=f_{\alpha,i}^{\prime}(y)+ \varepsilon \cdot \partial_{\alpha} f_{\alpha,i}^{\prime}(y) + \mathcal{O}(\varepsilon^2), \label{eq:tay1}
%		\end{equation}
		%hence, (I) simplifies to
%		\begin{equation*}
%			\text{(I)}=\frac{-\varepsilon \cdot \partial_{\alpha} f_{\alpha,i}^{\prime}(g_{\alpha,i}(x)) - \mathcal{O}(\varepsilon^2)}{f_{\alpha+\varepsilon,i}^{\prime}(g_{\alpha+\varepsilon,i}(x))\cdot f_{\alpha,i}^{\prime}(g_{\alpha,i}(x))}
%		\end{equation*}
then taking the limits as $\varepsilon \to 0$
%		\begin{equation*}
%			\lim_{\varepsilon \to 0} \frac{\text{(I)}}{\varepsilon}= \frac{-\partial_{\alpha} f_{\alpha,i}^{\prime}(g_{\alpha,i}(x))}{(f_{\alpha,i}^{\prime}(g_{\alpha,i}(x)))^2}.
%		\end{equation*}
%Since $\alpha\mapsto f_{\alpha,i} \in C^2$, 
		%we can make the assumptions by equation \eqref{eq:commutation}.
		therefore by the definition in \eqref{eq.eqeks} and \eqref{eq:par g}, we get the required conclusion.
	\end{proof}
	
	\begin{lemma}\label{part_PF}
		For $\varphi \in C^1(S^1\setminus\{0\})$, then for all $x\in S^1\setminus\{0\}$, $\alpha\in (0,1)$, and $i \in \{1,d\}$,
		\begin{align}
			\partial_{\alpha}\mathcal{L}_{\alpha}\varphi(x)
			%&=-X_{\alpha,i}^{\prime}\mathcal{N}_{\alpha,i}\varphi(x)-X_{\alpha,i}\mathcal{N}_{\alpha,i}(\varphi^{\prime}/f_{\alpha,i}^{\prime})(x)+X_{\alpha,i}\mathcal{N}_{\alpha,i}(\varphi f_{\alpha,i}^{\prime \prime}/(f_{\alpha,i}^{\prime})^2)(x) \nonumber\\
			=- \sum_{i \in \{1,d\}}(X_{\alpha,i}\,\mathcal{N}_{\alpha,i}\varphi)^{\prime}(x). \label{eq:part_Pf}
		\end{align}
		In particular, $m(\partial_{\alpha}\mathcal{L}_{\alpha}\varphi(x))=0$. 
	\end{lemma}
	\begin{proof} See \cite[Lemma 4.2]{L24}.
	\end{proof}
	\begin{lemma} \label{partwo}
		For $\alpha \mapsto X_{\alpha,i}\mathcal{N}_{\alpha,i}\varphi(x) \in C^3$, for $i\in \{1,d\}$,
		\begin{equation}
			\partial_{\alpha}^{2}\mathcal{L}_{\alpha}\varphi(x)=\sum_{i \in \{1,d\}}\left[-((\partial_{\alpha}X_{\alpha,i})(\mathcal{N}_{\alpha,i}\varphi))^{\prime}(x)+X_{\alpha,i}^{\prime}(X_{\alpha,i}\mathcal{N}_{\alpha,i}\varphi)^{\prime}(x)+X_{\alpha,i}(X_{\alpha,i}\mathcal{N}_{\alpha,i}\varphi)^{\prime \prime}(x)\right]. \label{secondder}
		\end{equation}
		%    Furthermore, $m(\partial_{\alpha}^{2}\mathcal{L}_{\alpha}\varphi(x))=0$.
	\end{lemma}
	\begin{proof}
		% Note that $\alpha\mapsto X_{\alpha,i}\mathcal{N}_{\alpha,i} \varphi(x) \in C^{3}$. 
		From Lemma \ref{part_PF}, we simply take the partial derivative and simplify, which gives us the result.
		By the Leibniz integral rule, $m(\partial_{\alpha}^{2}\mathcal{L}_{\alpha}\varphi(x))=\partial_{\alpha}m(\partial_{\alpha}\mathcal{L}_{\alpha}\varphi(x))=0$, from Lemma \ref{part_PF}.
	\end{proof}
	\begin{remark}
		The use of the Leibniz integral rule above is justified by the bound shown in \eqref{eq:part.sq} that $\vert \partial_t^{2}\mathcal{L}_{t} (\mathcal{L}_{\alpha}^{n-1-j} (\mathbf{1})) \vert < \infty.$
	\end{remark}
	%Next, we give lemmas on some estimates on the upper bounds of $\partial_{\alpha}^{2} \mathcal{L}_{\alpha} \varphi (x)$ and $(\partial_{\alpha}^{2} \mathcal{L}_{\alpha} \varphi)^{\prime}(x)$.
	
	%\begin{lemma}
	%    Suppose that $\alpha \in (0,1)$, $\varphi \in \mathcal{C}(\alpha, b_1, b_2, b_3) \cap \mathcal{C}_*(\alpha, a, b)$...\todo{We shall compute assumptions for this to hold}
	%    \begin{equation*}
		%        \vert \partial_{\alpha}^2 \mathcal{L}_{\alpha} \varphi(x)\vert< \infty \quad \text{ and } \quad \vert (\partial_{\alpha}^2 \mathcal{L}_{\alpha} \varphi)^{\prime}(x)\vert <\infty
		%    \end{equation*}
	%\end{lemma}
	%\begin{proof}
	%    Let
	%\end{proof}
	\section{The linear response formula} 
	This section will be dedicated to proving the linear response formula in Theorem \ref{LR}, firstly for observables $\psi \in L^\infty(m)$, thereafter show how the result can be extended to $\psi \in L^q(m)$. 
	%\begin{proof}[Proof of Theorem \ref{LR}]
	We achieve the proof in three main steps outlined as follows.
	Firstly, we shall show that our claimed linear response formula is well-defined. Subsequently, we will show that the map $\beta \mapsto \int \psi \circ f_\beta^n \,dx$ is locally Lipschitz continuous at $\beta=\alpha \in \left[0,1 \right)$. Finally, we show the linear response formula indeed holds.
	\subsection{Well-defined formula}\label{sec.welldef}
	Let $\psi \in L^\infty(m)$, we show that the right hand side of \eqref{eq:linearresponse} is well-defined. Now, integrating by parts on the circle, we have that
	\begin{equation*}
		\int_{0}^{1}\sum_{i={1,d}}(X_{\alpha,i}\mathcal{N}_{\alpha,i}\varphi)^{\prime} dx=0.
	\end{equation*}
	
	Next, we show that it is Lebesgue integrable, for each $i \in \{1,d\}$ and $\alpha \in [0,1)$,
	\begin{comment}
		A function $f:[a,b] \to \mathbf{R}$ is Lebesgue integrable if the integral of its absolute value is finite. Specifically,
		$\int_a^b \vert f(x) \vert dx < \infty$.
	\end{comment}
	\begin{equation}
		\Vert (X_{\alpha,i}\mathcal{N}_{\alpha,i}(h_\alpha))^{\prime} \Vert_1 = \Vert X_{\alpha,i}(\mathcal{N}_{\alpha,i}(h_\alpha))^{\prime} + X_{\alpha,i}^{\prime}\mathcal{N}_{\alpha,i}(h_\alpha) \Vert_1< \infty.  \label{eq:finiteX}
	\end{equation}
	
	Indeed, 
	%equation \eqref{eq:finiteX} 
	the above
	is true for $\alpha=0$. 
	%Substituting $\alpha=0$ in 
	From \eqref{eq:zerothderivative}, we have that $g_{0,i}(x)\le \hat{C}\,x$, $\hat{C}>0$ and $g_{\alpha,i}^{\prime}(x)\le k$, $k>0$.
	%since we have the following bounds $g_{\alpha,i}(x)\le C\,x$ and $g_{\alpha,i}^{\prime}(x)\le k$, $k$ a constant, equations  \eqref{eq:zzeroth} and \eqref{eq:ffirst}, yields, $ X_{0,i}(x)<~ \infty$ and $X^{\prime}_{0,i}(x) < \infty$, for $x \in S^1\setminus\{0\}$. 
	Since $h_0|_{S^1}=1$ we have from \eqref{eq:PF_branch} that 
	
	$$\mathcal{N}_{0,i}h_0(x) \le k \quad  \text{for } i = {1, d},$$ 
	from the above inequality and \eqref{eq:first},
	\begin{align*}
		\int\left|(X_{0,i}\mathcal{N}_{0,i}(h_0))^{\prime}\right|\,dx&\le \int \left| X_{0,i}^\prime(x)\mathcal{N}_{0,i}h_0(x)\right| \,dx\\
		&\le \hat{C} \int \left( 1+\vert \ln (\vert x \vert) \vert \right)\, dx, \quad \forall x \in S^1,
		%\\
		%&\le k X_{0,i}^\prime(x) \le \hat{C} (1 + \ln{(\vert C x \vert)}), \quad \forall x \in S^1.
	\end{align*}
	%, \hat{C}>0
	%\begin{align}
	%   X_{0,i}(x)=  , \quad \quad X^{\prime}_{0,i}(x)=  , \quad \quad \forall x \in S^1.
	%\end{align}
	whose integral is finite. 
	
	For $0<\alpha <1$, \eqref{eq:dens.bound} gives the following bounds,
	\begin{equation}
		\mathcal{N}_{\alpha,i}(h_\alpha)(x) \le h_\alpha (x) \le c_2 \vert x \vert^{-\alpha}, \quad \quad \vert h^{\prime}_{\alpha} (x) \vert \le c_2 \vert x \vert^{-(1+\alpha)}, \quad c_2>0.  \label{eq:inq}
	\end{equation}
	
	From Lemma \ref{inv.cone}, it implies that
	\begin{equation}\label{eq:enprime}
		\vert (\mathcal{N}_{\alpha,i}(h_\alpha))^\prime (x)\vert \le \left(\frac{a_1}{\vert x \vert} +b_1\right)\mathcal{N}_{\alpha,i}(h_\alpha)(x) \le c_2 (a_1\vert x \vert^{-(1+\alpha)} + b_1\vert x \vert^{-\alpha}), \quad a_1,b_1,c_2>0.
	\end{equation}
	
	Recalling \eqref{eq:zeroth} and \eqref{eq:first}, together with \eqref{eq:inq} and \eqref{eq:enprime}
	\begin{align} \label{eq:An}
		\Vert (&X_{\alpha,i}\mathcal{N}_{\alpha,i}(h_\alpha))^{\prime} \Vert_1 \nonumber\\
		%&= \Vert X_{\alpha,i}(\mathcal{N}_{\alpha,i}(h_\alpha))^{\prime} + X_{\alpha,i}^{\prime}\mathcal{N}_{\alpha,i}(h_\alpha) \Vert_1\\
		&\le \int_0^1 \left[\tilde{C} \vert x \vert^{\alpha + 1} (1+\vert \ln (\vert x \vert)\vert) \cdot c_2 \vert x \vert^{-(1+\alpha)}(a_1 + b_1\vert x \vert)+ \tilde{C}\vert x \vert^\alpha (1 + \vert \ln (\vert x \vert)) \cdot c_2 \vert x \vert^{-\alpha} \right]\,dx \nonumber\\
		%  &\le \int_0^1 \{C\vert x \vert^\alpha (1 + \vert \ln (\vert x \vert) \cdot c_2 \vert x \vert^{-\alpha}+ C \vert x \vert^{\alpha + 1} \vert \ln (\vert x \vert)\vert \cdot c_2 \vert x \vert^{-(1+\alpha)}(a_1 + b_1\vert x \vert) \}dx \\
		&\le   \tilde{C} \int_{0}^{1}\left[1 + (a_1+b_1|x|)\right](1+|\ln (| x |)|)\,dx < \infty.
	\end{align}
	
	By the Neumann series,
	\begin{equation*}
		(\id- \mathcal{L}_{\alpha})^{-1}= \sum_{j=0}^{\infty} \mathcal{L}_{\alpha}^j,
	\end{equation*}
	hence, the right hand side of \eqref{eq:linearresponse} may be written as
	\begin{equation}
		\left \vert \sum_{j=0}^{\infty}  \int_{S^1} \psi \mathcal{L}_{\alpha}^j \left[\sum_{i \in \{1,d\}}(X_{\alpha,i}(\mathcal{N}_{\alpha,i}(h_{\alpha}))^{\prime})\right]dx \right\vert \le  \sum_{j=0}^{\infty} \left \vert  \int_{S^1} \psi \mathcal{L}_{\alpha}^j \left[\sum_{i \in \{1,d\}}(X_{\alpha,i}(\mathcal{N}_{\alpha,i}(h_{\alpha}))^{\prime})\right]dx \right\vert. \label{eq:summa}
	\end{equation}
	%$$\displaystyle \sum_{j=0}^{\infty}  \int_{S^1} \psi \mathcal{L}_{\alpha}^j \bigg[\sum_{i \in \{1,d\}}(X_{\alpha,i}(\mathcal{N}_{\alpha}(h_{\alpha}))^{\prime})\bigg]dx$$
	
	Our next task is to show that this series is absolutely convergent for $\alpha \in [0,1)$. We achieve this in two parts. Firstly for $\alpha \in (0,1)$, then for $\alpha=0$. For $\alpha \in (0,1)$,
	%\begin{equation}
	%   \bigg \vert \sum_{j=0}^{\infty}  \int_{S^1} \psi \mathcal{L}_{\alpha}^j \bigg[\sum_{i \in \{1,d\}}(X_{\alpha,i}(\mathcal{N}_{\alpha}(h_{\alpha}))^{\prime})\bigg]dx \bigg\vert \le  \sum_{j=0}^{\infty} \bigg \vert  \int_{S^1} \psi \mathcal{L}_{\alpha}^j \bigg[\sum_{i \in \{1,d\}}(X_{\alpha,i}(\mathcal{N}_{\alpha}(h_{\alpha}))^{\prime})\bigg]dx \bigg\vert. \label{eq:summa}
	%\end{equation}
	we check the hypothesis of Gou\"ezel's theorem in section \ref{Gouezel}. We define the  function
	$$ F_\alpha(x)=\sum_{i \in \{1,d\}}\frac{(X_{\alpha,i}\mathcal{N}_{\alpha,i}(h_\alpha))^{\prime}(x)}{h_{\alpha}(x)},$$
	the bounds for $h_\alpha$ in the inequality \eqref{eq:dens.bound} gives that $F_\alpha(0)=0$. By Lemma \ref{part_PF} we have that
	%From equation \eqref{eq:bpzero}, 
	$\int F_\alpha h_\alpha dx=0$.
	%possesses a zero mean
	We also claim that $F_\alpha$ is  H\"older. Indeed, 
	\begin{equation}\label{eq.eqef}
		F_\alpha(x)=\sum_{i \in \{1,d\}}X_{\alpha,i}(x)\frac{(\mathcal{N}_{\alpha,i}(h_\alpha))^{\prime}(x)}{h_{\alpha}(x)} +\sum_{i \in \{1,d\}}X_{\alpha,i}^\prime(x) \frac{\mathcal{N}_{\alpha,i}(h_\alpha)(x)}{h_{\alpha}(x)},
	\end{equation}
	$X_{\alpha,i}^\prime $ 
	%(resp. $X_{\alpha,i}$) 
	being H\"older, follows from the bound on %(resp.~\eqref{eq:first}) that
	$X_{\alpha,i}^{\prime\prime}$
	%(resp. $X_{\alpha,i}^\prime $) is bounded.
	in \eqref{eq:ssecond}.
	% Since the estimates for $X_{\alpha,i}^\prime $ and $X_{\alpha,i}$ in equations \eqref{eq:ffirst} and \eqref{eq:zzeroth} are independent of $i$. 
	We simplify the second sum above using \eqref{eq:PF} and \eqref{eq:PF_branch},
	%$$\mathcal{L}_{\alpha}\varphi(x)=\sum_{y \in f_{{\alpha},i}^{-1}(x), 2\le i \le d-1}\frac{\varphi(y)}{f_{\alpha}^{\prime}(y)}+ \sum_{i \in \{1,d\}}\mathcal{N}_{\alpha,i} \varphi(x)$$
	\begin{align*}
		\sum_{i \in \{1,d\}}X_{\alpha,i}^\prime(x) \frac{\mathcal{N}_{\alpha,i}(h_\alpha)(x)}{h_{\alpha}(x)} &\le \max_i X_{\alpha,i}^\prime(x)\sum_{i \in \{1,d\}}\frac{(\mathcal{N}_{\alpha,i}(h_\alpha))(x)}{h_{\alpha}(x)}\\
		&=\max_i X_{\alpha,i}^\prime(x)-\max_i X_{\alpha,i}^\prime(x)\sum_{\begin{smallmatrix}
				y \in f_{\alpha,i}^{-1}(x) \\
				2 \le i \le d-1
		\end{smallmatrix}}\frac{h_\alpha(y)}{h_\alpha(x) f_{\alpha}^{\prime}(y)}.
	\end{align*}
	
	We have used the fact that $\mathcal{L}_\alpha h_\alpha=h_\alpha$ in the last equation. $h_\alpha(y)$ is locally  Lipschitz and $\displaystyle \frac{1}{h_\alpha(x)}$ is  H\"older, with $f_\alpha ^\prime$ bounded.
	% function from \eqref{eq:firstderivative}. 
	%Lipschitz continuous. 
	The product and sum are bounded functions. Hence, is a bounded  H\"older function. Similarly, from \eqref{eq:enprime}, \eqref{eq:zeroth} and \eqref{eq:dens.bound} we have that the first sum in \eqref{eq.eqef} is also  H\"older.
%	of bounded Lipschitz functions is Lipschitz.
	
	We observe from \eqref{eq:zeroth}, \eqref{eq:first}, \eqref{eq:inq} and \eqref{eq:enprime} that for any 
	%fixed
	$\varepsilon >0$, 
	% $\varepsilon \in (0,1)$ 
	there exists some $M_\varepsilon>0$ such that for $x \in S^1\setminus\{0\}$,
	%Because it vanishes at $0$
	\begin{align*}
		\left\vert F_\alpha(x) \right\vert &\le  \vert x \vert^\alpha \left[1 + (a_1+b_1|x|)\right](1+|\ln (| x |)|)\\
		& \le M_\varepsilon \vert x \vert^{\alpha(1- \varepsilon/2)}.
	\end{align*}
	%$\vert F_{\alpha,i}(x) \vert \le $. 
	
	From the theorem of Gou\"ezel in section \ref{Gouezel}, taking $\gamma= \alpha(1- \varepsilon/2)>0$, then 
	$$\min \left\{\frac{1}{\alpha} , \frac{1}{\alpha} (1+ \gamma)-1\right\}= \frac{1}{\alpha}  (1+ \gamma)-1> 1/\alpha-\varepsilon.$$
	%\eqref{eq:summa} is given by
	
	By the duality of the Perron-Frobenius operator, we have that
	\begin{align}
		\sum_{j=0}^{\infty} \left \vert  \int_{S^1} \psi \mathcal{L}_{\alpha}^j \left[\sum_{i \in \{1,d\}}(X_{\alpha,i}\mathcal{N}_{\alpha,i}(h_{\alpha}))^{\prime}\right]dx \right\vert&= \sum_{j=0}^{\infty} \left \vert  \int_{S^1} \psi \mathcal{L}_{\alpha}^j \left(F_\alpha \cdot h_\alpha\right)\,dx \right\vert\nonumber\\
		%&= \sum_{j=0}^{\infty} \left \vert  \int_{S^1}  \psi \mathcal{L}_{\alpha}^j F_\alpha\,d\mu_\alpha \right\vert\nonumber\\
		&\le \Vert \psi \Vert_{\infty} \sum_{j=0}^{\infty}  \int_{S^1} \left\vert \mathcal{L}_{\alpha}^j F_\alpha\right\vert \,d\mu_\alpha\nonumber\\
		&\le C K_\alpha \Vert \psi \Vert_{\infty}\frac{1}{j^{(1/\alpha)-\varepsilon}} \label{eq.sum0}.
	\end{align}
	
	This series is only summable only when $\varepsilon < \frac{1}{\alpha}-1$. 
	%The results for the bound on the Perron-Frobenius operator from Theorem \ref{Gouezel} is not applicable for $\alpha=0$.
	
	Now, for $\alpha=0$, applying Proposition \ref{prop.zerodecay} with $\beta \in (0,1)$ fixed, there is a constant $C_q> 0$ such that  to $(X_{0,i}\mathcal{N}_{0,i}(h_0))^{\prime}+ C_q \in \mathcal{C}_{*,1}(\beta, 1, a, b_1)$. For 
	$$\varphi= (X_{0,i}\mathcal{N}_{0,i}(h_0))^{\prime}  \le \tilde{C}(1 + |\ln{(\vert x \vert)|}) \in \mathcal{C}_{*,1}+\mathbb{R}, $$ 
	%\quad \Vert \varphi \Vert_1 < \infty. we have shown this already
	%from \eqref{eq:finiteX}
	\begin{equation*}
		\left\vert \int_{S^1} \psi \mathcal{L}_{0}^j\left[\sum_{i \in \{1,d\}}(X_{0,i}\mathcal{N}_{0,i}(h_0))^{\prime}\right] dx\right\vert \le \frac{Cab_1}{(1-\beta)(\log j)j^{-2+1/\beta}} \Vert \psi \Vert_\infty, \quad \forall j \ge 1,
	\end{equation*}
	%This is summable for $\beta < 1/3$.\todo{Check this estimate again} 
	hence, the claimed linear response formula is well-defined.
	\subsection{Local Lipschitz continuity} \label{sec.lip}
	%% of $\beta \mapsto \int \psi \circ f_\beta^n \,dx$}
%We remark that every Lipschitz function $\varphi$ can be decomposed as 
%\begin{equation}\label{eq.Decom}
%	\varphi=c_\varphi \mathbf{1} + \hat{\varphi},
%\end{equation}
%
%where $c_\varphi = \int \varphi\,dx$ and $\int \hat{\varphi}\, dx=0$.
%By the definition of decay of correlation \eqref{eq.DOC} and the decomposition \eqref{eq.Decom},
%\begin{align}\label{eq.eqcoref}
%	\text{Cor}(\varphi, \psi \circ f^n)&=\left\vert \int \varphi(\psi\circ f^n) \,dx -\int \varphi \,dx \int \psi\, dx \right\vert\nonumber\\
%	&= \left\vert \int \psi \mathcal{L}^n_\alpha\varphi \,dx -\int \varphi \,dx \int \psi\, dx \right\vert\nonumber\\
%	&= \left\vert \int \psi \mathcal{L}^n_\alpha(c_\varphi \mathbf{1} + \hat{\varphi}) \,dx -c_\varphi \int \psi\, dx \right\vert\nonumber\\
%	&= \left\vert \int \psi \mathcal{L}^n_\alpha \hat{\varphi} \,dx\right\vert.
%\end{align}
%
%The last equality uses the fact that $\mathcal{L}_\alpha^n \mathbf{1}=\mathbf{1}$.
Next, we assume without loss of generality that for $\psi \in L^\infty(m)$, $\int \psi d\mu_{\alpha}=0$. Our aim is to show that $\beta \mapsto \int \psi \circ f_\beta^n \,dx$ is Lipschitz continuous at $\beta = \alpha$. 
%From \eqref{eq.Decom}, 
We may write the zero average function $F_\varrho(\cdot)$ as 
$$F_\varrho= h_\varrho -\mathbf{1} \in \mathcal{C}_{*,1} + \R,\quad \varrho>0,$$
then apply the rate of decay result in Lemma \ref{Qdecay} to get an estimate for the decay of correlation
%for \eqref{eq.eqcoref}.
%then by iterating \ref{C1} and using \ref{C3}, 
\begin{align}
	\cor(F_\varrho, \psi\circ f_\alpha^n)&=\left| \int F_\varrho(\psi\circ f_\alpha^n) \,dx -\int F_\varrho \,dx\int \psi\circ f_\alpha^n \,dx \right|\nonumber\\
%	&=\bigg\vert \int \psi \mathcal{L}_{\varrho}^n F_\varrho \,dx \bigg\vert \quad\quad \quad \quad\quad \quad \quad\quad \quad  \left(\Leftarrow \int F_\varrho \,dx=0\right) \nonumber\\
%	% &=  \bigg\vert \int \psi \mathcal{L}_{\varrho}^n (h_\varrho -\mathbf{1}) dx \bigg\vert  
%	&=  \bigg\vert \int \psi \mathcal{L}_{\varrho}^n h_\varrho \,dx- \int \psi \mathcal{L}_{\varrho}^n\mathbf{1} \,dx \bigg\vert \nonumber\\
	&=  \bigg\vert \int \psi \,d\mu_\varrho - \int \mathbf{1}(\psi \circ f_{\varrho}^n) \,dx \bigg\vert \label{eq.lipcor}.
\end{align}

In the last equality, we use the fact that $\mathcal{L}^n_\varrho h_\varrho= h_\varrho$ and the duality property of the Perron-Frobenius operator. %Now, we 

Let $\varrho=\alpha > 0$ 
%and $\int_{S^1} F_$\alpha$ dx=0$
, then from Lemma \ref{Qdecay}, we have that
\begin{equation}
	\bigg\vert \int \mathbf{1}(\psi \circ f_{\alpha}^n) \,dx \bigg\vert \le C_\alpha \Vert \psi \Vert_\infty \frac{(\log n)^{1/\alpha}}{n^{1/\alpha-1}}.   \label{eq:pert_alpha}
\end{equation}

By \cite[Theorem 5]{YL99}, the unperturbed system, $\varrho=\alpha=0$ has the following estimate
%The spectral gap\todo{Why and how did baladi get this estimate?} of $\mathcal{L}_0$ on $C^1$, there exists a constant $C\ge 1$ such that
\begin{equation}
	\bigg\vert \int \mathbf{1}(\psi \circ f_{0}^n)\, dx \bigg\vert \le C {\theta^n}, \quad \theta<1.   \label{eq:unpert_alpha}
\end{equation}
%\frac{\Vert \psi \Vert_\infty}{2^n}.   \label

Suppose that $\varrho$ is any $\beta > 0$, \eqref{eq.lipcor} becomes
\begin{equation}
	\bigg\vert \int \psi \mathcal{L}_{\beta}^n F_\beta \,dx \bigg\vert=\bigg\vert \int \psi \,d\mu_\beta - \int \mathbf{1}(\psi \circ f_{\beta}^n) \,dx \bigg\vert\le C_\beta \Vert \psi \Vert_\infty \frac{(\log n)^{1/\beta}}{n^{1/\beta-1}}. \label{eq:pert_beta}
\end{equation} 

Choosing $n$ large enough, depending on $\alpha$ and $\beta$. The inequalities \eqref{eq:pert_alpha}, \eqref{eq:unpert_alpha}, \eqref{eq:pert_beta} are $\smallO(\beta - \alpha)$.
\begin{comment}
	The Little oh notation
	$f(x)=\smallO(g(x)), \quad x \to a$, $a \in \R, \infty, -\infty$
	
	Tells us that $g(x)$ grows much faster than $f(x)$, as $x \to a$. Symbolically,
	$$\lim_{x \to a} \left| \frac{f(x)}{g(x)} \right|=0$$
\end{comment}
That is, fixing $\xi >0$, there is $C>0$ such that for all
\begin{equation}
	\quad  n(\alpha, \beta, \xi)=: n > C\left(C_{\max\{\alpha, \beta\}}(\beta - \alpha)^{-(1+\xi)}\right)^{1/(-1+1/\max\{\alpha, \beta\})}, \label{eq:en}
\end{equation}
%C_{\max\{\alpha, \beta\}}
we have that
\begin{equation}
	\bigg\vert \int(\psi \circ f_{\alpha}^n)\,dx\bigg\vert + \bigg\vert \int \psi \,d\mu_\beta - \int (\psi \circ f_{\beta}^n)\,dx  \bigg\vert \le C (\beta - \alpha)^{1+\xi}. \label{eq:bound}
\end{equation}

What we want to ultimately show is that
\begin{equation*}
	\lim_{\beta \to \alpha} \frac{1}{\beta-\alpha} \left[\left(\int \psi \,d\mu_\beta - \int \psi \,d\mu_\alpha \right) - \left(\int (\psi \circ f_{\beta}^n)\,dx- \int(\psi \circ f_{\alpha}^n)\,dx \right) \right] = 0.
\end{equation*}

We have that the term in the square bracket is bounded by \eqref{eq:bound}. 
%our ultimate goal is to show that
%$$\frac{1}{\beta-\alpha} \bigg(\int (\psi \circ f_{\beta}^n)dx- \int(\psi \circ f_{\alpha}^n)dx \bigg)$ gives us the linear response formula in the end.
Suppose that $n:= n(\alpha,\beta, \xi)$, and let $\mathbf{1}$ be the constant function $\equiv 1$, we have by using the duality property of the Perron-Frobenius operator
\begin{align}
	\frac{1}{\beta-\alpha} \bigg(\int (\psi \circ f_{\beta}^n)\,dx- \int(\psi \circ f_{\alpha}^n)\,dx \bigg) 
	%&=  \frac{1}{\beta-\alpha} \int_{0}^{1} \psi (\mathcal{L}_{\beta}^n \mathbf{1}-\mathcal{L}_{\alpha}^n \mathbf{1})\,dx \nonumber\\
%	&=\frac{1}{\beta-\alpha} \int_{0}^{1} \psi \sum_{j=0}^{n-1}\mathcal{L}_{\beta}^{j}(\mathcal{L}_{\beta}-\mathcal{L}_{\alpha})\mathcal{L}_{\alpha}^{n-1-j} (\mathbf{1})\,dx \nonumber\\
	&= \sum_{j=0}^{n-1} \int_{0}^{1} \psi \mathcal{L}_{\beta}^{j}\bigg(\frac{(\mathcal{L}_{\beta}-\mathcal{L}_{\alpha})}{\beta-\alpha}(\mathcal{L}_{\alpha}^{n-1-j} (\mathbf{1}))\bigg)\,dx, \label{teles}
\end{align}
notice that we used the telescoping sum that for every $\alpha, \beta, n$,
$$\mathcal{L}_{\beta}^n \mathbf{1}-\mathcal{L}_{\alpha}^n \mathbf{1}=
\sum_{j=0}^{n-1}\mathcal{L}_{\beta}^{j}(\mathcal{L}_{\beta}-\mathcal{L}_{\alpha})\mathcal{L}_{\alpha}^{n-1-j} (\mathbf{1}).
$$

%Hence, by the above equation we have that

%\subsubsection{theorem:\,}

Taylor's formula gives that for any $\varphi \in C^{2}\left(S^1\setminus\{0\}\right), \beta \neq \alpha, x \neq 0$ 
\begin{equation}
	\frac{(\mathcal{L}_{\beta}-\mathcal{L}_{\alpha})\varphi(x)}{\beta-\alpha}= \partial_{\alpha}\mathcal{L}_{\alpha} \varphi (x)+ \frac{1}{\beta-\alpha} \int_{\alpha}^{\beta}(\beta-t) \partial_{t}^2 \mathcal{L}_{t} \varphi (x) \,dt. \label{tay_beta-PF}
\end{equation} 

Assuming that $\beta > \alpha > 0$, $n \ge 1$ in \eqref{teles}, Lemma \ref{part_PF} and Lemma~ \ref{partwo} gives that
\begin{align}
	\sum_{j=0}^{n-1} \int_{0}^{1} \psi \mathcal{L}_{\beta}^{j}\bigg(\frac{(\mathcal{L}_{\beta}-\mathcal{L}_{\alpha})}{\beta-\alpha}(\mathcal{L}_{\alpha}^{n-1-j} (\mathbf{1}))\bigg)dx&= \sum_{j=0}^{n-1} \int_{0}^{1} \psi \mathcal{L}_{\beta}^{j}\bigg(\partial_{\alpha}\mathcal{L}_{\alpha} (\mathcal{L}_{\alpha}^{n-1-j} (\mathbf{1})) \nonumber
	\\&+ \frac{1}{\beta-\alpha} \int_{\alpha}^{\beta}(\beta-t) \partial_{t}^2 \mathcal{L}_{t} (\mathcal{L}_{\alpha}^{n-1-j} (\mathbf{1}))\,dt\bigg)\,dx \nonumber
\end{align}
\begin{align}
	\sum_{j=0}^{n-1} \int_{0}^{1} \psi \mathcal{L}_{\beta}^{j}&\left(\frac{(\mathcal{L}_{\beta}-\mathcal{L}_{\alpha})}{\beta-\alpha}(\mathcal{L}_{\alpha}^{n-1-j} (\mathbf{1}))\right)dx=\nonumber\\ &\quad\quad\quad\underbrace{-\sum_{j=0}^{n-1} \int_{0}^{1} \psi \mathcal{L}_{\beta}^{j}\left(\sum_{i \in \{1,d\}}(X_{\alpha,i}\mathcal{N}_{\alpha,i}\mathcal{L}_{\alpha}^{n-1-j} (\mathbf{1}))^{\prime}\right)dx}_{A_n} \nonumber
	\\ &\quad\quad\quad\quad\quad\quad\quad+\underbrace{\int_{\alpha}^{\beta} \frac{\beta-t}{\beta-\alpha}  \sum_{j=0}^{n-1} \int_{0}^{1} \psi \mathcal{L}_{\beta}^{j} \left( \partial_{t}^2 \mathcal{L}_{t} (\mathcal{L}_{\alpha}^{n-1-j} (\mathbf{1}))\right)\,dx\,dt}_{B_n}.
\end{align}

\textbf{Case 1}\,: We use the theorem of Gou\"ezel's in section \ref{Gouezel} to show that $A_n$ and $B_n$ are summable, for $0< \alpha < 1$.
\subsubsection{Summability of $A_n$\,:}$\,$ Checking the assumptions of \cite[Theorem 2.4.14]{gouezel2004vitesse} (see the end of section \ref{Gouezel}), allows us to give conclusions about the summability of this series. Using integration by parts $(X_{\alpha,i}\mathcal{N}_{\alpha,i}\mathcal{L}_{\alpha}^{n-1-j} (\mathbf{1}))^{\prime}$ is a zero average function. Next, $(X_{\alpha,i}\mathcal{N}_{\alpha,i}\mathcal{L}_{\alpha}^{n-1-j} (\mathbf{1}))^{\prime}$ is bounded. Finally, to check that $\displaystyle H_\alpha(x)=~ \sum_{i \in \{1,d\}}\frac{(X_{\alpha,i}\mathcal{N}_{\alpha,i}(\mathcal{L}_{\alpha}^{n-1-j} (\mathbf{1})))^{\prime}}{h_\alpha}$ is H\"older, we notice from the calculations in \eqref{eq:An} that the numerator is bounded. Next, using the same bounds, and \eqref{eq:dens.bound}, $|H^\prime_\alpha(x)| \le C x^{\alpha -1}$.
%it is sufficient to check that there exists an $M\ge 0$ such that
%\begin{align*}
%	|H^\prime_\alpha(x)| \le C x^{\alpha -1}
%	%&= \left|\left(\frac{(X_{\alpha,i}\mathcal{N}_{\alpha,i}(\mathcal{L}_{\alpha}^{n-1-j} (\mathbf{1})))^{\prime}(x)}{h_\alpha(x)}\right)^{\prime}\right|\\
%%	\le \sum_{i \in \{1,d\}} \left|\frac{h_\alpha(x)(X_{\alpha,i}\mathcal{N}_{\alpha,i}(\mathcal{L}_{\alpha}^{n-1-j} (\mathbf{1})))^{\prime\prime}(x)-(X_{\alpha,i}\mathcal{N}_{\alpha,i}(\mathcal{L}_{\alpha}^{n-1-j} (\mathbf{1})))^{\prime}(x)h_\alpha^\prime(x)}{h_\alpha^2(x)}\right|\le M.
%\end{align*} 
%Indeed, by \eqref{eq:dens.bound} and the calculations below, this is bounded. 
Also, $H_\alpha(0)=0$,
hence, same as in the calculations of \eqref{eq.sum0}, and the fact that we assumed that $0<\alpha<\beta<1$, we have that

\begin{equation*}
	\sum_{j=0}^{n-1} \int_{S^1} \bigg\vert \psi \mathcal{L}_{\beta}^j\bigg(\sum_{i \in \{1,d\}} (X_{\alpha,i}\mathcal{N}_{\alpha,i}(\mathcal{L}_{\alpha}^{n-1-j} (\mathbf{1})))^{\prime}\bigg)\bigg\vert \,dx \le C_{\beta} \Vert \psi \Vert_{\infty} \sum_{j=0}^{n-1}\frac{1}{j^{1/\beta - \varepsilon}},
\end{equation*}
which is summable as $n \to \infty$.
\subsubsection{Summability of $B_n$\,:\,} Now, we have that from \eqref{secondder}, let $\varphi=\mathcal{L}_{\alpha}^{n-1-j}(\mathbf{1}) \in \mathcal{C}\cap \mathcal{C}_{*,1}$, by the invariance of the cone, for any $\alpha\le t \le \beta$, $\displaystyle \mathcal{L}_{t} (\mathcal{L}_{\alpha}^{n-1-j} (\mathbf{1})) \in \mathcal{C}$.

{\it{Claim}:} $\vert \partial_t^{2}\mathcal{L}_{t} (\mathcal{L}_{\alpha}^{n-1-j} (\mathbf{1})) \vert < \infty.$\\

{\it{Proof of claim:}}  We find the bounds on $\partial_{\alpha}X_{\alpha,i}(x)$ and $\partial_{\alpha}X^{\prime}_{\alpha,i}(x)$. Using the chain rule, we note from \eqref{eq.eqeks} that,
\begin{equation*}
	\partial_{\alpha}X_{\alpha,i}(x)= (\partial_{\alpha}g_{\alpha,i}(x))\,(v_{\alpha}^{\prime}\circ g_{\alpha,i}(x))+\partial^2_{\alpha}f_\alpha \circ g_{\alpha,i}(x).
\end{equation*}
% Now, from \eqref{eq:zzeroth}, \eqref{eq:impl}, \eqref{eq:partX} and \eqref{eq:par g}, note that for $i =1, d$,
% \begin{equation*}
	%    \partial_{\alpha}X_{\alpha,i}(x)= -X_{\alpha,i}(x) \cdot X^\prime_{\alpha.i}(x)+\partial^2_{\alpha}(f_\alpha) \circ g_{\alpha,i}
	%\end{equation*}
	
	By \eqref{eq:firstderivative} and \eqref{eq.eqvee},
	\begin{align}\label{eq.eqpartialf}
		\begin{split}
			\partial_\alpha f_\alpha^\prime(x) &\approx |x|^\alpha \ln(|x|),\\
			% \partial_{\alpha} f_{\alpha}(x) &\approx \sgn(x)\,\vert x \vert^{\alpha+1} \ln(|x|) \\
			\partial_{\alpha}^2 f_{\alpha}(x) &\approx \ln(\vert x \vert)  \partial_{\alpha} f_{\alpha}(x).
			%\partial_{\alpha}(\vert x \vert^{\alpha+1})
		\end{split} 
	\end{align}
	% with $\partial^{2}_{\alpha}f_{\alpha}\circ g_{\alpha,i}(x)=0$. Hence, using \eqref{eq:par g},
	% \begin{align*}
		%     \partial_{\alpha}X_{\alpha,i}(x)= -X_{\alpha,i}(x)\frac{\partial_{\alpha}(f_{\alpha})^{\prime}\circ g_{\alpha}(x)}{f_{\alpha}^{\prime}\circ g_{\alpha}(x)}= -X_{\alpha,i}(x) \cdot \partial_{\alpha} \ln (f_{\alpha}^{\prime}\circ g_{\alpha}(x))
		% \end{align*}

	%\begin{remark}
	%For each of the branches, the partial derivatives exists for the inverse branches, and satisfies the commutation relations
	%\begin{equation}
	%		\partial_{\alpha}g_{\alpha,i}^{\prime}\approx (\partial_{\alpha}g_{\alpha,i})^{\prime} \qand
	%		\partial_{\alpha}f_{\alpha,i}^{\prime}\approx (\partial_{\alpha}f_{\alpha,i})^{\prime}.
	%\label{eq:commutation}
	%\end{equation}
	%\end{remark}
	
	Hence, from \eqref{eq:par g}, \eqref{eq.eqpartialf}, \eqref{eq.eqeks}
	\begin{equation}\label{eq.eqpartialx}
		\partial_{\alpha}X_{\alpha,i}(x)\approx \left(- \frac{1}{f_{\alpha}^{\prime}(g_{\alpha,i}(x))} |g_{\alpha,i}(x)|^\alpha  + 1\ \right) \ln(|g_{\alpha,i}(x)|) X_{\alpha,i}(x).
	\end{equation}
	
	Using the estimate $g_{\alpha,i}(x)\le Cx$, \eqref{eq:zeroth} and  \ref{s2},
	\begin{equation}\label{eq:partialeks}
		\vert \partial_{\alpha}X_{\alpha,i}(x) \vert \le \tilde{C} \vert x \vert^{\alpha+1} (1+\vert \ln (\vert x \vert)\vert)^2,
	\end{equation}
	next, differentiate \eqref{eq.eqpartialx} with respect to $x$,  to get the following bounds
	\begin{equation}\label{eq:partialeksprime}
		\vert \partial_{\alpha}X^{\prime}_{\alpha,i}(x) \vert \le C \vert x \vert^{\alpha} (1+\vert \ln (\vert x \vert)\vert)^2.
	\end{equation}
	%    \begin{align*}
		%        \vert \partial_\alpha X_{\alpha,i} ^{\prime \prime} (x)\vert \leq
		%    \end{align*}
	
	Now, we see that since $\varphi \in \mathcal{C}_{*,1}(\alpha) \cap \mathcal{C}$, Lemma \ref{inv.cone} implies that
	\begin{align*}
		\vert (\mathcal{N}_{\alpha,i}\varphi)^{\prime \prime}(x)\vert &\le \bigg(\frac{a_2}{x^2}+b_2\bigg)\mathcal{N}_{\alpha,i}\varphi(x)\le \bigg(\frac{a_2}{x^2}+b_2\bigg)\cdot \frac{2c_2}{\vert x \vert^\alpha} m(\varphi) \le 2c_2\vert x \vert^{-(\alpha+2)}(a_2+ b_2 x^2).
		% &\\
		%   \vert (X_{\alpha,i} \mathcal{N}_{\alpha,i}(\varphi))^{\prime \prime}(x)\vert &\le \frac{2ab_2 c_2 C \vert x \vert^{\alpha+1} \vert \ln (\vert x \vert) \vert}{\vert x \vert^{\alpha+2}} m(\varphi) \le \frac{C \vert \ln (\vert x \vert) \vert}{\vert x \vert} m(\varphi).
	\end{align*}
	
	From Lemma \ref{partwo}, we have that
	\begin{align*}
		\partial_{t}^{2}\mathcal{L}_{t}\varphi(x)&=\sum_{i \in \{1,d\}}[\underbrace{-((\partial_{t}X_{t,i})(\mathcal{N}_{t,i}\varphi))^{\prime}(x)}_\text{(I)}+\underbrace{X_{t,i}^{\prime}(X_{t,i}\mathcal{N}_{t,i}\varphi)^{\prime}(x)}_\text{(II)}+\underbrace{X_{t,i}(X_{t,i}\mathcal{N}_{t,i}\varphi)^{\prime \prime}(x)}_\text{(III)}]
		%\\
		%  &=: (I)+(II)+(III)
	\end{align*}
%	Now, differentiating and simplifying the terms above, we have that
%	\begin{align*}
%		\text{(I)}&=-\left[\partial_{t}X_{t,i}(x)\, (\mathcal{N}_{t,i}\varphi(x))^{\prime}+\partial_{t}X^{\prime}_{t,i}(x) \,(\mathcal{N}_{t,i}\varphi(x))\right],\\
%		\text{(II)}&=(X_{t,i}^{\prime}(x))^2 \, \mathcal{N}_{t,i}\varphi(x) + X_{t,i}^{\prime}(x)\,X_{t,i}(x)\,(\mathcal{N}_{t,i}\varphi(x))^{\prime}, \\
%		\text{(III)}&=X_{t,i}(x)\left[X_{t,i}^{\prime \prime}(x)\, \mathcal{N}_{t,i}\varphi(x)+ 2X^{\prime}_{t,i}(x)\,(\mathcal{N}_{t,i}\varphi(x))^{\prime}+X_{t,i}(x)\, (\mathcal{N}_{t,i}\varphi(x))^{\prime \prime}\right].
%	\end{align*}
%	
%	%  Now, taking $\varphi=\mathcal{L}_{\alpha}^{n-j-1}(\mathbf{1}) \in \mathcal{C}_{*,1}(\alpha) \cap \mathcal{C}$, 
Now,simplifying the terms above, we have that	
	from \eqref{eq:zeroth}, \eqref{eq:first}, \eqref{eq:ssecond}, \eqref{eq:Kone1}, \eqref{inv.cone}, \eqref{eq:partialeks} and \eqref{eq:partialeksprime}, we bound the above as follows
	%we may bound as follows
	%\todo{fix the bounds and check p 12 in LR}
	\begin{align*}
		\vert \text{(I)} \vert 
		%&\le \bigg(\tilde{C} \vert x \vert^{t+1} (1+\vert \ln (\vert x \vert)\vert)^2 \cdot 2c_2 \vert x \vert^{-(\alpha+1)}(a_1+b_1\vert x \vert) m(\varphi)\bigg)\\
		%&\quad \quad \quad \quad \quad \quad \quad \quad \quad \quad \quad \quad \quad \quad \quad \quad \quad \quad \quad \quad+\bigg( \tilde{C} \vert x \vert^{t} (1+\vert \ln (\vert x \vert)\vert)^2 \cdot 2c_2 \vert x \vert^{-\alpha} m(\varphi)\bigg)\\
	%	% &\le \bigg(\tilde{C} \vert x \vert^{t-\alpha} (1+\vert \ln (\vert x \vert)\vert)^2 \cdot 2c_2 (a_1+b_1\vert x \vert) m(\varphi)\bigg)+\bigg( \tilde{C} \vert x \vert^{t-\alpha} (1+\vert \ln (\vert x \vert)\vert)^2 \cdot 2c_2 m(\varphi)\bigg)
		% \\
		&\le C \vert x \vert^{t-\alpha}(1+|\ln (\vert x \vert)|)^2 [\max\{a_1, b_1\}(1+\vert x \vert) + 1].
	\end{align*}
	\begin{align*}
		\vert \text{(II)} \vert 
		%&\le  \bigg(C \vert x \vert^{2t} (1 + \vert \ln (\vert x \vert)\vert)^2 \cdot 2c_2 \vert x \vert^{-\alpha} m(\varphi)\bigg)\\
	%	&\quad \quad \quad \quad+\bigg(C \vert x \vert^{t}(1+\vert \ln (\vert x \vert)\vert)\cdot \vert x \vert^{t+1}(1+\vert \ln(\vert x \vert)\vert) \cdot 2c_2 \vert x \vert^{-(\alpha+1)}(a_1+b_1\vert x \vert) m(\varphi)\bigg)\\
		&\le C \vert x \vert^{2t-\alpha}(1+|\ln (\vert x \vert)|)^2 [\max\{a_1, b_1\}(1+\vert x \vert) + 1].
	\end{align*}
	\begin{align*}
		\vert \text{(III)} \vert 
		%&\le \vert x \vert^{t+1}(1+\vert \ln (\vert x \vert)\vert)\bigg[\vert x \vert^{t-1}(1+\vert \ln (\vert x \vert)\vert)\cdot 2c_2 \vert x \vert^{-\alpha}+4Cc_2\vert x \vert^{t}(1+\vert \ln (\vert x \vert)\vert) \cdot \\&
	%	%\quad \quad \quad 
	%	\quad 
	%	\quad \quad \quad \quad \vert x \vert^{-(\alpha+1)}(a_1+b_1\vert x \vert) + \vert x \vert^{t+1}(1+\vert \ln (\vert x \vert)\vert) \cdot2c_2\vert x \vert^{-(\alpha+2)}(a_2+ b_2 x^2) \bigg]m(\varphi)\\
%		%% \vert \text{(III)} \vert 
%		% &\le C \vert x \vert^{t+1}\vert \ln (\vert x \vert)\vert\bigg[\vert x \vert^{t-\alpha-1}\vert \ln (\vert x \vert)\vert+\vert x \vert^{t-\alpha-1}(1+\vert \ln (\vert x \vert)\vert) (a_1+b_1\vert x \vert)  \\& \quad \quad \quad \quad \quad \quad \quad \quad \quad \quad \quad \quad \quad \quad \quad \quad \quad \quad \quad \quad+\vert x \vert^{t-\alpha-1}\vert \ln (\vert x \vert)\vert(a_2+ b_2 x^2) \bigg]m(\varphi)\\
%		%  \vert \text{(III)} \vert \le
%		% &\le C \vert x \vert^{2t-\alpha} (\ln (\vert x \vert))^2(\max\{a_1, a_2\} + \max\{b_1, b_2\} |x|)\\
		&\le C \vert x \vert^{2t-\alpha}(1+|\ln (\vert x \vert)|)^2 [1+\max\{a_1, a_2\}+\max\{b_1, b_2\}\vert x \vert ].
	\end{align*}
	
	Since $x \in S^1\setminus\{0\}$ and $0<\alpha\le t \le \beta < 1$, 
	\begin{equation}\label{eq:part.sq}
		|\partial_{t}^{2}\mathcal{L}_{t}\varphi(x)| \le C |x|^{t-\alpha}(1+|\ln (\vert x \vert)|)^2 [1+\max\{a_1, a_2\}+\max\{b_1, b_2\}\vert x \vert ],
	\end{equation}
	which proves our claim. 
	Next, we check the hypothesis of Gou\"ezel's theorem in section \ref{Gouezel}. Firstly,
	%we conclude that $\text{(I)}$ is bounded by $C |x|^{t-\alpha}(\ln(|x|))^2$ and that both $\text{(II)}$ and $\text{(III)}$ are bounded by $C |x|^{2t-\alpha}(\ln(|x|))^2$.
	% % However, since we sum over $i \in \{1,d\}$, this proves our claim. 
	observe that just as in section \ref{sec.welldef}, we can find a $\gamma>0$ such that
	%taking $\beta$ close enough to $\alpha$,
	$$\vert \partial_{t}^2 \mathcal{L}_{t} (\mathcal{L}_{\alpha}^{n-1-j} (\mathbf{1}))\vert \le C_\gamma |x|^{\gamma},$$
	where $C_\gamma$ is independent of $n$ and $j$. 
	We see also that from Lemma \ref{partwo}, that $\partial_{t}^2 \mathcal{L}_{t} (\mathcal{L}_{\alpha}^{n-1-j} (\mathbf{1}))$ has zero mean. 
	Define a  H\"{o}lder function 
	$\displaystyle G_\alpha(x)=~\frac{\partial_{t}^2 \mathcal{L}_{t} (\mathcal{L}_{\alpha}^{n-1-j} (\mathbf{1}))(x)}{h_\alpha(x)}.$
%	Indeed, it is sufficient to show that there exists $M\ge 0$ such that
%	\begin{align*}
%		\left |G^\prime(x) \right| =\left| \frac{h_\alpha(x)(\partial_{t}^2 \mathcal{L}_{t} (\mathcal{L}_{\alpha}^{n-1-j} (\mathbf{1}))(x))^\prime-\partial_{t}^2 \mathcal{L}_{t} (\mathcal{L}_{\alpha}^{n-1-j} (\mathbf{1}))(x)h^\prime_\alpha(x)}{(h_\alpha(x))^2}\right|\le M.
%	\end{align*}
	In addition to the bounds we got above and the bounds in \eqref{eq:dens.bound}, to get the required bounds on $\left |G_\alpha^\prime (x) \right|$,
%	to get the bounds on $\left |G_\alpha^\prime (x) \right| \le M$, 
	we only need get the bounds on $(\partial_{t}^2 \mathcal{L}_{t} (\mathcal{L}_{\alpha}^{n-1-j} (\mathbf{1}))(x))^\prime$. By the commutation relation \eqref{eq:commutation}, we have to differentiate each of the three terms $\text{(I)}-\text{(III)}$ above. 
	%First, get that the 
	The bounds on $\partial_{\alpha}X^{\prime \prime}_{\alpha,i}(x)$, is
	\begin{equation}\label{eq:partialeksprimeprime}
		\partial_{\alpha}X^{\prime \prime}_{\alpha,i}(x) \le C \vert x \vert^{\alpha-1} (1+\vert \ln (\vert x \vert)\vert)^2.
	\end{equation}
	
%	We have that
%	\begin{align*}
%		\text{(I$^\prime$)}&=-\left[\partial_t X_{t,i}\,(\mathcal{N}_{t,i}\varphi)^{\prime\prime}+ 2 \partial_t X_{t,i}^\prime\,(\mathcal{N}_{t,i}\varphi)^{\prime}+ \partial_t X_{t,i}^{\prime\prime}\,(\mathcal{N}_{t,i}\varphi)\right],\\
%		\text{(II$^\prime$)}&=2(X_{t,i}^\prime)^2(\mathcal{N}_{t,i}\varphi)^\prime+2X_{t,i}^\prime X_{t,i}^{\prime\prime}\mathcal{N}_{t,i}\varphi+X_{t,i}^\prime X_{t,i}(\mathcal{N}_{t,i}\varphi)^{\prime\prime}+X_{t,i}^{\prime\prime} X_{t,i}(\mathcal{N}_{t,i}\varphi)^{\prime},\\
%		\text{(III$^\prime$)}&= X_{t,i}\left[X_{t,i}^{\prime \prime \prime} \mathcal{N}_{t,i}\varphi+X_{t,i}^{\prime \prime} (\mathcal{N}_{t,i}\varphi)^\prime+ 4X^{\prime}_{t,i}\cdot(\mathcal{N}_{t,i}\varphi)^{\prime \prime}+2X^{\prime \prime}_{t,i}\cdot(\mathcal{N}_{t,i}\varphi)^{\prime}+X_{t,i}\cdot (\mathcal{N}_{t,i}\varphi)^{\prime \prime \prime}\right]\\
%		&\quad\quad\quad\quad\quad\quad\quad\quad\quad\quad\quad\quad\quad\quad\quad\quad\quad\quad\quad\quad\quad\quad
%		+2X_{t,i}^\prime\cdot (\mathcal{N}_{t,i}\varphi)^{\prime \prime}] 
%		+ X_{t,i}^\prime\left[X_{t,i}^{\prime \prime} \mathcal{N}_{t,i}\varphi+ 2X^{\prime}_{t,i}\cdot(\mathcal{N}_{t,i}\varphi)^{\prime}\right],
%	\end{align*}
	 We then bound the derivatives as follows, taking $\varphi=\mathcal{L}_{\alpha}^{n-j-1}(\mathbf{1}) \in \mathcal{C}_{*,1}(\alpha) \cap \mathcal{C}$. Using the \eqref{eq:zeroth}, \eqref{eq:first}, \eqref{eq:ssecond}, \eqref{eq:third}, 
	\eqref{eq:partialeks}, \eqref{eq:partialeksprime}, \eqref{eq:partialeksprimeprime}, there is a $C>0$ such that
	\begin{align}
		\left| \text{(I$^\prime$)}\right| &\le C |x|^{t-\alpha-1}(1+|\ln(|x|)|)^2\,[1+\max\{a_1,a_2\}+\max\{b_1,b_2\}|x|], \nonumber\\
		\left| \text{(II$^\prime$)}\right|&\le C |x|^{2t-\alpha-1}\left(1+|\ln(|x|)|\right)^2 \, \left[1+\max\{a_1,a_2\}+\max\{b_1,b_2\}|x|\right], \nonumber\\
		\left| \text{(III$^\prime$)}\right| &\le C |x|^{2t-\alpha-1}\left(1+|\ln(|x|)|\right)^2 \, [1+\max\{a_1,a_2, a_3\}+\max\{b_1,b_2, b_3\}|x|], \nonumber
	\end{align}
	for $x \in S^1\setminus\{0\}$ and $0<\alpha\le t \le \beta < 1$, 
	\begin{equation*}
		|(\partial_{t}^2 \mathcal{L}_{t} (\mathcal{L}_{\alpha}^{n-1-j} (\mathbf{1}))(x))^\prime| \le C |x|^{t-\alpha-1}(1+|\ln(|x|)|)^2 \, [1+\max\{a_1,a_2, a_3\}+\max\{b_1,b_2, b_3\}|x|],
	\end{equation*}
	which is finite and thus allows us to conclude that $G(x)$ is  H\"{o}lder. Also, $G(0)=0$. Since it satisfies the assumptions of the theorem,
%	in section \ref{Gouezel}
	we easily bound $B_n$ as follows
	\begin{align}
		\int_{\alpha}^{\beta} \frac{\beta-t}{\beta-\alpha}  \sum_{j=0}^{n-1} \int_{0}^{1} \psi \mathcal{L}_{\beta}^{j} \left( \partial_{t}^2 \mathcal{L}_{t} (\mathcal{L}_{\alpha}^{n-1-j} (\mathbf{1}))\right)\,dx\,dt &\le \frac{\Vert \psi \Vert_\infty}{\beta-\alpha}  \int_{\alpha}^{\beta} (\beta-t) \left(C_{\beta} \sum_{j=0}^{n-1}\frac{1}{j^{1/\beta-\varepsilon}}\right)\, dt \nonumber \\
		&=\Vert \psi \Vert_\infty\cdot (\beta-\alpha)  \left(C_{\beta} \sum_{j=0}^{n-1}\frac{1}{j^{1/\beta-\varepsilon}}\right),
		\label{lipbound}
	\end{align}
	%For small $\varepsilon$, 
	as in \eqref{eq.sum0}, the term in the bracket is summable,
	%\todo{since we are summing ver finite terms}
	and 
	%for any $n$, 
	\eqref{lipbound}~$\to 0$ as $\beta \to \alpha$. Similarly,
%	In the same vein, 
	for $\alpha \in (0,1)$, with $\beta < \alpha$, we get the same estimate as before, on substituting
	\begin{equation*}
		\sum_j^{n-1}\mathcal{L}_{\beta}^j(\mathcal{L}_{\beta}-\mathcal{L}_{\alpha})\mathcal{L}_{\alpha}^{n-j-1}=- \sum_j^{n-1}\mathcal{L}_{\alpha}^j(\mathcal{L}_{\alpha}-\mathcal{L}_{\beta})\mathcal{L}_{\beta}^{n-j-1}
	\end{equation*}
	into \eqref{teles}. Hence, the Taylor's formula now is
	for any $\varphi \in C^{2}(S^1\setminus\{0\}), \beta \neq \alpha, x \neq 0$ 
	\begin{equation}
		\frac{(\mathcal{L}_{\alpha}-\mathcal{L}_{\beta})\varphi(x)}{\alpha-\beta}= \partial_{\beta}\mathcal{L}_{\beta} \varphi (x)+ \frac{1}{\alpha-\beta} \int_{\beta}^{\alpha}(t-\beta) \partial_{t}^2 \mathcal{L}_{t} \varphi (x) dt \label{eq:taylor2}
	\end{equation} 
	
	\textbf{Case 2}\,: For $\alpha=0$, \eqref{teles} can also be bounded in a similar manner, using instead, Proposition~ \ref{prop.zerodecay}. The decomposition in the Taylor's formula can be done as thus,
	\begin{equation}
		\sum_j^{n-1}\mathcal{L}_{\beta}^j(\mathcal{L}_{\beta}-\mathcal{L}_{0})\mathcal{L}_{0}^{n-j-1}=- \sum_j^{n-1}\mathcal{L}_{0}^j(\mathcal{L}_{0}-\mathcal{L}_{\beta})\mathcal{L}_{\beta}^{n-j-1}. \label{eq:zerosub}
	\end{equation}
	
	Substituting $\alpha=0$ into \eqref{teles}, 
%	\begin{align*}
%		\frac{1}{\beta} \left(\int (\psi \circ f_{\beta}^n)\,dx- \int(\psi \circ f_{0}^n)\,dx \right) = \sum_{j=0}^{n-1} \int_{0}^{1} \psi \mathcal{L}_{\beta}^{j}\left(\frac{(\mathcal{L}_{\beta}-\mathcal{L}_{0})}{\beta}(\mathcal{L}_{0}^{n-1-j} (\mathbf{1}))\right)\,dx,
%	\end{align*}
	and make the substitution from \eqref{eq:zerosub}. Also, making the substitution of $\alpha=0$ in \eqref{eq:taylor2} and subsequently using
%	\begin{align*}
%		\frac{1}{\beta} \bigg(\int (\psi \circ f_{\beta}^n)\,dx- \int(\psi \circ f_{0}^n)\,dx \bigg) = -\sum_{j=0}^{n-1} \int_{0}^{1} \psi \mathcal{L}_{0}^{j}\left(\frac{(\mathcal{L}_{0}-\mathcal{L}_{\beta})}{\beta}(\mathcal{L}_{\beta}^{n-1-j} (\mathbf{1}))\right)\,dx.
%	\end{align*}
%	
%	From equation \eqref{eq:taylor2}, we have that
%	\begin{align*}
%		\frac{1}{\beta} \bigg(\int (\psi \circ f_{\beta}^n)dx- \int(\psi \circ f_{0}^n)\,dx \bigg) = \sum_{j=0}^{n-1} \int_{0}^{1} \psi \mathcal{L}_{0}^{j}&\bigg(\partial_{\beta}\mathcal{L}_{\beta} (\mathcal{L}_{\beta}^{n-1-j} (\mathbf{1})) \nonumber\\
%		&+ \frac{1}{\beta} \int_{0}^{\beta}(t-\beta) \partial_{t}^2 \mathcal{L}_{t} (\mathcal{L}_{\beta}^{n-1-j} (\mathbf{1})) \,dt\bigg)\,dx,
%	\end{align*}
%	next, using 
	Lemma \ref{part_PF} and Lemma \ref{partwo} as before, we have
	\begin{align*}
		\frac{1}{\beta} &\left(\int (\psi \circ f_{\beta}^n)\,dx- \int(\psi \circ f_{0}^n)\,dx \right) = \\&\quad\quad\quad\quad\quad\quad\quad\quad\quad\underbrace{-\sum_{j=0}^{n-1} \int_{0}^{1} \psi \mathcal{L}_{0}^{j}\bigg(\sum_{i \in \{1,d\}}(X_{\beta,i}\mathcal{N}_{\beta,i}\mathcal{L}_{\beta}^{n-1-j} (\mathbf{1}))^{\prime}\bigg)\,dx}_{A_{n_0}} \\
		&\quad\quad \quad\quad\quad\quad\quad\quad\quad\quad\quad\quad\quad\quad+ \underbrace{\int_{0}^{\beta}\frac{t-\beta}{\beta} \sum_{j=0}^{n-1} \int_{0}^{1} \psi \mathcal{L}_{0}^{j}\left(\partial_{t}^2 \mathcal{L}_{t} (\mathcal{L}_{\beta}^{n-1-j} (\mathbf{1}))\right)\,dx\,dt}_{B_{n_0}}.
	\end{align*}
	
	The summability of $A_{n_0}$ and $B_{n_0}$ follows from Proposition \ref{prop.zerodecay}. This shows that for $\psi \in L^{\infty}(m)$ and $\beta \in [0,1)$, $\beta \mapsto \int \psi \circ f_\beta^n \,dx$ is locally Lipschitz.
	\subsection{Convergence to the limit} \label{sec.conv}
	Finally, we show the differentiability of $\beta \mapsto \int \psi d\mu_{\beta}$,
	%, for $\psi \in L^\infty$
	at $\beta=\alpha \in [0,1)$. The idea is to show that as $\beta \to \alpha$, $B_n \to 0$ (resp. $B_{n_0} \to 0$) , and $A_n$ (resp. $A_{n_0}$) converges to the claimed entity. Recall  \eqref{eq:bound},
	with $n$ as in \eqref{eq:en}, setting $n(\beta)= n(\alpha, \beta, \xi)\,\text{ for small } \xi>0$. It suffices to check that when %$\beta \to \alpha> 0$ or 
	$\beta \to \alpha^+= 0$,
%	\vspace{1in}
	\begin{align}
		\sum_{j=0}^{n(\beta)} \int_{0}^{1} \psi \mathcal{L}_{\beta}^{j}\bigg(\sum_{i \in \{1,d\}}(X_{\alpha,i}\mathcal{N}_{\alpha,i}&\mathcal{L}_{\alpha}^{n-j} (\mathbf{1}))^{\prime}\bigg)\,dx\nonumber\\
		&+ \int_{\alpha}^{\beta} \frac{\beta-t}{\beta-\alpha}  \sum_{j=0}^{n(\beta)-1} \int_{0}^{1} \psi \mathcal{L}_{\beta}^{j} \bigg( \partial_{\alpha}^2 \mathcal{L}_{t} (\mathcal{L}_{\alpha}^{n-1-j} (\mathbf{1}))\bigg)\,dx\,dt, \label{eq:Bzero}
	\end{align}
	converges to
	\begin{equation}
		\sum_{j=0}^{\infty} \int_{S^1} \psi \mathcal{L}_{\alpha}^{j}\bigg(\sum_{i \in \{1,d\}}(X_{\alpha,i}\mathcal{N}_{\alpha,i}(h_{\alpha}))^{\prime}\bigg)dx=  \int_{S^1} \psi \sum_{j=0}^{\infty} \mathcal{L}_{\alpha}^{j}\bigg(\sum_{i \in \{1,d\}}(X_{\alpha,i}\mathcal{N}_{\alpha,i}(h_{\alpha}))^{\prime}\bigg)\,dx, \label{eq:LRform}
	\end{equation}
	the linear response formula. By the summability of $B_n$ (see \eqref{lipbound}), the second term of \eqref{eq:Bzero}
	\begin{equation*}
		\int_{\alpha}^{\beta} \frac{\beta-t}{\beta-\alpha}  \sum_{j=0}^{n(\beta)-1} \int_{0}^{1} \psi \mathcal{L}_{\beta}^{j} \bigg( \partial_{\alpha}^2 \mathcal{L}_{t} (\mathcal{L}_{\alpha}^{n-1-j} (\mathbf{1}))\bigg)\,dx\,dt \to 0, \quad \text{ as } \beta \to \alpha.
	\end{equation*}
	
	Next, for $\alpha \in [0,1)$, fix $\eta>0$,
	%small
	we can take $R=R_\eta$ large enough so that the tail of the series in $\eqref{eq:LRform} < \frac{\eta}{4}$, while the tail of the first term of $\eqref{eq:Bzero} < \frac{\eta}{4}$ uniformly in $\beta$. Let $\eta >0$ (small), 
	\begin{align*}
		&\sum_{j=R_\eta}^{\infty} \int_{0}^{1} \psi \mathcal{L}_{\alpha}^{j}\bigg(\sum_{i \in \{1,d\}}(X_{\alpha,i}\mathcal{N}_{\alpha,i}(h_{\alpha}))^{\prime}\bigg)\,dx < \frac{\eta}{4},\\
		&\sum_{j=R_\eta}^{n(\beta)} \int_{0}^{1} \psi \mathcal{L}_{\beta}^{j}\bigg(\sum_{i \in \{1,d\}} \left(X_{\alpha,i}\mathcal{N}_{\alpha,i}\mathcal{L}_\alpha^{n-1-j}(\mathbf{1})\right)^\prime\bigg)\,dx < \frac{\eta}{4},
	\end{align*}
	since both series converges (from section \ref{sec.welldef} and section \ref{sec.lip}). % we may choose $K_\eta$ such that.\\
	% uniformly in $\beta$ by $$\int_{0}^{1} \psi \mathcal{L}_{\beta}^{j}\bigg(\sum_{i \in \{1,d\}}(X_{\alpha,i}\mathcal{N}_{\alpha,i}(h_{\alpha}))^{\prime}\bigg)dx \le $$ 
	Now for every fixed $0\le j \le R_\eta$,
	% we want to show that, 
	we show that the difference tends to $0$, as 
	%$\beta \to \alpha>0$ or 
	$\beta \to \alpha^+=0$
	\begin{align}
		&\sum_{R_\eta}^{\infty}\bigg\{ \int_{0}^{1} \bigg[ (\psi \circ f_{\beta}^j)\bigg(\sum_{i \in \{1,d\}}(X_{\alpha,i}\mathcal{N}_{\alpha,i}(\mathcal{L}_{\alpha}^{n(\beta)-j}(\mathbf{1}))^{\prime})\bigg) - (\psi \circ f_{\alpha}^j ) \bigg(\sum_{i \in \{1,d\}}(X_{\alpha,i}\mathcal{N}_{\alpha,i}(h_{\alpha}))^{\prime}\bigg)\bigg]\,dx\bigg\} \nonumber \\
		&=\sum_{R_\eta}^{\infty} \varsigma. \label{eq:diff2}
	\end{align}
	It is enough now to show that $\vert \varsigma \vert < \frac{\eta}{2R_{\eta}}$ i.e $\forall j \le R_\eta$ and for $\beta = \beta(\eta) \to \alpha$ as $\eta \to 0$. Naturally, as $\eta \to 0$, $\beta(\eta) \to \alpha$. So it is sufficient to show that $\exists$ $N_\eta \ge 1$ so that
	\begin{align}
		\bigg\Vert \sum_{i \in \{1,d\}}(X_{\alpha,i}&\mathcal{N}_{\alpha,i}(\mathcal{L}_{\alpha}^{n}(\mathbf{1})))^{\prime} - \sum_{i \in \{1,d\}}(X_{\alpha,i}\mathcal{N}_{\alpha,i}(h_{\alpha}))^{\prime} \bigg\Vert_1 \nonumber\\
		&= \left\| \sum_{i \in \{1,d\}} \left((X_{\alpha,i}\mathcal{N}_{\alpha,i}(\mathcal{L}_{\alpha}^{n}(\mathbf{1})))^{\prime}- (X_{\alpha,i}\mathcal{N}_{\alpha,i}(h_{\alpha}))^{\prime}\right)\right\|_1 \le \frac{\eta}{2R_\eta}, \quad \forall n\ge N_\eta.  \label{eq:Bee}
	\end{align}
	%Thus as $j \le K_\eta$, now up to taking $\beta(\eta)$ close enough to $\alpha$, we need to show that.\\
	
	\textbf{Case 1}\,: At $\alpha=0$, this holds trivially by \eqref{eq:Bee} and since $h_{0}=\mathbf{1}$.
	
	\textbf{Case 2}\,: $\alpha \in (0,1)$, fix $\alpha$. Now, set 
	$\phi_n:=\mathcal{L}_\alpha^{n}{(\mathbf{1})}.$
	By the Leibniz rule, we get
	\begin{align}
		&\left\| \sum_{i \in \{1,d\}} \left[(X_{\alpha,i}\mathcal{N}_{\alpha,i}(\mathcal{L}_{\alpha}^{n}(\mathbf{1})))^{\prime}- (X_{\alpha,i}\mathcal{N}_{\alpha,i}(h_{\alpha}))^{\prime}\right]\right\|_1 \nonumber\\
	%	&=  \left\| \sum_{i \in \{1,d\}}\left[ X_{\alpha,i}^{\prime} \mathcal{N}_{\alpha,i}(\phi_n)+ X_{\alpha,i}(\mathcal{N}_{\alpha,i}(\phi_n))^{\prime}- X_{\alpha,i}^{\prime}\mathcal{N}_{\alpha,i}(h_{\alpha})-X_{\alpha,i}(\mathcal{N}_{\alpha,i}(h_{\alpha}))^{\prime}\right] \right\|_1 \nonumber\\
	%	&\le  \left\| \sum_{i \in \{1,d\}}\left[ X_{\alpha,i}^{\prime} \mathcal{N}_{\alpha,i}(\phi_n)- X_{\alpha,i}^{\prime}\mathcal{N}_{\alpha,i}(h_{\alpha})\right]\right\|_1 +  \left\| \sum_{i \in \{1,d\}} \left[ X_{\alpha,i}(\mathcal{N}_{\alpha,i}(\phi_n))^{\prime}-X_{\alpha,i}\left(\mathcal{N}_{\alpha,i}(h_{\alpha})\right)^{\prime} \right]\right\|_1
		\nonumber\\
		&\le \underbrace{  \left\| \max_i X_{\alpha,i}^{\prime}\sum_{i \in \{1,d\}}  \mathcal{N}_{\alpha,i}(\phi_n- h_{\alpha})\right\|_1}_\text{(I)} + \underbrace{  \left\|     \sum_{i \in \{1,d\}}\left[X_{\alpha,i}\left(\mathcal{N}_{\alpha,i}(\phi_n-h_{\alpha})\right)^{\prime}\right]\right\|_1}_\text{(II)}
		%&=: (Q_1) + (Q_2).
	\end{align}
	
	$X_{\alpha,i}^\prime \in L^\infty(m)$, by the H\"older inequality, \eqref{eq:PF}, and the contraction property of $\mathcal{L}_\alpha(\cdot)$
	\begin{align}
		\text{(I)}
		% &\le \left\| \max_i X_{\alpha,i}^{\prime}\right\|_{\infty}  \left\| \sum_{i \in \{1,d\}}  \mathcal{N}_{\alpha,i}(\phi_n- h_{\alpha})\right\|_1 \nonumber\\
		&\le \left\| \max_i X_{\alpha,i}^{\prime}\right\|_{\infty}  \left\| \mathcal{L}_\alpha(\phi_n- h_{\alpha})\right\|_1 \nonumber\\
		&\le C_\alpha  \Vert \phi_n-h_{\alpha} \Vert_{1}. \nonumber
	\end{align}
	
	From \cite[Theorem 5]{YL99}, we have that
	\begin{equation}
		\text{(I)} \le C_\alpha n^{1-1/\alpha}, \label{eq:LSYest}
	\end{equation}
	%The above 
	which is only summable for $\alpha < 1/2$. However, it tends to zero for all $\alpha \in (0,1)$.
	\begin{remark}
		Setting $\psi=\mathbf{1}$ and $\varphi=\mathbf{1}-h_\alpha$,	we may as well use Lemma \ref{Qdecay} to bound (I) above.
	\end{remark}
	Since $X_{\alpha,i} \in L^\infty(m)$ we need only show that $\displaystyle\text{(II)} < \frac{\eta}{2R_{\eta}}$.
	Now, same as before, we let $\varphi \in \{\mathbf{1}, h_{\alpha}\}$, for any $\bar{x} \in S^1\setminus\{0\}$, $n \ge 0$, $X_{\alpha,i}\in L^\infty(m)$, by the H\"older inequality
	\begin{align*}
		\text{(II)}
		%& \le  \left\| \sum_{i \in \{1,d\}}\left[X_{\alpha,i}\left(\mathcal{N}_{\alpha,i}(\phi_n-h_{\alpha})\right)^{\prime}\right]\right\|_1 \\
		&\le \int_{S^1}\left\vert  \sum_{i \in \{1,d\}}\left[X_{\alpha,i}\left(\mathcal{N}_{\alpha,i}(\phi_n-h_{\alpha})\right)^{\prime}\right] \right\vert dz\\
		&=
		\underbrace{\int_0^{\bar{x}} \left|  \sum_{i \in \{1,d\}}\left[X_{\alpha,i}\left(\mathcal{N}_{\alpha,i}(\phi_n-h_{\alpha})\right)^{\prime}\right] \right| dz}_\text{(A)}+ \underbrace{\int_{\bar{x}}^{1-{\bar{x}^\prime}} \left|  \sum_{i \in \{1,d\}}\left[X_{\alpha,i}\left(\mathcal{N}_{\alpha,i}(\phi_n-h_{\alpha})\right)^{\prime}\right] \right| dz}_\text{(B)}\\
		&\quad\quad\quad\quad\quad\quad\quad\quad\quad\quad\quad\quad\quad\quad\quad\quad\quad\quad\quad\quad+  \underbrace{\int_{\bar{x}^\prime}^1 \left|  \sum_{i \in \{1,d\}}\left[X_{\alpha,i}\left(\mathcal{N}_{\alpha,i}(\phi_n-h_{\alpha})\right)^{\prime}\right] \right| dz}_\text{(C)}.
	\end{align*}
	
	We recall the bounds in \eqref{eq:zeroth} and the definition of the cone equation \eqref{eq:Kone1} and in Lemma \ref{inv.cone} with $\varphi=\{h_\alpha, \mathbf{1}\}$, there exists $a_1,b_1,c_2>0$ such that
	%\eqref{eq:enprime}.
	\begin{equation*}
		\vert (\mathcal{N}_{\alpha,i}\mathcal{L}^n(\varphi))^\prime (x)\vert \le \left(\frac{a_1}{\vert x \vert} +b_1\right)\mathcal{N}_{\alpha,i}(\varphi)(x) \le c_2 (a_1\vert x \vert^{-(1+\alpha)} + b_1\vert x \vert^{-\alpha}), \,  \forall n \ge 1, x \in S^1\setminus\{0\}.
	\end{equation*}
	\begin{align}
		(\text{A}) &\le 2 \sum_{i \in \{1,d\}}\int_0^{\bar{x}} \left|  X_{\alpha,i}(z)\left(\mathcal{N}_{\alpha,i}\mathcal{L}^n(\varphi)\right)^{\prime}(z) \right| \,dz\nonumber\\
		%&\le \tilde{C} \int_0^{\bar{x}} |z|^{\alpha+1}(1+|\ln(|z|)|) (a_1\vert z \vert^{-(1+\alpha)} + b_1\vert z \vert^{-\alpha})\,dz , \quad \tilde{C} \ge1 \nonumber\\
		&\le \tilde{C} \int_0^{\bar{x}} (1+|\ln(|z|)|) (a_1 + b_1\vert z \vert)\,dz , \quad \tilde{C} \ge1 \nonumber\\
		&\le \tilde{C} \bar{x}\left[a_1\left(\left|\ln(|\bar{x}|)\right|+2\right)+\frac{b_1|\bar{x}|}{2}\left(2+\frac{|\bar{x}|}{2}\right)\right].\label{eq:eksbar}
	\end{align}
	
	The estimate for (C) is similar. We observe that \eqref{eq:eksbar} is small for small $\bar{x}$. To estimate (B), we choose $n$ large enough, so that this is so small, and we bound it. Now, by H\"older inequality and equation \eqref{eq:PF}
	\begin{align}
		\text{(B)}&\le \int_{\bar{x}}^{1-{\bar{x}^\prime}} \left|\max_i X_{\alpha,i} \sum_{i \in \{1,d\}}\left(\mathcal{N}_{\alpha,i}\mathcal{L}^n(\mathbf{1}-h_\alpha)\right)^{\prime} \right| \,dz\nonumber\\
		&\le C_\alpha \int_{\bar{x}}^{1-{\bar{x}^\prime}} \left| ( \mathcal{L}_{\alpha}^{n+1}(\mathbf{1}-h_\alpha))^{\prime}(z)\right| dz +C_\alpha \sum_{\begin{smallmatrix}
				y \in f_{\alpha,i}^{-1}(x) \\
				2 \le i \le d-1
		\end{smallmatrix}} \int_{\bar{x}}^{1-{\bar{x}^\prime}} \left| \frac{(\mathcal{L}_{\alpha}^{n}(\mathbf{1}-h_\alpha))^{\prime}(g_{\alpha,i}(x))}{f_{\alpha}^{\prime}(g_{\alpha,i}(x))} \right| dz. \label{eq:eks}
	\end{align} 
	
	The second integral is ignored, of course, if it is a $d=2$ branch circle map. To estimate the first integral in the above equation, we proceed as follows. For $x_0$ (resp. $x_0^\prime$) in the neighbourhood of the intermittent fixed point, 
	%we proceed as follows. 
	define 
	$x_0 \in (0, \varepsilon_0]$ and 
	\begin{equation}
		\begin{split}
			\bar{x}=x_l=g_{\alpha,1}^l(x_0)
			\qand
			{\bar{x}^\prime}=x^\prime_l=g_{\alpha,d}^l(x_0^\prime)
			%     = z_{l-1}^{\prime}
			, \quad l \ge 1,
		\end{split}
	\end{equation}
	we define similarly $g_{\alpha,d}^l$ in the neighbourhood around $[ -\varepsilon_0,0)$. This gives rise to the sequence, $(x_l)_l$ (resp.  $(x_l^\prime)_l$ ). From  \cite[Remark 3.63]{JFA20},
	\begin{align}\label{eq.assymp}
		x_l \approx l^{-1/\alpha} \qand x_l^\prime \approx l^{-1/\alpha} \quad \quad \text{ for } l \ge 1.
	\end{align}
	%\bar{x}=x_l :=(g_{\alpha,i})^l(x_{0}) \le \quad \quad \text{ for } l \ge 1, i=1,d.
	%$x_n$ has the same asymptotics as $n^{-1/\alpha}$ \cite{LSV98}.\\
	Next, for every $1 \le m \le n$, and $l \ge 1$, set
	%\todo{maybe the $n$ can be changed to $k$}
	\begin{align}\label{eq.y}
		\begin{split}
			y_{m}(x_l)&=g_{\alpha,1}^m(x_l)=g_{\alpha,1}^m(g_{\alpha,1}^l)(x_{0})=g_{\alpha,1}^{l+m}(x_{0})=x_{l+m} \quad \quad \text{ for } l \ge 1,\\
			y_{m}(x_l^\prime)&=g_{\alpha,d}^m(x_l^\prime)=g_{\alpha,d}^m(g_{\alpha,d}^l)(x_{0}^\prime)=g_{\alpha,d}^{l+m}(x_{0}^\prime)=x^\prime_{l+m} \quad \quad \text{ for } l \ge 1.
		\end{split}
	\end{align}
	
	Set
	\begin{equation}\label{eq.pphi}
		\phi= \phi_{n+1-m}-h_{\alpha},
	\end{equation}
	with $\mathcal{L}_{\alpha}^n(\mathbf{1})=\phi_n$. Hence,
	\begin{equation*}
		\mathcal{L}_{\alpha}^m(\phi)=  \mathcal{L}_{\alpha}^m(\mathcal{L}_{\alpha}^{n+1-m}(\mathbf{1})-h_{\alpha})= \mathcal{L}_{\alpha}^{n+1}(\mathbf{1}-h_{\alpha}),
	\end{equation*}
	and
	\begin{align*}
		(\mathcal{L}^m_\alpha(\phi))^{\prime}(x)
		%=\frac{d}{dx}\left(\sum_{f_{\alpha}^m(y)=x} \frac{\phi(y)}{(f_{\alpha}^m)^{\prime}(y)}\right)\\
		=\sum_{f_{\alpha}^m(y)=x}\frac{1}{(f_{\alpha}^m)^{\prime}(y)}\cdot \frac{\phi(y)}{(f_{\alpha}^m)^{\prime}(y)}\left(\frac{\phi^{\prime}(y)}{(f_{\alpha}^m)^{\prime}(y)}-\phi(y)\frac{(f_{\alpha}^m)^{\prime \prime}(y)}{((f_{\alpha}^m)^{\prime}(y))^2} \right),
	\end{align*}
	taking $-\bar{x}= \bar{x}^\prime$, we see however that
	\begin{align*}
		\int_{\bar{x}}^{1-\bar{x}^\prime} \vert ( \mathcal{L}_{\alpha}^{n+1}(\mathbf{1}-h_\alpha))^{\prime}(z)\vert \,dz &= \int_{S^1} \left| \chi_{|x|>\bar{x}}\left[( \mathcal{L}_{\alpha}^{n+1}(\mathbf{1}-h_\alpha))^{\prime}(z)\right]\right| \,dz\\
		&\le \left\|\ \chi_{|x|>\bar{x}}\left[( \mathcal{L}_{\alpha}^{n+1}(\mathbf{1}-h_\alpha))^{\prime}(z)\right] \right\|_{1}\\
	%	&= \left\|\ \chi_{|x|>\bar{x}}\left[( \mathcal{L}_{\alpha}^m(\phi))^{\prime}(z)\right] \right\|_{1}\\
	%	&\le \left\| \mathcal{L}^{m}_{\alpha}\left(\chi_{|y|>y_m}\frac{\vert \phi^{\prime}\vert}{(f_{\alpha}^m)^{\prime}}\right)\right\|_1 + \left\| \mathcal{L}^{m}_{\alpha}\left(\chi_{|y|>y_m}\frac{\vert \phi\vert \vert (f_{\alpha}^m)^{\prime \prime}}{((f_{\alpha}^m)^{\prime})^2}\right)\right\|_1\\
		&\le \underbrace{\left\| \chi_{|y|>y_m}\left| \phi^{\prime}\right| \cdot ((f_{\alpha}^m)^{\prime})^{-1}\right\|_1}_M + \underbrace{\left\| \chi_{|y|>y_m} \cdot \left| \phi\right| \left| (f_{\alpha}^m)^{\prime \prime} ((f_{\alpha}^m)^{\prime})^{-2}\right|\right\|_1}_N
	\end{align*}
	%\end{align*}
	%note that in the above, we have used that 
	%Hence,
	%\begin{align}
	%	\left\| \chi_{|x|>\bar{x}}\left[( \mathcal{L}_{\alpha}^{n+1}(\varphi))^{\prime}\right] \right\|_{1} 
	\begin{align*}
		M \le ((f_{\alpha}^m)^{\prime})^{-1} \Vert \chi_{|y|>y_m}\vert \phi^{\prime}\vert \Vert_1 
	\end{align*}
	
	To estimate $((f_{\alpha}^m)^{\prime})^{-1}$ for some $|y| \ge y_m(x_l)$, we use the bounded distortion property, of $f_\alpha$ \cite[Lemma 5]{YL99} on $(y_m, f_\alpha(y_m))= (y_m, y_{m-1})$, by \eqref{eq:zerothderivative}, \eqref{eq.assymp}
	and \eqref{eq.y}
	\begin{align}
		((f_{\alpha}^m)^{\prime})^{-1}&\le C \frac{f_{\alpha}(y_m)-y_m}{f_{\alpha}(x_l)-x_l}\le C \frac{\sgn(y_m) \vert y_m \vert^{\alpha+1}}{\sgn(x_l) \vert x_l \vert^{\alpha+1}}= C \frac{\sgn(x_{l+m}) \vert x_{l+m} \vert^{\alpha+1}}{\sgn(x_l) \vert x_l \vert^{\alpha+1}}\nonumber\\
		&\le C_\alpha \left(1 + \frac{m}{l}\right)^{-(1+1/\alpha)}. \label{eq:emone}
	\end{align}
	
	From \eqref{eq:enprime}, \eqref{eq.y} and  \eqref{eq.pphi} 
	\begin{align}
		\left\| \chi_{|y|>y_m}\vert \phi^{\prime}\vert \right\|_1 &\le \left\| \chi_{|y|>y_m}\left(\vert \phi_{n+1-m}^{\prime}\vert + \vert h^\prime_\alpha \vert\right) \right\|_1 \le C_\alpha \int_{y_m(x_l)}^{y_m(x^\prime_l)}c_2 \left(a_1\vert x \vert^{-(1+\alpha)} + b_1\vert x \vert^{-\alpha}\right)\,dy \nonumber\\
		&\le C_{\alpha}\left( \frac{-a_1c_2\sgn(y)}{\alpha} | y |^{-\alpha} + \frac{b_1c_2 \sgn(y)}{(1-\alpha)}|y|^{(1-\alpha)}\right)\bigg\vert_{y_m(x_l)}^{y_m(x^\prime_l)}\nonumber\\ 
		&\le C_{\alpha}y_m^{-\alpha} \le C_\alpha (l+m). \label{eq:emtwo}
	\end{align}
	
	Hence, we have from \eqref{eq:emone} and \eqref{eq:emtwo} that
	\begin{equation}
		M \le C_\alpha\, l\bigg(1 + \frac{m}{l}\bigg)^{-1/\alpha}. \label{eq:em}
	\end{equation}
	
	From \eqref{eq:LSYest}, we have that
	\begin{align}
		N \le C_{\alpha,m}\Vert \phi_{n+1-m}-h_{\alpha} \Vert_1 \le C_{\alpha,m} (n+1-m)^{1-1/\alpha}, \label{eq:second}
	\end{align}
	where $C_{\alpha,m}= \sup_m \sup_x \vert (f_{\alpha}^m)^{\prime \prime} ((f_{\alpha}^m)^{\prime})^{-2}\vert$.
	
	To estimate the second integral in \eqref{eq:eks}, we use the change of variables, and this gives
	$$\int_{g_{\alpha,i}(\bar{x})}^{1}  X_{\alpha,i}(f_{\alpha}(v))\mathcal{L}_{\alpha}^{n}(\varphi))^{\prime}(v) \,dv$$
	and then proceed as in the first integral.
	
	We choose $l \ge 1$, such that $\bar{x}$ is small enough to make (A) and (C) (see \eqref{eq:eksbar}) small, $m \ge l$ so that %equation 
	\eqref{eq:em} is small and we choose $\beta$ close enough to $\alpha$, so that $n(\beta)$ is large enough to make  
%	equation 
	\eqref{eq:second} small. Which shows what we want for $\psi \in L^\infty(m)$.
	%\subsubsection{Second integral:}$\,$ 

	%\qed
	\subsection{Observables in $L^q$} We generalize the result to any $\psi \in L^q(m)$, $(1-\alpha)^{-1} < q < \infty$, for $\alpha \in [0,1)$. Let us define a bounded function
	 $\psi_M(x)= \min \{\psi(x), M\}.$ 
	By the Chebyshev's inequality, for $\psi \in L^q(m)$,
%	\begin{equation*}
%		\Leb(\{\psi(x)>M\}) \le \frac{\Vert \psi \Vert_{q}^{q}}{M^q}.
%	\end{equation*}
	and letting $\Vert \psi \Vert_{q}=1$ 
%	in the above equation
	, we have 
	
	\begin{equation}\label{eq.cheb}
		\Leb(\{\psi(x)>M\}) \le {M^{-q}},
	\end{equation}
	and $\Vert \psi - \psi_M \Vert_r \le M^{1-q/r}$, for $r  \ge 1$. Indeed,
	\begin{equation*}\label{eq.pssi}
		\psi_M(x)= 
		\begin{cases}
			%\displaystyle
			\psi(x), & \mbox{if }    \psi(x) \le M;\vspace{.1cm}\\
			%\displaystyle
			M, & \mbox{if } \psi(x)>M,%
		\end{cases}
	\end{equation*}
	From \eqref{eq.cheb}, we have that $\Vert \psi - \psi_M \Vert_r \le M^{1-q/r}.$
%	\begin{align}
%		\Vert \psi - \psi_M \Vert_r &\le \left(\int_{\psi(x)>M} dx \right)^{1/p}\left(\int_{\psi(x)>M} |\psi - \psi_M|^q \,dx \right )^{1/q}, \quad\quad \frac{1}{p}+\frac{1}{q}=\frac{1}{r}, \nonumber\\
%		&\le (\Leb({\psi(x)>M}))^{1/p}\Vert \psi \Vert_{q}\nonumber\\
%		&\le M^{1-q/r}. \label{eq.emem}
%	\end{align}
%	
	Decomposing $\psi(x)$ into $	\psi(x)= (\psi(x)-\psi_M(x)) + \psi_M(x),$
%	\begin{equation} \label{eq.dec}
%		\psi(x)= (\psi(x)-\psi_M(x)) + \psi_M(x),
%	\end{equation}
	we show that the limit is well defined (just as in subsection \ref{sec.welldef}) for $\psi \in L^q(m)$. We only need show that $\displaystyle\left|\int \psi \mathcal{L}_\alpha^j \left(F_{\alpha} h_\alpha \right) dx \right|$ is summable.
%	\begin{align}
%		\left|\int \psi \mathcal{L}_\alpha^j \left( F_{\alpha} h_\alpha \right) dx \right|= \left|\int [(\psi(x)-\psi_M(x)) + \psi_M(x)] \mathcal{L}_\alpha^j \left(F_{\alpha} h_\alpha \right) dx \right|.\nonumber
%	\end{align}
	
	If $\psi_M(x)= \psi(x)$, the above is trivially true. So, for $\psi_M(x)=M(j):=j^\eta$
	\begin{align*}
		\left|\int \psi \mathcal{L}_\alpha^j \left( F_{\alpha} h_\alpha \right) dx \right| &\le \underbrace{\left|\int (\psi(x)-\psi_M(x)) \mathcal{L}_\alpha^j (F_{\alpha} h_\alpha ) dx \right|}_\text{(I)} + \underbrace{\left|\int \psi_M(x) \mathcal{L}_\alpha^j (F_{\alpha} h_\alpha ) \,dx \right|,}_\text{(II)}
		%    &=: (I)+(II).
	\end{align*}
	by the H\"older inequality, $F_{\alpha}$ being bounded and from the bounds on $\Vert \psi - \psi_M \Vert_r$ 
%	equation \eqref{eq.emem} 
	%we simplify $\text{(I)}$ as 
	\begin{align*}
		\text{(I)} \le \Vert \psi-\psi_M \Vert_r \cdot \Vert F_{\alpha} \Vert_{r^\prime} \le C j^{\eta(1-q/r)}, \quad \text{ where } 1/r + 1/r^\prime =1
	\end{align*}
	% and $F_{\alpha}$ is bounded.
	%For the above to be summable, 
	%$\eta(1-q/r)<-1$
	
	We simplify (II) as in \eqref{eq.sum0},
	%subsection \ref{sec.welldef}
	using \cite[Theorem 2.4.14]{gouezel2004vitesse}
	\begin{align*}
		\text{(II)} &\le j^\eta \int | \mathcal{L}_\alpha^j (F_{\alpha} h_\alpha ) |dx\\
	%	&\le j^\eta \cdot K_\alpha \frac{1}{j^{(1/\alpha)-\varepsilon}}\\ 
		&\le K_\alpha \frac{1}{j^{(1/\alpha)-\varepsilon-\eta}}.
	\end{align*}
	(I) and (II) are summable for $\displaystyle \frac{1}{(q/r-1)}< \eta <1/\alpha - \varepsilon-1$.
	
	To extend Subsection \ref{sec.lip} to $\psi \in L^q(m)$, we use the same decomposition of $\psi$ and follow a similar calculation as above to show the summability of $A_n$ and $B_n$.
	In subsection \ref{sec.conv}, we take $M(n)=~ n^\eta$, $\eta \in(0, 1/\alpha -1)$. Substituting the decomposition of $\psi$ 
	%from equation \eqref{eq.dec} 
	into \eqref{eq:diff2}, we have that
	%\vspace{1in}
	\begin{align*}
		\varsigma=&\bigg\{ \int_{S^{1}} \bigg[ ((\psi-\psi_M) \circ f_{\beta}^j)\bigg(\sum_{i \in \{1,d\}}(X_{\alpha,i}\mathcal{N}_{\alpha,i}(\mathcal{L}_{\alpha}^{n(\beta)-j}(\mathbf{1}))^{\prime})\bigg) - \left((\psi-\psi_M) \circ f_{\alpha}^j \right)\\
		&\quad\cdot\bigg(\sum_{i \in \{1,d\}}(X_{\alpha,i}\mathcal{N}_{\alpha,i}(h_{\alpha}))^{\prime}\bigg)\bigg]\,dx\bigg\}
		+\bigg\{ \int_{S^{1}} \bigg[ (\psi_M \circ f_{\beta}^j)\bigg(\sum_{i \in \{1,d\}}(X_{\alpha,i}\mathcal{N}_{\alpha,i}(\mathcal{L}_{\alpha}^{n(\beta)-j}(\mathbf{1}))^{\prime})\bigg)\\
		& \quad\quad\quad\quad\quad\quad\quad\quad\quad\quad\quad\quad\quad\quad\quad\quad\quad\quad\quad- (\psi_M \circ f_{\alpha}^j ) \bigg(\sum_{i \in \{1,d\}}(X_{\alpha,i}\mathcal{N}_{\alpha,i}(h_{\alpha}))^{\prime}\bigg)\bigg]\,dx\bigg\}\\
		&=: (\text{I})+(\text{II}).
	\end{align*}
	\begin{equation*}
		(\text{I}) \le M^{(1-q/r)}\left(\Vert X_{\alpha,i}\mathcal{N}_{\alpha,i}(\mathcal{L}_{\alpha}^{n(\beta)-j}(\mathbf{1}))^{\prime}\Vert_{r^\prime} + \Vert X_{\alpha,i}\mathcal{N}_{\alpha,i}(h_{\alpha}))^{\prime}\Vert_{\infty}\right).
	\end{equation*}
	
	Hence, choosing $M(n)$ big enough, we can bound the above by $\eta/2R_{\eta}$.
	$(\text{II}) \to 0$ as $\beta \to \alpha$. Now, we can choose $\beta= \beta(\eta, M(n))$ such that $|(\text{II})|<\eta/2R_{\eta}$.

	\section{Statistical stability for the solenoid map}
Here, we present an example of a degree 2 circle map that satisfies the conditions \ref{s1}-\ref{s3} in section \ref{se.inte}. For $0<\alpha<1$, Let
	
	\begin{equation*}
		f_{\alpha}(x)=
		\begin{cases}
			x(1+2^{\alpha} x^{\alpha}),  & 0\leq x <\frac{1}{2}\\
			x-2^{\alpha} (1-x)^{(\alpha + 1)},  & \frac{1}{2}\leq x \leq1.
		\end{cases}
	\end{equation*}

	%clearly, this satisfies the conditions \ref{s1}-\ref{s3} . 
	We now proceed to verify the assumptions 
	%in section \ref{se.mech}
	\ref{A1}-\ref{A3}. Indeed, for $\alpha\in (0,1)$ and $x \in I_1$, $\partial_{\alpha}f_{\alpha,1}(x)= 2^{\alpha} x^{\alpha+1} \ln (2x)$, and for $x \in I_2$, $\partial_{\alpha}f_{\alpha,2}(x)=~-2^{\alpha}~ (1-~x)^{\alpha+1} ~\ln (2(1-~x))$. This verifies the assumption \ref{A1}.
	%given in equation \eqref{Eq.eqvee}. 
	To verify 
	%equation \eqref{endp}
	assumption \ref{A2}, observe that $I_{1,+}=I_{2,-}=\frac{1}{2}$, $\partial_{\alpha}f_{\alpha,1}(\frac{1}{2})=\partial_{\alpha}f_{\alpha,2}(\frac{1}{2})=0$. The assumption \ref{A3} is easily verified.

	Consider a solid torus $M=S^1 \times \mathbb{D}$, where $\mathbb{D}$ is the unit disk, $F_{\alpha}:M \to M$  a skew product defined 
	
	\begin{equation}\label{eq.solll}
		(x,\mathbf{y}) \mapsto (f_{\alpha}(x), g_x(\mathbf{y})) := F_{\alpha}(x,\mathbf{y})=\left(f_{\alpha}(x), \frac{1}{2}\cos(2\pi x)+ \frac{1}{5}y, \frac{1}{2}\sin(2\pi x)+ \frac{1}{5}z \right)
	\end{equation}
	$f_{\alpha}:S^1 \to S^1$, the intermittent circle map, with an ergodic SRB measure $\mu_\alpha$. since contrary to the classical solenoid attractor, the solenoid with intermittency introduced in \cite{AP08} has the intermittent map in the base dynamics. 
%	The solenoid map contained in the solid torus $M \in \mathbb{R}^3$, defined
%	%we could see this 
%	as a skew product between a solid disc and a circle. Geometrically, the transformation stretches the torus to twice its length, shrinks its diameter by a factor of $5$, then twists it and doubles it over, placing the object back into the solid torus without self-intersecting. $F_{\alpha}$ an embedding into itself by means of the a degree $d=2$ 
%	%$S^1$
%	circle map and intersects each disk in two 
%	%(the $S^1$ is of degree $d=2$) 
%	smaller disks of $1/5$ the diameter. At each iterate, $F_\alpha$ contracts volume by a factor of $2$, yet there is expansion in the $x$ direction albeit non-uniform, since contrary to the classical solenoid attractor, the solenoid with intermittency introduced in \cite{AP08} has the intermittent map in the base dynamics. 
%	
%	%The solenoid attractor
%	%$$ \Omega= \cap_{n=1}^\infty F_\alpha^n(M)$$
%	%has zero-Lebesgue measure.
%	%%and is (locally) topologically a line segment cross a two dimensional Cantor set. The set $\Omega$ is an attractor, since all points in M limit to A under iteration by F upon $\Omega$.
%	%%We present in this section a diffeomorphism introduced in \cite{AP08}. 
%	%
%	%
%	%
%	%%\\
%	%%\begin{equation*}
%	%%	\Gamma^s=\{\{x\} \times \mathbb{D}: x \in S^1\}
%	%%\end{equation*}
%	%%For $x, y \in \gamma \in \Gamma^s \subset M$, $n \ge 1, \exists \, \lambda \in (0,1)$, such that \cite{JFA20},
%	%\begin{equation}
%	%	d(F_\alpha ^n (p), F_\alpha ^n(q)) \le \lambda^n d(p,q), \quad p,q \in M. \label{eq:dist}
%	%\end{equation}
%	%%
%	
	Clearly, from \eqref{eq.solll} there exists a semi-conjugacy between $f_\alpha$ and $F_\alpha$, given by $	f_\alpha \circ \pi= \pi \circ F_\alpha,$
%	\begin{equation}
%		f_\alpha \circ \pi= \pi \circ F_\alpha,
%	\end{equation}
	with $\pi:M \to S^1$, the natural projection. The next result gives the existence of a probability measure on $M$.
	\begin{lemma} \cite[Lemma 4.31]{JFA20}\label{lift}
		For $0< \alpha<1$, there exists an F-invariant Borel probability measure $\eta_\alpha$ on $M$ such that $\pi_*\eta_\alpha=\mu_\alpha$. Moreover, the support of $\eta_\alpha$ coincides with $\Omega$.
	\end{lemma}
	%\begin{proof}
	%	See 
	%\end{proof}
	%\begin{theorem}
	%    If $\pi_*(\hat{\mu})= \mu_\alpha$ $\implies$ $\hat{\mu}$ is an SRB measure.
	%\end{theorem}
	%\begin{proof}
	%    See \cite[Theorem 4.7, Lemma 4.31]{JFA20}
	%\end{proof}
	%
	%From Lemma \ref{lift} the measure which projects down, $\eta$ is $F_\alpha$ invariant, hence, it is a fixed point for $F_\alpha$. 
	
	%i.e
	%
	%$$(F_\alpha)_* \eta=\eta$$
	%
	
	Where $\Omega$ is a compact attractor. To keep the notations simple, we drop the subscript for $\mu_\alpha, \eta_\alpha$ and simply write $\mu, \eta$. In what follows, we will show that the measure $\eta$ in Lemma \ref{lift}  is in fact the unique $SRB$ measure of $F$. 
	\begin{proposition}
		Suppose that $\mu$ is an SRB measure and $\pi_*\eta= \mu$, then $\eta$ is the unique SRB measure.
	\end{proposition}
	\begin{proof}
		We know that $R$ is constant on stable disks, hence we have that
		\begin{equation}
			\sum_{j=0}^{\infty} \eta_0\{R > j\}=\sum_{j=0}^{\infty} \eta_0\{R \circ \pi > j\}=\sum_{j=0}^{\infty} \pi_*\eta_0\{R > j\}=\sum_{j=0}^{\infty} \mu_0\{R > j\}. \label{eq:norm}
		\end{equation}
		
		Let $A \subset S^1$ be any Borel measurable set. using \eqref{eq:norm} From \cite[Corollary 3.21]{JFA20}, \cite[Lemma 4.9]{JFA20}, \cite[Lemma 4.5]{JFA20}
		% Corollary \ref{circ.inducedd}, Theorem~ \ref{sol.inducedd}, Lemma \ref{full}
		,  and the semi-conjugacy property of the projection map, $\pi: M \to S^1$, we get
		\begin{align*}
			\pi_*\eta(A)&= \frac{1}{\sum_{j=0}^{\infty} \mu_0\{R > j\}}\sum_{j=0}^{\infty} F_{*}^{j}(\eta_0 (\pi^{-1}(A))\vert \{R > j\})\\
			&= \frac{1}{\sum_{j=0}^{\infty} \mu_0\{R > j\}}\sum_{j=0}^{\infty} \eta_0 (F^{-j} \circ \pi^{-1}(A)\vert \{R > j\})\\
			% &= \frac{1}{\sum_{j=0}^{\infty} \mu_0\{R > j\}}\sum_{j=0}^{\infty} \eta_0 ((\pi \circ F^{j})^{-1}(A)\vert \{R > j\})\\
			% &= \frac{1}{\sum_{j=0}^{\infty} \mu_0\{R > j\}}\sum_{j=0}^{\infty} \eta_0 ((f^{j} \circ \pi)^{-1}(A)\vert \{R > j\})\\
			&= \frac{1}{\sum_{j=0}^{\infty} \mu_0\{R > j\}}\sum_{j=0}^{\infty} f_*^{j}\pi_*(\eta_0 (A)\vert \{R > j\})\\
			% &= \frac{1}{\sum_{j=0}^{\infty} \mu_0\{R > j\}}\sum_{j=0}^{\infty} f_*^{j}(\mu_0 (A) \vert \{R > j\})\\
			&= \frac{1}{\sum_{j=0}^{\infty} \mu_0\{R > j\}}\sum_{j=0}^{\infty} f_*^{j}(\mu_0 (A) \vert \{R > j\})= \mu(A).
		\end{align*}
	\end{proof}

	Since the $\eta$ is the unique SRB measure of $F$, we now show that it is statistically stable. 
	%By statistical stability, we mean that 
	For any continuous $\phi:M \to \R$, let $\phi^{+/-}:S^1 \to \R$ defined for each $x \in S^1$ by
	\begin{equation*}
		\phi^-(x)=\inf_{p \in \pi^{-1}(x)} \phi(p) \quad \qand \quad 	\phi^+(x)=\sup_{p \in \pi^{-1}(x)} \phi(p).
	\end{equation*}
	From \cite[Lemma 4.31]{JFA20}, 
	\begin{equation} \label{lim.coin}
		\int \phi\, d\eta_{\alpha}= \lim_{k\to \infty} \int (\phi \circ F^{k}_\alpha)^+\, d\mu_\alpha= \lim_{k\to \infty} \int (\phi \circ F^{k}_\alpha)^-\, d\mu_\alpha.
	\end{equation}
	
	For ease of notation, we shall write the subscript $\alpha_n$ simply as $n$ and the map and measure at $\alpha_0$ as $F$ and $\mu$ respectively.
	The density of $f_\alpha$, $h_\alpha \in L^p(m)$.
	
	%
	%We have that, since $\displaystyle m(J_k) \approx \frac{1}{k^{\frac{\alpha+1}{\alpha}}}$, $\alpha<1$,
	%\begin{align*}
	%	\int h_\alpha^p dm \approx \sum_{k\geq n} k^p m(J_k)^p \approx \sum_{k\geq n} \frac{1}{k^{p/\alpha}}
	%\end{align*}
	%$\int h_\alpha^p dm<\infty \Leftrightarrow p>\alpha$.

	\begin{lemma}\label{lem.lem} For any $k \ge 1$, we have
		$$\displaystyle \lim_{n \to \infty} \int (\phi \circ F^k_n)^+\, d\mu_n= \int (\phi \circ F^k)^+\, d\mu$$	
	\end{lemma}
	\begin{proof}
		\begin{align*}
			\bigg|\int (\phi \circ F^k_n)^+\, d\mu_n- & \int (\phi \circ F^k)^+\, d\mu \bigg| \le\\
			& \underbrace{\bigg| \int\left( (\phi \circ F^k_n)^+\, - (\phi \circ F^k)^+\right)\, d\mu_n \bigg|}_{(A)} +\underbrace{\left| \int (\phi \circ F^k)^+\, d\mu_n-\int (\phi \circ F^k)^+\, d\mu\right|}_{(B)}.
		\end{align*}
		
		By the H\"older inequality, we have that
		\begin{align*}
			(A) &\le \int \left|h_n\left[(\phi \circ F^k_n)^+- (\phi \circ F^k)^+\right]\right| \,dm\\
			&\le \|h_n\|_p \left(\int \left|(\phi \circ F^k_n)^+ - (\phi \circ F^k)^+\right|^q\, dm\right)^{1/q}, \quad \text{ with } \frac{1}{p}+\frac{1}{q}=1.
		\end{align*}
		
		For a fixed $k$, since the solenoid maps are continuous, we have that for $n$ big enough, $(A) \to 0$.
		The linear response result in Theorem \ref{LR} implies statistical stability on $S^1$. Hence, $(B) \to 0$ as $n \to \infty$.
	\end{proof}
	\begin{maintheorem}
		$\displaystyle \lim_{n\to \infty} \int \phi \,d\eta_{n}= \int \phi \,d\eta.$
	\end{maintheorem}
	\begin{proof}
		Given an arbitrary $\varepsilon>0$, take $\delta>0$ such that  $d(p,q) < \delta \,  \Rightarrow \, |\phi(p)-\phi(q)| < \varepsilon$, for $p,q \in M$. We have that each disk $\pi^{-1}(x)$, is contracted by $F_n$ by a factor of $1/5$. There is a $k_0 \ge 1$ such that for all $k \ge k_0$ and $x \in S^1$
		$$
		\diam(F_n^k(\pi^{-1}(x)))< \delta.
		$$
		
		For all $k \ge k_0$ and $m\ge 0$
		\begin{align*}
			(\phi \circ F^{k+m}_n)^+(x)-(\phi \circ F^{k}_n)^+(f^m(x))=\sup(\phi \circ F^{k+m}_n|_{\pi^{-1}(x)})-\sup(\phi \circ F^{k}_n|_{\pi^{-1}(f^m(x))}).
		\end{align*}
		
		Since $F^{k+m}_n({\pi^{-1}(x)}) \subset F^{k}_n({\pi^{-1}(f^m(x))})$, this then implies that the above equation is bounded by
		\begin{align*}
			\sup(\phi \circ F^{k+m}_n|_{\pi^{-1}(x)})-\sup(\phi \circ F^{k}_n&|_{\pi^{-1}(f^m(x))}) \\
			&\le \sup(\phi \circ F^{k+m}_n|_{\pi^{-1}(x)})-\inf(\phi \circ F^{k+m}_n|_{\pi^{-1}(x)})
			<\varepsilon.
		\end{align*}
		
		By the invariance of $\mu_n$, we have that $\displaystyle\left(\int (\phi \circ F^{k+m}_n)^+(x)\, d\mu_n\right)_{k,n}$ is uniformly Cauchy. Therefore, 
		\begin{align*}
			\lim_{n\to \infty}\lim_{k\to \infty} \int (\phi \circ F^{k}_n)^+(x) \, d\mu_n=\lim_{k\to \infty}\lim_{n\to \infty} \int (\phi \circ F^{k}_n)^+(x)\, d\mu_n,
		\end{align*}
		and from equation \eqref{lim.coin} and Lemma \ref{lem.lem}
		\begin{align*}
			\lim_{n\to \infty}\int \phi \,d\eta_{n}=\lim_{k\to \infty}\int (\phi \circ F^k)^+\, d\mu,
		\end{align*}
		using equation \eqref{lim.coin} again, completes the proof.
	\end{proof}

\bibliographystyle{acm}
%\bibliographystyle{apalike}
%\bibliography{ref.bib}

\end{document}